\documentclass{article}

\usepackage{
amsmath,
amsthm,
amscd,
amssymb,
}
\usepackage{wasysym}
\usepackage{fancyvrb}

\usepackage{hyperref}
\usepackage{algorithm}
\usepackage[noend]{algpseudocode}
\usepackage{comment}
\usepackage{tikz}
\usetikzlibrary{positioning,arrows,calc}
\usetikzlibrary{shapes,patterns}
\usetikzlibrary{decorations.pathreplacing}

\theoremstyle{plain}
\newtheorem{lemma}{Lemma}
\newtheorem{definition}{Definition}
\newtheorem{corollary}{Corollary}
\newtheorem{proposition}{Proposition}
\newtheorem{theorem}{Theorem}
\newtheorem{conjecture}{Conjecture}
\newtheorem{question}{Question}

\theoremstyle{remark}
\newtheorem{remark}{Remark}
\newtheorem{example}{Example}

\newcommand{\R}{\mathbb{R}}

\newcommand{\Z}{\mathbb{Z}}

\newcommand{\N}{\mathbb{N}}

\newcommand{\Aut}{\mathrm{Aut}}

\newcommand{\stN}{{\uparrow}}
\newcommand{\stE}{{\rightarrow}}
\newcommand{\stW}{{\leftarrow}}
\newcommand{\stS}{{\downarrow}}

\newcommand{\stateat}[2]{#1[#2]}

\newcommand{\pats}[3]{\mathcal{P}_{#2,#3}(#1)}
\newcommand{\npats}[3]{\mathcal{P}^*_{#2,#3}(#1)}

\newcommand{\pto}{\nrightarrow}

\newcommand{\Sym}{\mathrm{Sym}}
\newcommand{\GL}{\mathrm{GL}}
\newcommand{\Homeo}{\mathrm{Homeo}}
\newcommand{\RTM}{\mathrm{RTM}}

\newcommand{\qee}{\fullmoon}

\title{A Physically Universal Turing Machine}

\author{
  Ville Salo\thanks{First author supported by Academy of Finland grant 2608073211.}
  ~and~Ilkka T\"orm\"a\thanks{Second author supported by Academy of Finland grant 295095.}
  \\
  Department of Mathematics and Statistics \\
  University of Turku, Turku, Finland \\
  \texttt{vosalo@utu.fi, iatorm@utu.fi}
}

\date{\today}

\begin{document}
\maketitle

\begin{abstract}
We construct a two-dimensional Turing machine that is physically universal in both the moving tape and moving head model.
In particular, it is mixing of all finite orders in both models.
We also provide a variant that is physically universal in the moving tape model, but not in the moving head model.
\end{abstract}

\section{Introduction}

%Consider a dynamical system that evolves deterministically in time, and has a notion of space in the sense that specifying a state of the system is equivalent to describing the contents of each point in space.
In the theory of dynamical systems, one typically models states of the system as points in a space, and the time evolution as a self-map on this space.
However, in many dynamical systems the state is more naturally seen as specifying the ``contents'' of many ``regions'' of an ``underlying space''. If a metric is needed on the space of all states, one postulates that two states resemble each other if the contents are similar in large regions of the underlying space.\footnote{If the state space happens to be a subspace of a function space $S^T$, it is natural to formalize this with the compact-open topology. However, the reader should keep a more informal mindset.}
A basic example are cellular automata, where the underlying space is discrete, and consists of a (possibly infinite) number of cells, each of which holds a symbol from a finite alphabet. One can also think of Newtonian dynamics as such a system: the underlying space is the three-dimensional Euclidean spaces, and the dynamics describes the interaction of objects occupying the space.

In the context of dynamical systems supporting a notion of underlying space, one can imagine reformulations of classical dynamical properties in terms of the space and its contents. Here, we restrict mainly to notions from topological dynamics.
For example, topological transitivity intuitively means that given any bounded region $D$ of the space and two sets of contents $A, B$ for it, there is a state resembling $A$ at $D$ which eventually evolves into another state that resembles $B$ at $D$.
Topological transitivity implies some form of control over the evolution of the system: if one has the power to speficy the contents of the space everywhere but in the domain $D$, then one can force the system to transform from $A$ to $B$ via its own dynamical rules.
Weak and strong mixing, total transitivity and specification are examples of stronger notions of controllability considered in topological dynamics, where it is required that the dynamics is able to perform arbitrary transformations (possibly many successive ones) in a finite region, with constraints on the time the transformations take.
%Note that control theory studies a different kind of controllability where one can inject information into the system at any point in time, instead of only at the beginning of the trajectory.

The notion of physical universality was defined by Janzing in~\cite{Ja10} as a notion of perfect controllability, for (classical and quantum) cellular automata and for systems defined by a Hamiltonian operator.
In terms of dynamical systems with an underlying space in the sense of the previous paragraph, the intuition is that a system is physically universal if it is topologically transitive, and in the spatial definition of topological transitivity explained above, the contents of the space outside $D$ and the time required for the transformation do not depend on $A$ and $B$.
Rather, one can implement any (suitably regular) function on the set of all contents of $D$ by specifying the contents of its complement and an amount of time for which to allow the system to evolve.
One should imagine a ``machine'' that can analyze the contents of $D$ with arbitrary precision and then perform a predefined manipulation on the contents in finite time.

In the context of cellular automata, physical universality is given a formal definition in \cite{Ja10} (which we repeat in Section~\ref{sec:DefOfPUTM}). Here, the considered regions are finite and the ``machine'' is the contents of the cells outside the finite region. The existence of a physically universal cellular automaton was left open in~\cite{Ja10}, and a two-dimensional physically universal cellular automaton was later constructed in~\cite{Sc15a}.
Since then, constructions of a quantum version~\cite{Sc15b} and a one-dimensional version~\cite{SaTo17b} have also been published.
All of these have the additional property that the minimum time required to implement a particular function depends polynomially on its descriptional complexity (when implemented as a Boolean circuit).
Such automata were called efficiently physically universal in~\cite{Sc15a}.
%The notion of PU for CA is one of many strongly related to von Neumann's universal constructor as presented in~\cite{Ne66}, which can construct any finite pattern from the corresponding instructions).
%PU is not implied by computational universality, which many cellular automata exhibit in one form or another, nor is it known to imply any commonly used version of this property.

It would be interesting to be able to study physical universality in other dynamical systems. Following the informal discussion above, one can imagine what physical universality should roughly mean for a system supporting a notion of underlying space. For example, in the context of Newtonian mechanics, one would say that Newtonian dynamics itself is physically universal if we can perform the following kinds of experiments: specify a cubical region of space and place such a collection of objects with specific momenta outside it, that when the cube is filled with stationary spherical objects of some fixed size, after exactly three minutes it will contain the same number of objects in approximately the same positions if their number was odd, and no objects at all if their number was even.

Formalizing this idea, however, is not easy, even in the context of topological dynamics. To us, of particular interest are the following three questions:
\begin{enumerate}
\item How to formalize the notion of underlying space of a topological dynamical system?
\item How to formalize physical universality for a class of topological dynamical systems supporting a notion of an underlying space?
\item Once this has been defined for a particular class of topological dynamical systems, can we find a PU system in the class?
\end{enumerate}
We do not solve the general problem, but we present a solution to the latter two questions for a new class of systems for which the notion of space is intuitively clear: Turing machines, which have been studied as topological dynamical systems by several authors~\cite{Mo90,Ku97,BlCaNi02,DeBl04,Je13}. There are two standard ways of seeing a Turing machine as a topological dynamical system, which were introduced in~\cite{Ku97}, namely the moving head and moving tape models. In this paper, we define a notion of physical universality for Turing machines in both models, and present a two-dimensional Turing machine that is efficiently physically universal in both models. Of course, whether our formalization of PU is ``correct'' is up to debate (and we point out some deficiencies of the definition ourselves), but we believe it is in the correct spirit.

Our PU Turing machine may also be of independent interest. It resembles the famous Langton's ant~\cite{La86} in two ways.
First, its internal state stores a cardinal direction in which it moves on the two-dimensional tape, and its local rule is (almost) invariant under rotations.
We see it as a generalized turmite.
Turmites have very simple local dynamics, yet all nontrivial turmites are capable of universal computation in a certain sense~\cite{MaGaMeMo18}, so they are natural candidates for physical universality.
At present, we do not know whether there are nontrivial turmites in the sense of \cite{MaGaMeMo18} that are not physically universal (up to parity).
Second, our machine has no periodic orbits in the moving head model, meaning that the Turing machine head eventually escapes every finite region.
This was proved for Langton's ant in~\cite{BuTr92}, and it is a necessary condition for physical universality in both models.
Langton's ant was also proved to be topologically transitive (and mixing up to a parity condition) in the moving tape model in~\cite{BuTr93}, which is implied for our machine by physical universality.
The dynamics of Langton's ant has been studied further in e.g.\ \cite{Ga03}.

Some other properties of interest that we prove are that our machine in fact escapes all finite rectangular regions in polynomial time in the size of the region, and is topologically mixing in the moving head model. To our knowledge it is the first proven example of the latter behavior in Turing machines. In addition to our main construction, we present a Turing machine which is physically universal in the moving tape model but not the moving head model. We also discuss the invariance properties of PU on Turing machines and possible strenghtenings of our definition, and include many questions.

We constructed the proof on the blackboard, except that the trick shown in Figure~\ref{fig:SculptStage} as found by trial and error in a computer simulation. With the exception of Example~\ref{ex:N3} (which is not crucial for the results), computer simulation is not necessary for following the proof. Nevertheless, for readers interested in exploring our rule $\bar M$, we have included a @RULE file used by the Golly cellular automaton simulator.

%\begin{acknowledgement}
%We thank Anaj\'i Gajardo for pointing out the reference~\cite{BuTr93}.
%\end{acknowledgement}

\section{Preliminaries}
\label{sec:Prelim}

Partial functions from a set $A$ to another set $B$ are denoted $f : A \pto B$.
The set of finite words over an alphabet $A$ is denoted $A^*$.
The empty word is denoted by $\lambda$.

Fix a dimension $d$.
The \emph{full shift} over a finite alphabet $\Sigma$ is the set $\Sigma^{\Z^d}$, whose elements are called \emph{configurations}.
For any subset $N \subset \Z^d$, a \emph{shape-$N$ pattern} is an element $P \in \Sigma^N$.
The pattern \emph{occurs} in $x \in \Sigma^{\Z^d}$ if for some $\vec v \in \Z^d$ we have $x_{\vec v + \vec w} = P_{\vec w}$ for all $\vec w \in N$.
A \emph{subshift} is a subset of $\Sigma^{\Z^d}$ defined by a set $F$ of finite patterns as the set of those configurations where no $P \in F$ occurs.

Denote by $\tau : \Z^d \times \Sigma^{\Z^d} \to \Sigma^{\Z^d}$ the translation action defined by $\tau_{\vec v}(x)_{\vec w} = x_{\vec w + \vec v}$.
For a finite set $Q$ and $\# \notin Q$, let $X_Q \subset (Q \cup \{\#\})^{\Z^d}$ be the subshift of those configurations that contain at most one occurrence of an element of $Q$.
For $q \in Q$ and $\vec v \in \Z^d$, denote by $\stateat{q}{\vec v}$ the unique configuration $x \in X_Q$ with $x_{\vec v} = q$.
For $N \subset \Z^d$, we denote by $\pats{N}{Q}{\Sigma}$ the shape-$N$ patterns occurring in $X_Q \times \Sigma^{\Z^d}$, and by $\npats{N}{Q}{\Sigma}$ those where an element of $Q$ occurs.
For $n \in \N$, we denote $\pats{n}{Q}{\Sigma} = \pats{[0, n-1]^d}{Q}{\Sigma}$, and similarly for $\npats{n}{Q}{\Sigma}$.

A $d$-dimensional Turing machine is a $4$-tuple $M = (Q, \Sigma, N, \delta)$, where $Q$ is a finite state set, $\Sigma$ is a finite alphabet, $N \subset \Z^d$ is a finite neighborhood and $\delta : (Q \times \Sigma^N) \to (Q \times \Sigma^N \times N)$ is a transition function. A number $r \in \N$ is a \emph{radius} of $M$ if $N \subset [-r,r]^d$, and \emph{the radius} refers to its minimal radius.
The machine is associated with two topological dynamical systems.
In the moving head model, the state space is $X_Q \times \Sigma^{\Z^d}$, and the dynamics is given by $M(\#^{\Z^d}, x) = (\#^{\Z^d}, x)$ and $M(\stateat{q}{\vec v}, x) = (\stateat{p}{\vec v + \vec w}, x)$, where $\delta(q, \tau_{\vec v}(x)|_N) = (p, P, \vec w)$ and $y$ is obtained from $x$ by replacing the contents of $N + \vec v$ with $P$.
In the moving tape model, the state space is $Q \times \Sigma^{\Z^d}$, and the dynamics is given by  $M(q, x) = (p, \tau_{\vec w}(z))$ with $\vec w$ and $p$ as above and $z$ being obtained from $x$ by replacing the contents of $N$ by $P$.

An important subset of $\Z^2$ is the following: For $m, n > 0$, the \emph{outer $(m,n)$-border} is the set $B_{m,n} = \{-1, m\} \times [0, n-1] \cup [0, m-1] \times \{-1, n\} \subset \Z^2$.
Seeing $\Z^2$ as an undirected graph with edges $\{\{(m,n), (m+a,n+b)\} \;|\; |a| + |b| = 1\}$, the outer $(m,n)$-border consists of the coordinates
that do not belong to the set $[0, m-1] \times [0, n-1]$ but have a neighbor that does. We denote $B_n = B_{n,n}$.

\section{Physical Universality in Turing Machines}
\label{sec:DefOfPUTM}

The \emph{physical universality} (or \emph{PU}) of a $d$-dimensional cellular automaton $f$ on state set $\Sigma$ is defined as follows:
For all finite domains $D \subset \Z^d$ and all functions $g : \Sigma^D \to \Sigma^D$, there exists a partial configuration $y \in \Sigma^{\Z^d \setminus D}$ and a time $t \in \N$ such that for all patterns $P \in \Sigma^D$ we have $f^t(P \sqcup y)|_D = g(P)$.
We can assume that the domain $D$ is square-shaped.
Since a $d$-dimensional Turing machine $M = (Q, \Sigma, N, \delta)$ in the moving head model can be seen as a restriction of a certain cellular automaton $f$ on $((Q \cup \{\#\}) \times \Sigma)^{\Z^d}$ to the set $X_Q \times \Sigma^{\Z^d}$, it would make sense to define physical universality of $M$ by the analogous condition.

However, there are two complications that make the definition vacuous.
First, some patterns in $\pats{D}{Q}{\Sigma}$ contain the head of $M$, so any partial configuration $y \in ((Q \cup \{\#\}) \times \Sigma)^{\Z^d \setminus D}$ that satisfies $y \sqcup P \in X_Q \times \Sigma^{\Z^d}$ for all $P \in \pats{D}{Q}{\Sigma}$ cannot contain the head of $M$.
But if $P \in \pats{D}{Q}{\Sigma}$ does not contain the head either, $y \sqcup P$ is then a fixed point of $M$.
Hence we restrict the universal quantification to the set $\npats{D}{Q}{\Sigma}$ of patterns that contain the head, and require $y \in \Sigma^{\Z^d \setminus D}$ to not contain the head of $M$.

The second complication is that apart from trivial cases, there always exist distinct patterns $P, P' \in \npats{D}{Q}{\Sigma}$ that satisfy $P \sqcup y = M(P' \sqcup y)$ for all $y \in \Sigma^{\Z^d \setminus D}$.
This implies $M^t(P \sqcup y) = M^{t+1}(P' \sqcup y)$ for all $t \in \N$, which makes it impossible to implement arbitrary functions $g : \npats{D}{Q}{\Sigma} \to \npats{D}{Q}{\Sigma}$.
To get around this problem, we restrict our attention to subsets of $\npats{D}{Q}{\Sigma}$ that are in disjoint $M$-orbits when completed into full configurations by filling $\Z^d \setminus D$ by a fixed symbol $0 \in \Sigma$.

We say a tuple of patterns $(P_0,...,P_{k-1}) \in (\npats{m}{Q}{\Sigma})^k$ has \emph{disjoint $0$-orbits} if the $M$-orbits of $P_i \sqcup 0^{\Z^d \setminus [0, m-1]^d}$ (in the moving head model) are pairwise disjoint for distinct $i$. Similarly, we say a tuple $((q_0, P_0), \ldots, (q_{k-1}, P_{k-1})) \in Q \times \Sigma^{[0,m-1]^d}$ has \emph{disjoint $0$-orbits} if the configurations $(q_j, P_j \sqcup 0^{\Z^d \setminus [0, m-1]^d})$ have pairwise disjoint $M$-orbits (in the moving tape model).

%{\color{red} Tuples of patterns on the $P$-side should be distinct, maybe some notation for that so we don't have to say it everywhere. }

\begin{definition}
Let $M = (Q, \Sigma, N, \delta)$ be a $d$-dimensional Turing machine with $0 \in \Sigma$. Let $P = (P_0, \ldots, P_{k-1}) \in (\npats{m}{Q}{\Sigma})^k$ and $(R_0, \ldots, R_{k-1}) \in (\pats{m}{Q}{\Sigma})^k$ be tuples of patterns. We say \emph{$P$ is physically transformable to $R$ in the moving head model (in time $t \in \N$)}, and there exists a partial configuration $x \in \pats{\Z^d \setminus [0, m-1]^d}{Q}{\Sigma}$ such that $M^t(P_j \sqcup x)|_{[0, m-1]^d} = R_j$ for all $j \in [0, k-1]$. Similarly, if $P = ((q_0, P_0), \ldots, (q_{k-1}, P_{k-1})), R = ((p_0, R_0), \ldots, (p_{k-1}, R_{k-1})) \in (Q \times \Sigma^{[0,m-1]^d})^k$, we say $P$ is \emph{$P$ is physically transformable to $R$ in the moving tape model (in time $t \in \N$)} if there exists a partial configuration $x \in \Sigma^{\Z^d \setminus [0, m-1]^d}$ such that $M^t(q_j, P_j \sqcup x) = (p_j, x_j)$ with $x_j|_{[0, m-1]^d} = R_j$ for all $j \in [0, k-1]$.

We say $M$ is \emph{physically universal in the moving head model} if for all $k, m \in \N$, for any $P, R \in (\npats{m}{Q}{\Sigma})^k$ such that the patterns $P_i$ have disjoint $0$-orbits, $P$ is physically transformable to $R$. The machine $M$ is \emph{efficiently physically universal in the moving head model}, if the time $t$ can be bounded by a polynomial in the circuit complexities of the set of patterns $\{P_0, \ldots, P_{k-1}\}$ and the function $P_i \mapsto R_i$. \emph{(Efficient) physical universality in the moving tape model} is defined exactly analogously.
\end{definition}

The concept of ``physically transformable'' allows many different definitions of physical universality, and it is hard to say which one is the best. Observe that if $k = 1$, $P$ is physically transformable to $R$ for all patterns (more precisely $1$-tuples of patterns) $P, R$ if and only if $M$ is topologically transitive (in either model). The precise definitions of PU above are optimized for machine $M$, see Section~\ref{sec:AdditionalShit} for some variants.

%The conditions on the patterns $P_j$ and stated $q_j$ above are called \emph{having disjoint $0$-orbits}.

%\begin{remark}
%\label{rem:Deficiency}
%In this article, our main goal is to provide a physically universal Turing machine, and the above definition formalizes the type of physical universality it exhibits, one that seems very natural especially in the context of symbol-conserving (Definition~\ref{def:SymbolConserving}) turmites such as ours. We do not claim that the above definition is perfect, and point out an obvious deficiency: we do not know whether it is possible for a Turing machine to have the property that any two configurations containing a head and having finitely many nonzero symbols are in the same orbit in the moving tape model.\footnote{It is known that every Turing machine admits two configurations with disjoint orbits, which differ in only one cell and both contain a head. However, it is not known whether we can always find such finite-support configurations.} Such a Turing machine would be vacuously physically universal in the moving tape model, with the present definition. By Lemma~\ref{lem:FastEscape}, our Turing machine has the property that most patterns with all nonzero symbols inside $[0,m-1]^d$ have disjoint $0$-orbits.
%\end{remark}

Our main result is that efficiently physically universal Turing machines exist.

\begin{theorem}
\label{thm:Main}
There exists a two-dimensional Turing machine which is efficiently physically universal in both the moving tape and the moving head model.
\end{theorem}

This machine $M$ acts on the binary tape. It has $5$ internal states and radius $1$.
The machine does not have particularly good symmetry properties, as we have to break its symmetry to avoid a parity issue. However, $M$ is a trivial modification\footnote{In dynamical terms, the machines are flow equivalent.} of a very natural Turing machine which we call $\bar M$. The machine $\bar M$ has $4$ states which carry only its orientation. It is time-symmetric\footnote{This means it is conjugate to its inverse by an involution. In the moving head model, the involution lives in the automorphism group of $X_Q \times \Sigma^{\Z^d}$. See Lemma~\ref{lem:Inverse} for the precise statement.}, symbol-conserving (Definition~\ref{def:SymbolConserving}) and has rotational symmetry (by 90-degree rotations). The machine also has a natural escape property crucial in our proofs, see Lemma~\ref{lem:FastEscape}.

In Section~\ref{sec:Further} we sketch some additional constructions obtained by modifying the Turing machine. The following is proved by adding an invariant factor in the the moving head model.

\begin{theorem}
There exists a Turing machine that is physically universal in the moving tape model, but not the moving head model.
\end{theorem}

Our Turing machines are quite general, and readers may wonder if the ability to see several cells at once is essential. It is not: we show that any Turing machine in our sense can be simulated by a \emph{classical Turing machine}, namely a Turing machine that performs a permutation of the tape, and then moves in a cardinal direction or stays put as a function of the current state. In particular, we obtain the following.

\begin{theorem}
There exists a classical Turing machine that is efficiently physically universal in both the moving tape and the moving head model.
\end{theorem}

We make some dynamical remarks in Section~\ref{sec:AdditionalShit}. We show that physical universality is invariant under very general spatial conjugacies.

\section{Simulation of Turing machines by circuits}

We begin by dealing with the simple computational issues that arise from the fact our set of initial conditions $A$ is a (typically proper) subset of $\npats{m}{Q}{\Sigma}$. The main observation is that if $A$ is efficiently computable, then we can recover the original pattern and the time elapsed so far from the state after the head exits the region $[0,m-1]^d$, assuming the region is exited quickly.

\begin{definition}
  Let $D \subset \Z^d$, and let $M = (Q, \Sigma, N, \delta)$ be a $d$-dimensional reversible Turing machine.
  For a configuration $x = (q[\vec v], y) \in X_Q \times \Sigma^{\Z^d}$ and $D \subset \Z^2$ with $\vec v \in D$, we denote by $\tau^M_D(x)$ the smallest number $t \geq 0$ with $M^t(x) = (p[\vec w], z)$ where $\vec w \notin D$.
  It is called the \emph{$D$-escape time} of $x$.
  We also denote $\rho^M_D(x) = M^{\tau^M_D(x)}(x)$.
  
  Fix a special alphabet symbol $0 \in \Sigma$, and let $r \geq 0$ be a radius of $M$.
  For a pattern $P \in \npats{D}{Q}{\Sigma}$ containing the head, we denote $\tau^M(P) = \tau^M_D(P \sqcup 0^{\Z^d \setminus D})$ and $\rho^M_r(P) = M^{\tau^M(P)}(P \sqcup 0^{\Z^d \setminus D})|_{D + [-r, r]^d}$.
\end{definition}

We can simulate a Turing machine with a cellular automaton, and a cellular automaton with a circuit, which gives rise to the following result.

\begin{lemma}
\label{lem:CircuitSimulation}
  Let $r \geq 0$, and let $M = (Q, \Sigma, N, \delta)$ be a $d$-dimensional reversible Turing machine with $0 \in \Sigma$ and radius $r$.
  Let $A \subset \npats{m}{Q}{\Sigma}$ be a set of $m^d$-patterns containing the head and having disjoint $0$-orbits and such that $\rho^M$ is defined on each pattern of $A$.
  Then $\rho^M$ is injective on $A$, and the circuit complexity of the function $\rho^M(P) \mapsto (P, \tau^M_r(P))$ is $O(T (C + (m+2r)^d))$, where $C$ is the circuit complexity of $A$ and $T = \max \{ \tau(P) \;|\; P \in A \}$.
\end{lemma}

\begin{proof}
  The injectivity of $\rho^M$ follows from the reversibility of $M$, and the fact that it cannot affect any cells in the region $\Z^d \setminus [-r, m-1+r]^d$ before leaving $[0, m-1]^d$.

  Let $f$ be a cellular automaton on $((Q \cup \{\#\}) \times \Sigma)^{\Z^d}$ whose restriction to $X_Q \times \Sigma^{\Z^d}$ implements the reverse machine $M^{-1}$ in the moving head model.
  Then $r$ is a radius for $f$.
  We construct a Boolean circuit arranged in $T$ `layers', numbered from $0$ to $T-1$.
  The $k$th layer contains $(m + 2r)^d$ copies of a circuit $S$ that computes the local rule $F : ((Q \cup \{\#\}) \times \Sigma)^{[-r,r]^d} \to (Q \cup \{\#\}) \times \Sigma$ of $f$, which can be visualized as cells in a $d$-dimensional grid.
  On layer $0$ the $(2r+1)^d$ input symbols of each copy of $S$ are taken from the input values, and on layer $k > 0$ from the outputs of the copies of $S$ of layer $k-1$.
  Those copies of $S$ that would take their inputs from outside the grid receive $0$-symbols instead.
  
  In addition to the next layer, the outputs of layer $k$ are passed to a copy of a size-$O(C + m^d)$ circuit that computes membership in $A' = \{ P \sqcup 0^{[-r, m-1+r]^d \setminus [0, m-1]^d} \;|\; P \in A \}$.
  With $O(T + m^d)$ additional gates, we can decide whether $k$ is the first layer whose pattern is in $A'$, and if so, copy the pattern into the output wires of the circuit.
  We also copy the information about $k$ to the output.
  
  For any $P \in A$, we have $M^{-\tau^M_r(P)}(\rho^M(P) \sqcup 0^{\Z^d \setminus [-r, m-1+r]^d}) = P \sqcup 0^{\Z^d \setminus [0, m-1]^d}$, and $\tau^M_r(P)$ is the smallest number for which the central pattern of the resulting configuration is in $A'$, since the patterns in $A$ have disjoint $0$-orbits and the head does not leave $[0, m-1]^d$ before $\tau^M_r(P)$ steps when started from $P$.
  If we feed $\rho^M(P)$ to the circuit constructed above, then it will compute the smallest number $t \in \N$ with $M^{-t}(\rho^M(P) \sqcup 0^{\Z^d \setminus [-r, m-1+r]^d})|_{[-r, m-1+r]^d} \in A'$, and this number is exactly $\tau^M_r(P)$.
  If the head leaves the region $[0, m-1]^d$ during the simulation, then it will do so after $\tau^M_r(P)$ steps, and any computation results made by the circuit after that are irrelevant (and probably incorrect).
  In particular, the circuit will correctly simulate $M^{-1}$ for at least $\tau^M_r(P)$ steps.
  Thus the circuit implements the desired function.
\end{proof}

\section{The Machine}
\label{sec:TheMachine}

In this section we present a physically universal two-dimensional Turing machine.
It has five states and uses a binary alphabet.
It is a \emph{generalized turmite}, meaning that the states correspond to the four cardinal directions plus some auxiliary data, and the machine operates by scanning and manipulating its surroundings, possibly changing its direction, and possibly taking one step forward.

The formal definition of the machine $M$ is as follows.
The tape alphabet is $\Sigma = \{0, 1\}$ and the state set is $Q = \{\stN, \stN^*, \stE, \stW, \stS\}$.
One step of $M$ is divided into three phases.
In the first phase we perform one of the following (non-overlapping) local transformations, rotated by any multiple of 90 degrees, where $a \in \{0, 1\}$ is arbitrary:
\begin{equation}
  \begin{array}{|c|c|}
  \hline
  a & 1 \\
  \hline
  0 \stN & a \\
  \hline
  \end{array}
  \ \leftrightarrow \
  \begin{array}{|c|c|}
  \hline
  a & 0 \\
  \hline
  1 \stW & a \\
  \hline
  \end{array}
  %\quad \quad
  %\begin{array}{|c|c|}
  %\hline
  %1 \stN & a \\
  %\hline
  %a & 0 \\
  %\hline
  %\end{array}
  %\mapsto
  %\begin{array}{|c|c|}
  %\hline
  %0 \stE & a \\
  %\hline
  %a & 1 \\
  %\hline
  %\end{array}
\label{eq:Transes}
\end{equation}
Note that the transition can be applied in both directions: the pattern on the right hand side is replaced by the pattern on the left hand side.
In all cases, a single $1$ is moved diagonally toward the inner side of a turn on the path that the head of $M$ traces.
If no transformation is applicable, the machine maintains its position and state.
Note that $\stN^*$ does not occur in these patterns.
In the second phase, the head moves one step in the direction indicated by its state, unless the state is $\stN^*$, in which case it retains its state and position.
In the third phase, the head changes its state from $\stN$ to $\stN^*$ or vice versa if applicable, and retains its state if not.
Since each phase is specified by a bijection (the first and third phases are actually involutions), $M$ is reversible. The machine also conserves the number of $1$s in the configuration, i.e. it is symbol-conserving in the following sense.

% I added this def since it's nice to say the machine has this property in the intro.

\begin{definition}
\label{def:SymbolConserving}
A Turing machine $M = (Q, \Sigma, N, \delta)$ is \emph{symbol-conserving} if for all $(q[\vec v], x) \in X_Q \times \Sigma^{\Z^2}$, if $M(q[\vec v], x) = (p[\vec w], y)$ then there exists a permutation $\pi \in \Sym(\Z^2)$ with finite support such that $y = \pi(x)$, where $\Sym(\Z^2)$ acts on configurations by $\pi(x)_{\vec v} = x_{\pi^{-1}(\vec v)}$.
\end{definition}

We also define a simpler auxiliary machine $\bar M$ with state set $\bar Q = \{\stN, \stE, \stW, \stS\}$ and tape alphabet $\Sigma$, and which behaves by first applying one of the transformations in~\eqref{eq:Transes} if possible, and then advancing for one step in the direction of the head.
It is also reversible.

\begin{lemma}
\label{lem:Correspondence}
For each $z \in X_{\bar Q} \times \Sigma^{\Z^2}$ we have $\bar M(z) = M^k(z)$ for some $k \in \{1,2\}$, and for each $y \in X_Q \times \Sigma^{\Z^2}$ we have $M^k(y) \in X_{\bar Q} \times \Sigma^{\Z^2}$ for some $k \in \{0,1\}$.
\end{lemma}

We now show a connection between $\bar M$ and its inverse.
For this, we define a few auxiliary functions.
Denote the \emph{mirror map} by $\mu(a,b) = (-a,b)$ for $(a,b) \in \Z^2$, and extend it to $\Sigma^{\Z^2}$ by $\mu(x)_{\vec w} = x_{\mu(\vec w)}$.
Then extend it to $X_{\bar Q} \times \Sigma^{\Z^2}$ by $\mu(q[\vec v], x) = (q'[\mu(\vec v)], \mu(x))$, where $q'$ points in the direction opposite to $q$ if $q$ is horizontal, and $q' = q$ otherwise.
Define the \emph{bit flip map} by $\beta(x)_{\vec v} = 1 - x_{\vec v}$ for $x \in \Sigma^{\Z^2}$ and the \emph{opposite map} by $\omega(q[\vec v]) = \hat q[\vec v - \tilde q]$, where $\hat q$ points in the direction opposite to $q$, and $\tilde q$ is the neighbor of $\vec 0$ in the direction of $q$.
Finally, define $\sigma(q[\vec v], x) = \mu(\omega(q[\vec v]), \beta(x))$.

\begin{lemma}
  \label{lem:Inverse}
  The inverse of $\bar M$ satisfies $\sigma \circ \bar M^{-1} \circ \sigma = \bar M$.
\end{lemma}

\begin{proof}
  This follows from a simple case analysis.
\end{proof}

The idea behind defining both $M$ and $\bar M$ is that while the latter has a simpler structure, it suffers from a parity issue.
Since the head of $\bar M$ always moves to an adjacent coordinate, $\bar M^n(q[(a,b)], x) = (p[(c,d)], y)$ implies that the number $n+a+b+c+d$ is even.
The machine $M$ uses a modulo-$2$ counter to simulate one step of $\bar M$ in two steps in those configurations where the head travels to the north.

We now prove that $M$ has no periodic points, similarly to Langton's ant, and can only escape a finite island of $1$s by stepping out of it and walking to infinity along a straight line.
Lemma~\ref{lem:Correspondence} allows us to use $\bar M$ in place of $M$ in the proof.

\begin{lemma}
  \label{lem:PushTwice}
  The machine $\bar M$ cannot makes two right turns in a row.
\end{lemma}

\begin{proof}
  After $\bar M$ has made a right turn, it has a $1$ on its right and a $0$ behind it, so it cannot make another right turn.
\end{proof}

\begin{lemma}
  \label{lem:FastEscape}
  Let $\vec u \in \Z^2$ and $(q[\vec v], x) \in X_Q \times \Sigma^{\Z^2}$ be such that $\vec v \in \vec u + [0, n-1]^2$.
  Then there exists $t = O(n^4)$ with $M^t(q[\vec v], x) = (p[\vec w], y)$ and $\vec w \notin \vec u + [0, n-1]^2$.
\end{lemma}

\begin{proof}
  We prove the analogous result for $\bar M$, from which the claim follows by Lemma~\ref{lem:Correspondence}.
  
  We assume without loss of generality that $\vec u = (1,1)$, and denote $R = [1,n]^2$ and $R' = [0, n+1]^2$.
  We define a potential function $\pi : X_{\bar Q} \times \Sigma^{\Z^2}$ by $\pi(q[\vec v], x) = \pi_1(q,\vec v) + \pi_2(x)$, where
  \begin{align*}
  	\pi_1(\stE, \vec v) & {} = -2 |\vec v|_1 + 2 \\
  	\pi_1(\stN, \vec v) & {} = -2 |\vec v|_1 \\
  	\pi_1(\stW, \vec v) & {} = 2 |\vec v|_1 + 2 \\
  	\pi_1(\stS, \vec v) & {} = 2 |\vec v|_1
  \end{align*}
  and $\pi_2(x) = \sum_{\vec v \in R'} x_{\vec v} \cdot |\vec v|_1^2$.
  Concretely, $\pi_2$ is the sum of squares of the Manhattan distances from the origin to each $1$ in the square $R'$, which contains $R$ and its border.
  The intuition is that when the head of $\bar M$ makes a loop, it pulls $1$s closer to the center of $R'$, which decreases the value of $\pi_2$.
  We use $\pi_1$ as a balancing term to ensure that $\pi$ decreases at least every two steps.
  More explicitly, we claim that $\pi(y) > \pi(M^2(y))$ holds for all $y \in X_{\bar Q} \times \Sigma^{\Z^2}$ where the head lies in $R$.
  The maximum value of $\pi$ is $O(n^4)$, so the result follows from this.
  
  %Consider a move of the head of $M$ inside $R$.
  %We say an \emph{upper turn} is a turn from east to south, or from north to west, and a \emph{lower turn} is a turn from south to east or from west to north.
  Suppose that the head of $\bar M$ is at a coordinate $\vec v \in R$.
  If $\bar M$ does not make a turn on this step, then the value of $\pi_1$ decreases by $2$ regardless of the state and position of the head, and $\pi_2$ retains its value, so that $\pi$ also decreases by $2$.
  
  Suppose that $\bar M$ makes a right turn.
  If the head was facing east, then the value of $\pi_1$ increases by $4|\vec v|_1 - 4$ and $\pi_2$ decreases by $4|\vec v|_1 - 4$, which means that $\pi$ retains its value.
  If the head was facing west, then $\pi_1$ decreases by $4|\vec v|_1 + 4$ and $\pi_2$ increases by $4|\vec v|_1 + 4$, so $\pi$ again retains its value.
  If the head was facing north or south, then both $\pi_1$ and $\pi_2$ retain their value.
  Thus the value of $\pi$ stays the same on a right turn.
  
  Suppose then that $\bar M$ makes a left turn.
  If the head was facing north, then $\pi_1$ increases by $4|\vec v|_1$ and $\pi_2$ decreases by $4|\vec v|_1 + 4$, so $\pi$ decreases by $4$.
  If the head was facing south, then $\pi_1$ decreases by $4|\vec v|_1$ and $\pi_2$ increases by $4|\vec v|_1 - 4$, so $\pi$ decreases by $4$.
  If the head was facing east or west, then $\pi_1$ decreases by $4$ and $\pi_2$ retains its value.
  Thus the value of $\pi$ always decreases by $4$ on a left turn.
  
  All in all, the value of $\pi$ never increases, and only stays the same on a right turn.
  By Lemma~\ref{lem:PushTwice}, $\bar M$ cannot make two consecutive right turns, so the claim follows.
\end{proof}

\begin{corollary}
  \label{cor:TMNoPeriodic}
  The machine $M$ has no nontrivial periodic points in the moving head model.
\end{corollary}

\begin{lemma}
  \label{lem:NoEscape}
  Let $x = (q[\vec{v}], y) \in X_Q \times \Sigma^{\Z^2}$ be a configuration with a finite number of $1$s, and let $P = [a,b] \times [c,d]$ be a finite rectangular region such that $\vec v \in P$ and $y_{\vec w} = 0$ for all $\vec w \notin P$.
  For all $k \geq 0$, the configuration $M^k(x)$ does not contain a $1$ outside $P' = [a-1,b+1] \times [c-1,d+1]$.
\end{lemma}

\begin{proof}
  We again prove the result for $\bar M$ and apply Lemma~\ref{lem:Correspondence}.

  Let $n \geq 0$ be the lowest integer such that $\bar M^{n+1}(x)$ contains a $1$ outside of $P$, if one exists.
  We assume without loss of generality that the head is pointing north in $\bar M^n(x)$ at some coordinate $(i,j) \in \Z^2$.
  We first observe that for all $m < n$, the head of $\bar M^m(x)$ is inside $P$.
  Namely, if the head was outside $P$, then it could not make a turn, since that would involve moving a $1$ out of $P$ or standing on a $1$ that is already outside $P$, neither of which happens before $n$ steps.
  Since $P$ is rectangular, the head can never return near it, and hence cannot move a $1$ outside it in the transition from $\bar M^n(x)$ to $\bar M^{n+1}(x)$.
  
  The transition from $\bar M^n(x)$ to $\bar M^{n+1}(x)$ is either a left or right turn, and in either case, a $1$ is moved from $P$ to its complement.
  Suppose first that the turn is to the left, so that $\bar M^n(x)_{(i,j)} = 0$ and $(i,j) \notin P$.
  Since the head is inside $P$ in $\bar M^{n-1}(x)$, we have $(i,j-1) \in P$.
  The rectangular shape of $P$ implies $(i+1,j+1) \notin P$, so this coordinate contains a $0$ in $\bar M^n(x)$.
  Then the head cannot make a left turn from $\bar M^n(x)$, a contradiction.
  
  Suppose then that $\bar M$ turns to the right, so that $\bar M^n(x)_{(i,j)} = 1$.
  Then we have $(i,j) \in P$ and $(i+1,j-1) \notin P$.
  %This implies $n > 0$, since all $1$s are initially in $P$.
  %Hence $(i, j-1) \in P'$, since it is the position of the head in $\bar M^{n-1}(x)$.
  We split into two cases depending on whether $P$ contains $(i, j-1)$.
  Suppose first that $(i, j-1) \in P$ and denote $R = [i+1, \infty) \times [j-1, j+1] \subset \Z^2$.
  Since $P$ is rectangular, we have $R \cap P = \emptyset$, so that $\bar M^{n+1}(x)_{\vec v} = 0$ for all $\vec v \in R \setminus \{(i+1,j-1)\}$.
  In $\bar M^{n+1}(x)$, the head is at $(i+1,j)$ facing east, after which it keeps traveling east inside the region $R$, never making a turn again.
  
  Suppose then that $(i, j-1) \notin P$.
  Since $P$ is rectangular, we have $(i', j-1) \notin P$ and hence $\bar M^n(x)_{(i', j-1)} = 0$ for all $i' \in \Z$.
   Hence the head travels east along the south border of $P$ until it either exits $P$ or makes another right turn, after which it can never make a turn again.
  In all cases, no $1$s are moved outside $P'$.
\end{proof}

The proof also shows that once the head leaves the region $[a-1,b+1] \times [c-1,d+1]$, it never makes a turn again, traveling in a straight line forever.
We will later use this fact to construct a gadget that catches the head and redirects it into a system of circuitry.
Using Lemma~\ref{lem:Correspondence} and Lemma~\ref{lem:Inverse}, we obtain the analogues of Lemma~\ref{lem:FastEscape} and Lemma~\ref{lem:NoEscape} for $M^{-1}$, but with the roles of $0$ and $1$ reversed.

\begin{remark}
It is a well-known conjecture that Langton's ant has a similar property as the one proved in Lemma~\ref{lem:FastEscape}: that started from a configuration of finite support, the head eventually leaves the pattern, and begins traveling in a rational direction in an eventually periodic pattern (in particular the spacetime diagram is semilinear). Since Langton's ant and $\bar M$ have the same state set and tape alphabet, one may wonder if they are the same machine in some dynamical sense. They are quite different: The distinct limit behavior on finite points shows that they are not conjugate by linear automorphisms of $\Z^2$ and automorphisms of $X_Q \times \Sigma^{\Z^2}$. In fact, since Langton's ant leaves ``garbage'' behind it on its diagonal journeys, the two machines cannot even admit spatiotemporal blockings that are conjugate by linear transformations of the tape and conjugacies between the configurations spaces.
\end{remark}

\section{Sequential Automata}

In this section, we define a class of token-based circuits that will later be simulated by $M$.
We do this in the framework of sequential automata; see \cite{Pr83} for a classical overview on this subject.
We extend this formalism slightly by introducing a time parameter that tracks the progress of the token inside the circuit.
% TODO: cite something? Probably Priese.

\begin{definition}
A \emph{(deterministic weighted) sequential automaton} or \emph{s-automaton} is a $5$-tuple $A = (Q, I, O, \delta, \mu)$, where $Q$ is a finite set of \emph{internal states}, $I$ is a finite set of \emph{input terminals}, $O$ is a finite set of \emph{output terminals} disjoint from $I$, $\delta : Q \times I \pto Q \times O$ is the \emph{partial transition function}, and $\mu : Q \times I \pto \R$ is the \emph{partial weight function} that assigns a weight to each transition, and is defined on the same set as $\delta$ is.
If $\mu$ is a constant function, we may replace it with said constant value.

We extend $\delta$ into a partial function from $Q \times I^*$ to $Q \times O^*$ by defining $\delta(q, \lambda) = (q, \lambda)$ and $\delta(q, w i) = (p, v o)$ if $\delta(q, w) = (p', v)$ and $\delta(p', i) = (p, o)$.
Likewise, we extend $\mu$ by $\mu(q, \lambda) = 0$ and $\mu(q, w i) = \mu(q, w) + \mu(p', i)$.
\end{definition}

An s-automaton is formally exactly equivalent to a deterministic weighted transducer, but we handle them as circuit components.
The intuition is that a single token is inserted into the component from an input terminal, and the component ejects it from an output terminal, possibly updating its internal state in the process.

\begin{definition}
\label{def:SAutSimulation}
An s-automaton $A = (Q, I, O, \delta, \mu)$ \emph{simulates} another s-automaton $B = (Q', I, O, \delta', \mu')$ from state $q' \in Q'$, if there exists a state $q \in Q$ such that for all words $w \in I^*$ such that $\delta'(q', w)$ is defined, we have $v = v'$ where $\delta(q, w) = (p, v)$ and $\delta'(q', w) = (p', v')$, and $\mu(q, w) = \mu'(q', w)$.
\end{definition}

Note that in this definition we do not require the transition functions of $A$ and $B$ to be defined on exactly the same set of input words: $A$ is allowed to process strictly more inputs than $B$.

\begin{definition}
If $A = (Q, I, O, \delta, \mu)$ and $B = (Q', I', O', \delta', \mu')$ are two s-automata, the \emph{product} $A \otimes B$ is defined as $(Q \times Q', I \dot\cup I', O \dot\cup O', \delta'', \mu'')$ where $\delta((q, q'), i) = ((p, q'), o)$ and $\mu((q, q'), i) = \mu(q, i)$ for $i \in I$ where $\delta(q, i) = (p, o)$, and $\delta((q, q'), i') = ((q, p'), o')$ and $\mu((q, q'), i') = \mu'(q', i')$ for $i' \in I'$ where $\delta(q', i') = (p', o')$.

For an s-automaton $A = (Q, I, O, \delta, \mu)$, $i \in I$ and $o \in O$, the \emph{$(i,o)$-feedback} $A^i_o$ is defined as $(Q, I \setminus \{i\}, O \setminus \{o\}, \delta', \mu')$, where $\delta'(q, i') = (p, o')$ if either $\delta(q, i') = (p, o')$ or we have a path $\delta(q, i') = (q_0, o)$, $\delta(q_k, i) = (q_{k+1}, o)$ for $k < n$, and $\delta(q_n, i) = (p, o')$, for $i' \in I \setminus \{i\}$ and $o' \in O \setminus \{o\}$.
We also define $\mu'(q, i') = \mu(q, i') + \sum_{k \leq n} \mu(q_k, i)$.
\end{definition}

These operations allow us to construct new s-automata by combining existing ones into circuits.

\begin{definition}
Let $\mathcal{A}$ be a set of s-automata and $n \geq 1$.
A \emph{normed network of size $n$} over $\mathcal{A}$ is an s-automaton constructed from $n$ copies of the automata in $\mathcal{A}$ using products and feedbacks.
\end{definition}

The term normed network was defined by 

We define two classes of s-automata that are relevant to our construction.
Both of them contain only constant-weight automata whose terminals can be used at most once.
For this reason, we call them \emph{disposable}.

\begin{definition}
\label{def:PrimComponents}
For $n \geq 1$, the \emph{disposable $n$-merge} is the s-automaton
\[
  \bar{m}^n = (\{a, b\}, \{i_0, i_1, \ldots, i_{n-1}\}, \{o\}, \delta, \max(1, n-1))
\]
where $\delta(a, i_k) = (b, o)$ for $k \in [0, n-1]$.
If $n = 1$, we call it the trivial merge, and denote $\bar{t} = \bar{m}^1$.
If $n = 2$, we call it the disposable merge, and denote $\bar{m} = \bar{m}^2$.

For $n, k \geq 1$, the \emph{disposable $(n,k)$-switch} is defined as the s-automaton $\bar{s}^{(n,k)} = (Q, I, O, \delta, \max(n, k(k-1)-1))$ where $Q = \{a^N_x \;|\; N \subset [0, k-1], x \in \{0,1\}\}$, $I = \{i_0, \ldots, i_{n-1}, i'_0, \ldots, i'_{k-1}\}$, $O = \{o_0, \ldots, o_{n-1}, o'_0(0), o'_0(1), \ldots, o'_{k-1}(0), o'_{k-1}(1)\}$ and $\delta(a^{[0,k-1]}_0, i_j) = (a^{[0,k-1]}_1, o_j)$ for $j \in [0, n-1]$ and $\delta(a^N_x, i'_j) = (a^{N \setminus \{j\}}_x, o'_j(x))$ for $N \subset [0,k-1]$, $j \in N$ and $x \in \{0,1\}$.
If $n = k = 1$, we call it the disposable switch, and denote $\bar{s} = \bar{s}^{(1,1)}$.
When $n = 1$ or $k = 1$, we may also suppress the respective indices from $I$ and $O$, and in the case of $k = 1$, denote $a^{\{0\}}_x = a_x$ and $a^\emptyset_x = b_x$.

The trivial merge, disposable merge and disposable switch are called \emph{primitive components}.
\end{definition}

The disposable $n$-merge simply combines $n$ inputs into one, and the purpose of the trivial merge is to control the weights of normed networks in our formalism.
The purpose of the disposable $(n,k)$-switch is to store one bit of information to be read later.
It has $n$ input terminals, any of which can be used to switch an internal bit from $0$ to $1$ before the other inputs are used.
Each of the $k$ additional input terminals is linked to one of two outputs depending on the bit, and they can be used in any order.

The disposable $n$-merge can be simulated by $n-1$ disposable $2$-merges and $O(n)$ trivial merges, and the disposable $(n,k)$-switch can be simulated by $n k$ disposable $(1,1)$-switches, $k$ disposable $k$-merges and $O(\max(n k, k^2))$ trivial merges.
The smallest versions of these s-automata are shown in Figure~\ref{fig:s-automata}, and the constructions of the larger versions are shown in Figure~\ref{fig:s-automata-cons}.

\begin{lemma}
\label{lem:NetworkWeight}
If $N$ is a normed network over $\bar{t}$, $\bar{m}$ and $\bar{s}$ of size $n$, then the weight of any transition of $N$ is at most $2 n$.
\end{lemma}

\begin{proof}
Each component of $N$ has weight $1$, and the token can traverse them at most twice.
\end{proof}

\begin{figure}[ht]
\begin{center}
\begin{tikzpicture}

\begin{scope}[xshift=-2.5cm]
\fill (0.25,0.5) circle (0.07);
\draw (0.25,0) -- (0.25,1);

\node [above] at (0.25,1) {$i$};
\node [below] at (0.25,0) {$o$};
\end{scope}

\begin{scope}
\draw [ultra thick] (0,0.5) -- (1,0.5);
\draw (0.25,0.5) -- (0.25,1);
\draw (0.5,0.5) -- (0.5,0);
\draw (0.75,0.5) -- (0.75,1);

\node [above] at (0.25,1) {$i_0$};
\node [above] at (0.75,1) {$i_1$};
\node [below] at (0.5,0) {$o$};
\end{scope}

\begin{scope}[xshift=3cm]
\draw (0,0.25) rectangle (1,0.75);
\draw (0.5,0.75) -- (0.5,1);
\draw (0.3,0.25) -- (0.3,0);
\draw (0.7,0.25) -- (0.7,0);
\draw (-0.25,0.5) -- (0,0.5);
\draw [dotted] (0,0.5) -- (1,0.5);
\draw (1,0.5) -- (1.25,0.5);

\node [above] at (0.5,1) {$i'$};
\node [below left] at (0.3,0) {$o'(0)$};
\node [below right] at (0.7,0) {$o'(1)$};
\node [left] at (-0.25,0.5) {$i$};
\node [right] at (1.25,0.5) {$o$};
\end{scope}

\end{tikzpicture}
\end{center}
\caption{The trivial merge $\bar{t}$, the disposable merge $\bar{m}$ and the disposable switch $\bar{s}$.}
\label{fig:s-automata}
\end{figure}
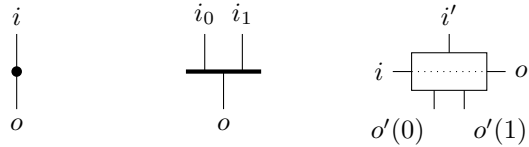

\begin{figure}[ht]
\begin{center}
\begin{tikzpicture}

\begin{scope}
\draw [ultra thick] (0,0.5) -- (3,0.5);
\draw (0.25,0.5) -- (0.25,1);
\draw (0.75,0.5) -- (0.75,1);
\draw (1.25,0.5) -- (1.25,1);
\draw (1.75,0.5) -- (1.75,1);
\draw (2.25,0.5) -- (2.25,1);
\draw (2.75,0.5) -- (2.75,1);
\draw (1.5,0.5) -- (1.5,0);

\node [above] at (0.25,1) {$i_0$};
\node [above] at (0.75,1) {$i_1$};
\node [above] at (1.25,1) {$i_2$};
\node [above] at (1.75,1) {$i_3$};
\node [above] at (2.25,1) {$i_4$};
\node [above] at (2.75,1) {$i_5$};
\node [below] at (1.5,0) {$o$};
\end{scope}

\begin{scope}[xshift=5cm,yshift=-1cm]
\draw (0,0.25) rectangle (3,2.75);
\foreach \a in {0,1,2}{
  \draw (\a+0.5,2.75) -- (\a+0.5,3);
  \draw (\a+0.3,0.25) -- (\a+0.3,0);
  \draw (\a+0.7,0.25) -- (\a+0.7,0);
}
\foreach \b in {0,1,2}{
  \draw (-0.25,\b+0.5) -- (0,\b+0.5);
  \draw [dotted] (0,\b+0.5) -- (3,\b+0.5);
  \draw (3,\b+0.5) -- (3.25,\b+0.5);
}

\node [above] at (0.5,3) {$i'_0$};
\node [above] at (1.5,3) {$i'_1$};
\node [above] at (2.5,3) {$i'_2$};
\node [below] at (0.5,0) {$o'_0(x)$};
\node [below] at (1.5,0) {$o'_1(x)$};
\node [below] at (2.5,0) {$o'_2(x)$};
\node [left] at (-0.25,0.5) {$i_0$};
\node [left] at (-0.25,1.5) {$i_1$};
\node [left] at (-0.25,2.5) {$i_2$};
\node [right] at (3.25,0.5) {$o_0$};
\node [right] at (3.25,1.5) {$o_1$};
\node [right] at (3.25,2.5) {$o_2$};
\end{scope}

\begin{scope}[yshift=-5cm]
\draw [ultra thick] (0.1,1.5) -- (0.9,1.5);
\draw [ultra thick] (1.1,1.5) -- (1.9,1.5);
\draw [ultra thick] (2.1,1.5) -- (2.9,1.5);

\draw [ultra thick] (0.6,1) -- (1.4,1);

\draw [ultra thick] (1.1,0.5) -- (1.9,0.5);

%\fill (0.25,1) circle (0.07);
%\fill (0.25,1.5) circle (0.07);
%\fill (0.75,1.5) circle (0.07);

\foreach \x in {0.25,0.75,...,2.75}{
  \draw (\x,2) -- (\x,1.5);
}
\draw (0.5,1.5) |- (0.75,1.25) -- (0.75,1);
\draw (1.5,1.5) |- (1.25,1.25) -- (1.25,1);
\draw (1,1) |- (1.25,0.75) -- (1.25,0.5);
\draw (2.5,1.5) |- (1.75,0.75) -- (1.75,0.5);
\draw (1.5,0.5) -- (1.5,0);

\fill (2.5,1) circle (0.07);

%\draw (0.25,0.5) -- (0.25,2);
%\draw (0.75,0.5) -- (0.75,0.75) -- (1,0.75) -- (1,1);
%\draw (0.75,1) -- (0.75,2);
%\draw (1.25,1) -- (1.25,1.25) -- (1.5,1.25) -- (1.5,1.5);
%\draw (1.25,1.5) -- (1.25,2);
%\draw (1.75,1.5) -- (1.75,2);
%\draw (0.5,0.5) -- (0.5,0);
\end{scope}

\begin{scope}[xshift=4.5cm,yshift=-6cm]
\foreach \x in {0,1.5,3}{
  \begin{scope}[xshift=\x cm]
  \foreach \y in {0,1.5,3}{
    \begin{scope}[yshift=\y cm]
      \draw (0,0.25) rectangle (1,0.75);
      \draw (0.5,0.75) -- (0.5,1);
      \draw (0.3,0.25) -- (0.3,0);
      \draw (0.7,0.25) -- (0.7,0);
      \draw (-0.5,0.5) -- (0,0.5);
      \draw [dotted] (0,0.5) -- (1,0.5);
      \draw (1,0.5) -- (1.4,0.5);
    \end{scope}
  }
  \draw [ultra thick] (0.6,0) -- (1.3,0);
  \draw (0.7,3) -| (1.2,0);
  \draw (0.7,1.5) -| (1.1,0);
  \draw (0.3,3) |- (0.5,2.5);
  \draw (0.3,1.5) |- (0.5,1);
  \draw (0.9,0) -- (0.9,-0.3);
  \draw (0.3,0) -- (0.3,-0.3);
  \fill (1.2,2.25) circle (0.07);
  \fill (1.1,0.75) circle (0.07);
  \fill (1.2,0.75) circle (0.07);
  \fill (0.3,0.1) circle (0.07);
  \fill (0.3,-0.1) circle (0.07);
  \end{scope}
}
\foreach \y in {0,1.5,3}{
  \fill (-0.15,\y+0.5) circle (0.07);
  \fill (-0.35,\y+0.5) circle (0.07);
}
\end{scope}

\end{tikzpicture}
\end{center}
\caption{Constructing the disposable $n$-merge $\bar{m}^n$ (on the left with $n = 6$) and the disposable $(n,k)$-switch $\bar{s}^{(n,k)}$ (on the right with $n = k = 3$) from primitive components. The trivial merges guarantee that every path through the normed network goes through an equal number of primitive components, all of which have weight $1$.}
\label{fig:s-automata-cons}
\end{figure}
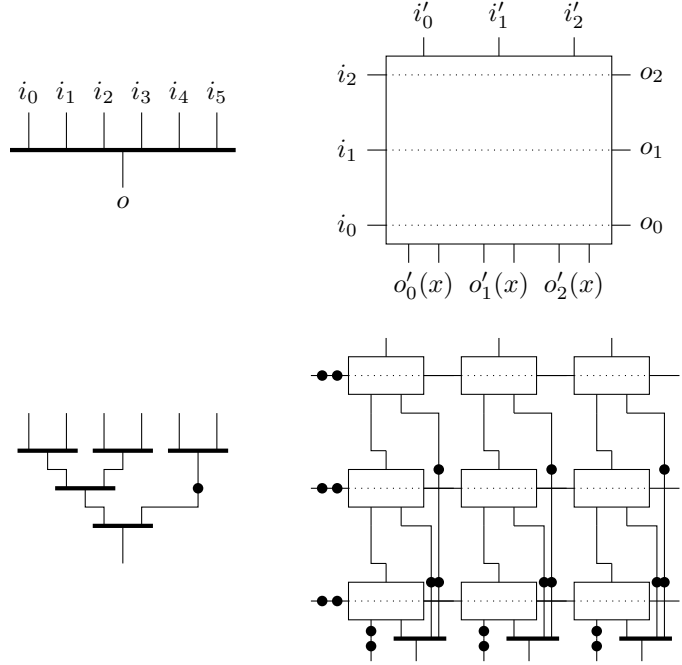

We show that any Boolean circuit can be simulated by a normed network over the merge and switch in a certain sense.
This result will be used as a black box later in the construction.

\begin{lemma}
\label{lem:CircuitAutomaton}
Let $k \in \N$, and let $C : \{0, 1\}^k \to \{0, 1\}^k$ be a Boolean circuit composed of $n \geq k$ gates (unlimited-fanin AND and unary NOT, each with unlimited fanout).
Then there exists a constant-weight normed network $A_C = (Q_C, I, O, \delta_C, T)$ of size $O(k n^2)$ and weight $T = O(k n^2)$ over the components $\bar{t}$, $\bar m$ and $\bar s$ with input terminal set $I = \{i_0(0), i_0(1), \ldots, i_{k-1}(0), i_{k-1}(1), \#, i^0, \ldots, i^{k-1}\}$ and output terminal set $O = \{o_0, \ldots, o_{k-1}, \$, o^0(0), o^0(1), \ldots, o^{k-1}(0), o^{k-1}(1)\}$, and a state $q_0 \in Q_C$ such that for all $w \in \{0, 1\}^k$ we have $\delta_C(q_0, i_0(w_0) i_1(w_1) \cdots i_{k-1}(w_{k-1}) \# i^0 \cdots i^{k-1}) = (p, o_0 o_1 \cdots o_{k-1} \$ o^0(v_0) o^1(v_1) \cdots o^{k-1}(v_{k-1}))$ for some $p \in Q$ and $v = C(w)$.
\end{lemma}

\begin{proof}
First, we replace each AND-gate of $C$ with fanin greater than $2$ by a series of binary AND-gates, and each gate with fanout greater than $1$ by a gate of the same type with fanout $1$ and a collection of \emph{splitters}, which are gates with fanin $1$ and fanout $2$ that copy their input bit into the two output wires.
This causes an at most quadratic blowup in the size of $C$.

In the construction, each wire $W$ of $C$ corresponds to a switch $\bar{s}_W$ in $A_C$, whose state is initially $a_0$.
Denote the input wires of $C$ by $W^0, \ldots, W^{k-1}$ and the output wires by $W_0, \ldots, W_{k-1}$.
For each $j \in [0, k-1]$, the input terminal $i_j(1)$ of $A_C$ is connected to the input terminal $i$ of $\bar{s}_{W^j}$, the output terminal $o$ of $\bar{s}_{W^j}$ is connected to a merge $\bar{m}_{W^j}$.
The terminal $i_j(0)$ is connected to the other input of $\bar{m}_{W^j}$ via a trivial merge, and the output terminal of $\bar{m}_{W^j}$ is connected to $o_j$.
In this way, if terminal $i_j(x)$ receives a token before $i_j(1-x)$, then the switch $\bar{s}_{W^j}$ is set to state $a_x$, and the token travels through the merge $\bar{m}_{W^j}$ and exits via $o_j$.
In particular, $\delta(q_0, i_0(w_0) \cdots i_{k-1}(w_{k-1})) = (p', o_0 \cdots o_{k-1})$ for a state $p'$ where each $\bar{s}_{W^j}$ has internal state $b_{w_j}$ and otherwise agrees with $q_0$.

Let $G$ be a gate of $C$.
Suppose first that $G$ is a NOT-gate that connects a wire $W$ to another wire $V$.
We connect terminal $o'(0)$ of $\bar{s}_W$ to terminal $i$ of $\bar{s}_V$, and connect terminal $o$ of $\bar{s}_V$ to an input terminal of a merge $\bar{m}_G$.
We also connect terminal $o'(1)$ of $\bar{s}_W$ to the other input terminal of $\bar{m}_G$.
Suppose then that $G$ is a splitter that connects a wire $W$ to two wires $V$ and $U$.
We connect terminal $o'(1)$ of $\bar{s}_W$ to terminal $i$ of $\bar{s}_V$, connect terminal $o$ of $\bar{s}_V$ to terminal $i$ of $\bar{s}_U$, and terminal $o$ of $\bar{s}_U$ into an input terminal of a merge $\bar{m}_G$.
Terminal $o'(0)$ of $\bar{s}_W$ is connected directly to the other input terminal of $\bar{m}_G$.
Finally, suppose that $G$ is an AND-gate that connects two wires $W$ and $V$ into a wire $U$.
We connect terminal $o'(0)$ of $\bar{s}_W$ and terminal $o'(0)$ of $\bar{s}_V$ to two input terminals of a $3$-merge $\bar{m}_G$, connect terminal $o'(1)$ of $\bar{s}_W$ to terminal $i'$ of $\bar{s}_V$, connect terminal $o'(1)$ of $\bar{s}_V$ to terminal $i$ of $\bar{s}_U$, and connect terminal $o$ of $\bar{s}_U$ into the third input terminal of $\bar{m}_G$.
In each case, the input terminal $i'$ of $\bar{s}_W$ is not connected to anything, and we introduced a merge $\bar{m}_G$ whose output terminal is likewise not yet connected to anything.
We denote these terminals by $i_G$ and $o_G$, respectively.
Finally, we add enough trivial merges to the gadgets that all paths from $i_G$ to $o_G$ pass through an equal number of primitive components.
See Figure~\ref{fig:s-automata-gates} for a visualization of this construction.

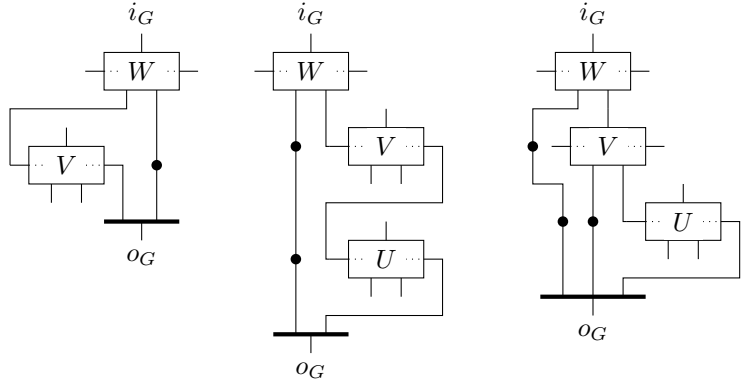
\begin{figure}[ht]
\begin{center}
\begin{tikzpicture}

\begin{scope}
\draw (0,0.25) rectangle (1,0.75);
\draw (0.5,0.75) -- (0.5,1);
\draw (0.3,0.25) -- (0.3,0) -- (-1.25,0) -- (-1.25,-0.75);
\draw (0.7,0.25) -- (0.7,-1.5);
\draw (-0.25,0.5) -- (0,0.5);
\draw [dotted] (0,0.5) -- (1,0.5);
\draw (1,0.5) -- (1.25,0.5);
\fill (0.7,-0.75) circle (0.07);

\node [fill=white] at (0.5,0.5) {$W$};
\node [above] at (0.5,1) {$i_G$};
\end{scope}

\begin{scope}[xshift=-1cm,yshift=-1.25cm]
\draw (0,0.25) rectangle (1,0.75);
\draw (0.5,0.75) -- (0.5,1);
\draw (0.3,0.25) -- (0.3,0);
\draw (0.7,0.25) -- (0.7,0);
\draw (-0.25,0.5) -- (0,0.5);
\draw [dotted] (0,0.5) -- (1,0.5);
\draw (1,0.5) -- (1.25,0.5) -- (1.25,-0.25);

\node [fill=white] at (0.5,0.5) {$V$};

\draw [ultra thick] (1,-0.25) -- (2,-0.25);
\draw (1.5,-0.25) -- (1.5,-0.5);

\node [below] at (1.5,-0.5) {$o_G$};
\end{scope}

\begin{scope}[xshift=2.25cm]
\draw (0,0.25) rectangle (1,0.75);
\draw (0.5,0.75) -- (0.5,1);
\draw (0.3,0.25) -- (0.3,-3);
\draw (0.7,0.25) -- (0.7,-0.5) -- (0.75,-0.5);
\draw (-0.25,0.5) -- (0,0.5);
\draw [dotted] (0,0.5) -- (1,0.5);
\draw (1,0.5) -- (1.25,0.5);
\fill (0.3,-0.5) circle (0.07);
\fill (0.3,-2) circle (0.07);

\node [fill=white] at (0.5,0.5) {$W$};
\node [above] at (0.5,1) {$i_G$};
\end{scope}

\begin{scope}[xshift=3.25cm,yshift=-1cm]
\draw (0,0.25) rectangle (1,0.75);
\draw (0.5,0.75) -- (0.5,1);
\draw (0.3,0.25) -- (0.3,0);
\draw (0.7,0.25) -- (0.7,0);
\draw (-0.25,0.5) -- (0,0.5);
\draw [dotted] (0,0.5) -- (1,0.5);
\draw (1,0.5) -- (1.25,0.5) -- (1.25,-0.25) -- (-0.3,-0.25) -- (-0.3,-1) -- (-0.25,-1);

\node [fill=white] at (0.5,0.5) {$V$};
\end{scope}

\begin{scope}[xshift=3.25cm,yshift=-2.5cm]
\draw (0,0.25) rectangle (1,0.75);
\draw (0.5,0.75) -- (0.5,1);
\draw (0.3,0.25) -- (0.3,0);
\draw (0.7,0.25) -- (0.7,0);
\draw (-0.25,0.5) -- (0,0.5);
\draw [dotted] (0,0.5) -- (1,0.5);
\draw (1,0.5) -- (1.25,0.5) -- (1.25,-0.25) -- (-0.3,-0.25) -- (-0.3,-0.5);

\node [fill=white] at (0.5,0.5) {$U$};

\draw [ultra thick] (-1,-0.5) -- (0,-0.5);
\draw (-0.5,-0.5) -- (-0.5,-0.75);

\node [below] at (-0.5,-0.75) {$o_G$};
\end{scope}

\begin{scope}[xshift=6cm]
\draw (0,0.25) rectangle (1,0.75);
\draw (0.5,0.75) -- (0.5,1);
\draw (0.3,0.25) -- (0.3,0) -- (-0.3,0) -- (-0.3,-1) -- (0.1,-1) -- (0.1,-2.5);
\draw (0.7,0.25) -- (0.7,0);
\draw (-0.25,0.5) -- (0,0.5);
\draw [dotted] (0,0.5) -- (1,0.5);
\draw (1,0.5) -- (1.25,0.5);
\fill (-0.3,-0.5) circle (0.07);
\fill (0.1,-1.5) circle (0.07);
\fill (0.5,-1.5) circle (0.07);

\node [fill=white] at (0.5,0.5) {$W$};
\node [above] at (0.5,1) {$i_G$};
\end{scope}

\begin{scope}[xshift=6.2cm,yshift=-1cm]
\draw (0,0.25) rectangle (1,0.75);
\draw (0.5,0.75) -- (0.5,1);
\draw (0.3,0.25) -- (0.3,-1.5);
\draw (0.7,0.25) -- (0.7,-0.5) -- (0.75,-0.5);
\draw (-0.25,0.5) -- (0,0.5);
\draw [dotted] (0,0.5) -- (1,0.5);
\draw (1,0.5) -- (1.25,0.5);

\node [fill=white] at (0.5,0.5) {$V$};
\end{scope}

\begin{scope}[xshift=7.2cm,yshift=-2cm]
\draw (0,0.25) rectangle (1,0.75);
\draw (0.5,0.75) -- (0.5,1);
\draw (0.3,0.25) -- (0.3,0);
\draw (0.7,0.25) -- (0.7,0);
\draw (-0.25,0.5) -- (0,0.5);
\draw [dotted] (0,0.5) -- (1,0.5);
\draw (1,0.5) -- (1.25,0.5) -- (1.25,-0.25) -- (-0.3,-0.25) -- (-0.3,-0.5);

\node [fill=white] at (0.5,0.5) {$U$};

\draw [ultra thick] (-1.4,-0.5) -- (0,-0.5);
\draw (-0.7,-0.5) -- (-0.7,-0.75);

\node [below] at (-0.7,-0.75) {$o_G$};
\end{scope}

\end{tikzpicture}
\end{center}
\caption{Implementations of the NOT-gate (left), splitter (middle) and AND-gate (right) as normed networks over $\bar{t}$, $\bar{m}$ and $\bar{s}$.}
\label{fig:s-automata-gates}
\end{figure}

Next, enumerate the gates $G_1, \ldots, G_n$ of $C$ in some order that is consistent with their evaluation (i.e perform a topological sort of the circuit as a directed acyclic graph).
Connect the input terminal $\#$ of $A_C$ to $i_{G_1}$.
For each $j = 1, \ldots, n-1$, connect $o_{G_j}$ to $i_{G_{i+1}}$, and finally, connect $o_{G_n}$ to the output terminal $\$$ of $A_C$.
When we now add a token to the terminal $\#$ in the state $p'$ where the values of the input wires have been determined, it travels to the terminal $i_{G_1}$, sets the values of the switches corresponding to the output wires of $G$, and exits through $o_{G_1}$.
Then it proceeds to $i_{G_2}$, repeating this for each gate before exiting through $o_{G_{k-1}}$ and then $\$$.
By the construction of the gates, each switch $\bar{s}_W$ in $A_C$ is correctly assigned to the state $a_x$, where $x$ is the value of the wire $W$ in $C$, until it evolves to state $b_x$ as the gate that it leads to is evaluated.
Furthermore, the number of primitive components that the token passes through does not depend on the states of the switches.
In the end, the switches $\bar{s}_{W_j}$ corresponding to the output wires $W_j$ hold the states $a_{C(w)_j}$.

Finally, we connect each terminal $i^j$ of $A_C$ to terminal $i'$ of $\bar{s}_{W^j}$, and for each $x \in \{0, 1\}$, connect terminal $o'(x)$ of $\bar{s}_{W_j}$ to terminal $o^j(x)$ of $A_C$.
We also add a chain of $O(n)$ new trivial merges right after each input terminal of the network other than $\#$ in order to balance its weight.
It is now easy to see that the claim holds.
\end{proof}

\section{Simulation of Normed Networks}

In this section, we show how normed networks can be simulated by the Turing machine $M$.
We use such simulations in the proof of physical universality of $M$, and this property is very time-sensitive, in the sense that we must be able to precisely control the number of steps that $M$ takes to construct the output pattern.
Recall that $B_{m,n}$ is the outer border of $[0, m-1] \times [0, n-1]$.

\begin{definition}
\label{def:AutomataByM}
Let $A = (S, I, O, \delta, \mu)$ be an s-automaton and $s_0 \in S$ its state, let $m, n > 0$, let $F : \N \to \N$ be a function and denote $D = [0, m-1] \times [0, n-1]$.
An \emph{$(m,n,F)$-simulation} of $A$ by $M$ from $s_0$ is a pair $(P, d)$ with $P \in \Sigma^D$ a pattern and $d : I \cup O \to B_{m,n}$ an injective function with the following properties.
\begin{itemize}
\item $d(e)$ is on the leftmost column of $B_{m,n}$ for $e \in I$, and on the rightmost column for $e \in O$.
\item For each $R \in \Sigma^D$ and $e \in I \cup O$, let $x(R, e) = (\stE[d(e)], R \sqcup 0^{\Z^2 \setminus D})$.
Then for each word $i = i_0 \cdots i_{k-1} \in I^*$ such that $\delta(s, i) = (t, o)$ is defined, we have a sequence of patterns $P = P_0, P_1, \ldots P_k$ in $\Sigma^D$ with $\rho^M_D(x(P_j, i_j)) = x(P_{j+1}, o_j)$ for each $j < k$ and $\sum_{j = 0}^{k-1} \tau^M_D(x(P_j, i_j)) = F(\mu(s_0, i))$.
\end{itemize}
If $F$ is multiplication by a constant $c$, as in $F(n) = c n$, we may replace it by $c$.
\end{definition}

The intuition behind Definition~\ref{def:AutomataByM} is that we have a single pattern corresponding to a state $s_0$ of $A$, and fixed positions around these patterns that correspond to input and output terminals.
When the head of $M$ enters the domain of the pattern through an input position, then it will eventually leave the domain through an exit position and manipulate the pattern into one that corresponds to another state, in a way that is consistent with the transition function $\delta$.
The correspondence between patterns and states need not be exact, as long as the systems are behaviorally equivalent.
Furthermore, the number of steps required to simulate a transition of $A$ is determined by its weight given by $\mu$, through the function $F$.
In our construction, $F$ will be a polynomial.

\begin{lemma}
\label{lem:AutomataByM}
There exist simulations by $M$ with equal weight function $n \mapsto c n$ of the trivial merge $\bar{t}$, the disposable merge $\bar m$, and the disposable switch $\bar s$ from states $a$, $a$ and $a_0$ respectively.
In these simulations, no two terminals have adjacent positions on the border of the domain.
\end{lemma}

\begin{proof}
See Figure~\ref{fig:MSimulatesAutomata} for a $(6,8,11)$-simulation of $\bar m$, and an $(8,13,22)$-simulation of $\bar s$.
The figures depict states that correspond to $a$ and $a_0$, respectively, and the other states are obtained by inserting the head at one of the positions corresponding to the input terminals and applying the escape map $\rho^M_D$ until no more inputs can be processed.
It is easy to verify that these processes exactly simulate the behavior of $\bar{m}$ and $\bar{s}$ on valid input sequences.

The trivial merge can be simulated by a single cell holding the symbol $0$.
Finally, we can pad the simulations with extra columns of $0$-cells to guarantee that the weight factor $c$ is the same in all of them.
\end{proof}

\begin{figure}[ht]
  \begin{center}
    \begin{tikzpicture}
      
      \draw [step=1/2] (0,0) grid (3,4);
      \draw [densely dotted,>->] (-0.25,0.75) -| ++(0.5,1) -- ++(3,0);
      \draw [densely dotted,>-] (-0.25,3.75) -| ++(2.5,-2); 
      \node at (0.75,0.25) {$1$};
      \node at (0.25,1.75) {$1$};
      \node at (1.75,1.25) {$1$};
      \node at (1.75,1.75) {$1$};
      \node at (1.75,2.25) {$1$};
      \node at (2.25,1.25) {$1$};
      \node at (2.25,2.25) {$1$};
      \node at (2.25,3.75) {$1$};
      % \node at (-0.25,0.75) {$\stE$};
      % \node at (-0.25,3.75) {$\stE$};
      % \node at (3.25,1.75) {$\stE$};
      \node [left] at (-0.25,0.75) {$i_0$};
      \node [left] at (-0.25,3.75) {$i_1$};
      \node [right] at (3.25,1.75) {$o$};
      
      \begin{scope}[xshift=5cm]
        \draw [step=1/2] (0,0) grid (4,6.5);
        \draw [densely dotted,>->] (-0.25,2.25) -| ++(1.5,0.5) -| ++(-1,2) -- ++(4,0);
        \draw [densely dotted,>->] (-0.25,5.75) -| ++(2.5,-2) -| ++(1,2.5) -- ++(1,0);
        \draw [densely dotted,->] (2.25,3.75) |- ++(1,-3) |- ++(1,1);
        \foreach \x/\y in {0/2.5,0/4.5,1.5/0,1.5/1.5,1.5/3,2/5.5,3/1.5,3/6,3.5/0,3.5/3}{
          \node at (\x+0.25,\y+0.25) {$1$};
        }
        % \node at (-0.25,2.25) {$\stE$};
        % \node at (-0.25,5.75) {$\stE$};
        % \node at (4.25,1.75) {$\stE$};
        % \node at (4.25,4.75) {$\stE$};
        % \node at (4.25,6.25) {$\stE$};
        \node [left] at (-0.25,2.25) {$i$};
        \node [left] at (-0.25,5.75) {$i'$};
        \node [right] at (4.25,1.75) {$o'(1)$};
        \node [right] at (4.25,4.75) {$o$};
        \node [right] at (4.25,6.25) {$o'(0)$};
      \end{scope}
      
    \end{tikzpicture}
    \caption{Simulations of $\bar{m}$ (left) and $\bar{s}$ (right) with the paths of the head of $M$ highlighted.}
    \label{fig:MSimulatesAutomata}
  \end{center}
\end{figure}

Our goal is to simulate any constant-weight normed network over these components.

\begin{lemma}
\label{lem:TMSimulatesNetwork}
Let $N = (S, I, O, \delta, \mu)$ be a normed network of size $n$ over $\bar{t}$, $\bar{m}$ and $\bar{s}$, and let $s_0 \in S$ be the state where the components are all in states $a$ or $a_0$.
Then there exist numbers $m = O(n^3)$ and $k = O(n^2)$, a function $F(i) = c i + d$ with $c, d = O(n^2)$, and an $(m,k,F)$-simulation of $N$ by $M$ from $s_0$.
\end{lemma}

\begin{proof}
We construct the simulation by placing the primitive components on a vertical column with $O(n)$ empty cells between them.
We then construct the wires so that the head of $M$ traverses each of them in an equal number of steps.

In this construction, a wire is represented by an empty region of $0$-symbols, and occasional $1$-symbols that cause the head of $M$ to turn right or left.
More explicitly, if $M$ encounters a single $1$-symbol directly on its path, then it will make a right turn as it moves on that symbol.
If $M$ encounters a $1$-symbol on the right hand side of its path, then it will make a left turn when that symbol can be moved onto the cell occupied by the head.
Since the $1$-symbols are moved when the head traverses a wire, it cannot be safely traversed again, but this is not an issue since the primitive components of Definition~\ref{def:PrimComponents}, and thus all normed networks over them, have the property that each input and output terminal can be used at most once.
A crossing of two wires is trivial to implement.

We now introduce a way for adding arbitrary delays in horizontal wires.
Figure~\ref{fig:Delays} shows two gadgets, one for delaying the head of $M$ for $9$ steps and one for $11$ steps.
By inserting copies of these gadgets into a wire of length $n$, we can implement any delay which is $O(n)$ and larger than a fixed constant.

\begin{figure}[ht]
  \begin{center}
    \begin{tikzpicture}
      
      \draw [step=1/2] (0,0) grid (3,2.5);
      \draw [densely dotted,>->] (-0.25,0.75) -| ++(1,1.5) -| ++(1.5,-1.5) -- ++(1,0);
      \node at (0.75,2.25) {$1$};
      \node at (1.25,0.25) {$1$};
      \node at (1.75,0.25) {$1$};
      \node at (2.25,2.25) {$1$};
      % \node at (-0.25,0.75) {$\stE$};
      % \node at (3.25,0.75) {$\stE$};
      
      \begin{scope}[xshift=5cm]
        \draw [step=1/2] (0,0) grid (4,3);
        \draw [densely dotted,>->] (-0.25,1.75) -| ++(2,0.5) -| ++(-0.5,-1.5) -| ++(1.5,1) -- ++(1.5,0);
        \foreach \x/\y in {0.5/0,0.5/2.5,2/1,2/2.5,2.5/1.5,3/0}{
          \node at (\x+0.25,\y+0.25) {$1$};
        }
        % \node at (-0.25,1.75) {$\stE$};
        % \node at (4.25,1.75) {$\stE$};
      \end{scope}
      
    \end{tikzpicture}
    \caption{Gadgets that implement a delay of $9$ steps (left) and $11$ steps (right) in a west-to-east wire. Note that traversing one step to the north takes two time steps.}
    \label{fig:Delays}
  \end{center}
\end{figure}

Label the components of $N$ as $A_1, \ldots, A_n$ and the wires as $W_1, \ldots, W_\ell$.
We choose $h = O(n)$, $k = h (n + 1)$ and $m = h (\ell + 1)$, and place a pattern $P_i$ that simulates $A_i$ in an all-$0$ configuration so that its southwest corner is at $(m - C, h i)$ for a constant $C$.
For each wire $W_i$ starting from an output terminal of some $P_j$, we extend a simulated wire to the east and make two left turns that redirect it to the west so that it passes above $P_j$.
This simulated wire is extended to the x-coordinate $h i$.
Within the rectangle spanned by the coordinates $(h i,0)$ and $(h(i+1)-1, k-1)$, we have enough space to implement an arbitrary delay of $O(n^2)$ steps using turns and the two delay components of Figure~\ref{fig:Delays}, so that each wire is traversed in exactly $e = O(n^2)$ steps.
The simulated wire is then extended east, either to an input terminal of some other $P_{j'}$ or the east border of the simulation region.
If $W_i$ originates from an input terminal of $N$, we use a similar construction, except that the simulated wire originates from the west border of the simulation region.
See Figure~\ref{fig:RowOfComponents} for a diagrammatic representation of the construction.

It is clear from the construction that the resulting $m \times k$-pattern $P$ realizes a simulation of $N$ by $M$ from $s_0$.
It remains to be shown that it is an $(m,k,F)$-simulation with $F(x) = e(x+1) + c x$, with $e$ being the number of steps in which $M$ traverses any wire in the simulation and $c$ the common weight factor of the simulated primitive components given by Lemma~\ref{lem:AutomataByM}.
But for any $i \in I^*$, the weight $\mu(s_0, i)$ is exactly the number of primitive components of $N$ that the token travels through in the computation of $\delta(s_0, i)$, or equivalently one less than the number of wires it travels through.
Hence the number of steps taken by $M$ in the simulation of $\delta(s_0, i)$ is $e(\mu(s_0,i)+1) + c \mu(s_0,i)$, as claimed.
\end{proof}

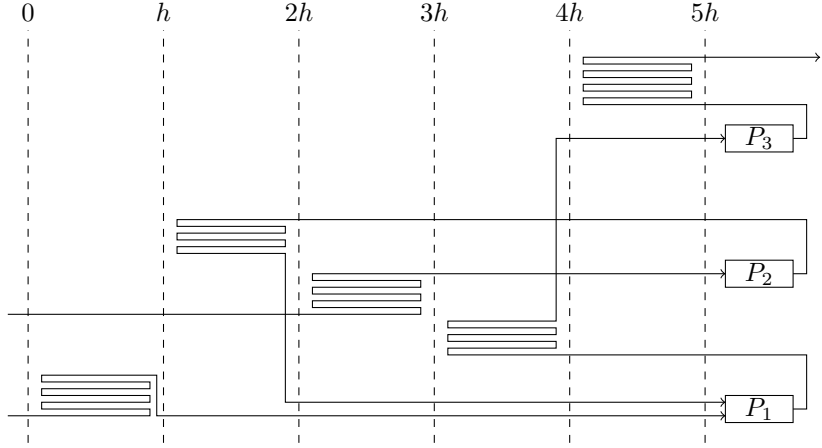
\begin{figure}[ht]
\begin{center}
\begin{tikzpicture}[scale=0.9]

  \foreach \y in {0,2,4}{
    \draw (10.3,\y) rectangle ++(1,0.4);
  }
  \foreach \y in {0,2,4,6,8,10}{
    \draw [dashed] (\y,-0.3) -- (\y,5.8);
  }

  \draw [->] (11.3,2.2) -| (11.5,3) -- (2.2,3) |- (3.8,2.9) |- (2.2,2.8) |- (3.8,2.7) |- (2.2,2.6) |- (3.8,2.5) |- (10.3,0.3);
  \draw [->] (-0.3,0.1) -| (1.8,0.2) -| (0.2,0.3) -| (1.8,0.4) -| (0.2,0.5) -| (1.8,0.6) -| (0.2,0.7) -| (1.9,0.1) -- (10.3,0.1);
  \draw [->] (-0.3,1.6) -- (5.8,1.6) |- (4.2,1.7) |- (5.8,1.8) |- (4.2,1.9) |- (5.8,2) |- (4.2,2.1) |- (10.3,2.2);
  \draw [->] (11.3,0.2) -| (11.5,1) -| (6.2,1.1) -| (7.8,1.2) -| (6.2,1.3) -| (7.8,1.4) -| (6.2,1.5) -| (7.8,4.2) -- (10.3,4.2);
  \draw [->] (11.3,4.2) -| (11.5,4.7) -| (8.2,4.8) -| (9.8,4.9) -| (8.2,5) -| (9.8,5.1) -| (8.2,5.2) -| (9.8,5.3) -| (8.2,5.4) -- (11.7,5.4);

  \node at (10.8,0.2) {$P_1$};
  \node at (10.8,2.2) {$P_2$};
  \node at (10.8,4.2) {$P_3$};

  \node [above] at (0,5.8) {$0$};
  \node [above] at (2,5.8) {$h$};
  \node [above] at (4,5.8) {$2h$};
  \node [above] at (6,5.8) {$3h$};
  \node [above] at (8,5.8) {$4h$};
  \node [above] at (10,5.8) {$5h$};

\end{tikzpicture}
\caption{A diagram of a simulation of a normed network by $M$.}
\label{fig:RowOfComponents}
\end{center}
\end{figure}

\section{Border Processes}
\label{sec:BorderProcesses}

In our construction, the head of $M$ repeatedly enters and exits the rectangular input region in order to extract information from it and manipulate its contents.
We now introduce tools that allow us to consider the border of a rectangle as an interface between its inside and outside and consider the two parts separately.
We first define the outside as an abstract process that satisfies certain properties, and then construct an implementation for its behavior as a concrete configuration of $M$.

A \emph{border process} of size $m \times n$ and depth $k$ is a pair $(\mathcal{P}, f)$ where $\mathcal{P} \subset \npats{[0,m-1] \times [0,n-1]}{Q}{\Sigma}$ is a set of patterns such that all $1$s and the head of $M$ are contained in $[k, m - 1 - k] \times [k, n - 1 - k]$, and $f : B_{m,n}^k \to B_{m,n}^k \times \N$ is a function such that whenever $f(\vec{b}_1, \ldots, \vec{b}_k) = (\vec{b}'_1, \ldots, \vec{b}'_k, t)$, each $\vec{b}'_i$ depends only on the prefix $\vec{b}_1, \ldots, \vec{b}_i$.
We denote $\vec{b}'_i = f[\vec{b}_1, \ldots, \vec{b}_i]$.
The border process represents the head of the Turing machine $M$ exiting and re-entering an $m \times n$ region containing a pattern from $\mathcal{P}$ exactly $k$ times, with the location of each entrance being a function of the locations of the previous exits.
The number $t$ represents the relative number of time steps required to complete the process; it plays a role in the definition of \emph{concrete realizations} later on.
We denote its maximum value by $T(f) = \max \{ t \in \N \;|\; w \in B_{m,n}^k, f(w) = (w', t) \}$.

Consider a border process $(\mathcal{P}, f)$ of size $m,n$ and depth $k$.
Take a pattern $P \in \mathcal{P}$ and let $x = x(P) = P \sqcup 0^{\Z^2 \setminus [0,m-1] \times [0,n-1]} \in X_Q \times \Sigma^{\Z^2}$.
We denote by $f(x) \in X_Q \times \Sigma^{\Z^2} \times \N$ the configuration and number obtained as follows.
By Lemma~\ref{lem:NoEscape}, the head eventually leaves the rectangle $[0, m-1] \times [0, n-1]$, and the $1$s stay contained in the region $[k-1, m-1-(k-1)] \times [k - 1, n - 1 - (k -1)]$.
Let $x_1 = \rho^M_{[0,m-1] \times [0,n-1]}(x)$ be the configuration right after the exit, let $\vec{b}_1 \in B_{m,n}$ be the location of the head, and let $\vec{b}'_1 = f[\vec{b}_1]$.
Let $y_1$ be the configuration obtained from $x_1$ by moving the head onto $\vec{b}'_1$, facing toward the region $[0, m-1] \times [0, n-1]$.
We let $y_1$ evolve until the head again leaves $[0, m-1] \times [0, n-1]$, at which point the $1$s are contained in $[k - 2, m - 1 - (k - 2)] \times [k - 2, n - 1 - (k - 2)]$.
Let $x_2 = \rho^M_{[0,m-1] \times [0,n-1]}(y_1)$ be the configuration right after the exit, let $\vec{b}_2 \in B_{m,n}$ be the position of the head and $\vec{b}'_2 = f[\vec{b}_1, \vec{b}_2]$, and let $y_2$ be the configuration obtained from $x_2$ by moving the head to $\vec{b}'_2$ facing toward $[0, m-1] \times [0, n-1]$.
In general, we keep defining configurations $y_i$ where the head faces toward the region $[0, m-1] \times [0, n-1]$ and the $1$s are contained in $[k - i, m - 1 - (k - i)] \times [k - i, n - 1 - (k - i)]$, let it evolve into a configuration $\rho^M_{[0,m-1] \times [0,n-1]}(y_i) = x_{i+1}$ where the head has just left $[0, m-1] \times [0, n-1]$ onto a coordinate $\vec{b}_{i+1} \in B_{m,n}$, define $\vec{b}'_{i+1} = f[\vec{b}_1, \ldots, \vec{b}_{i+1}]$, and construct $y_{i+1}$ from $x_{i+1}$ by moving the head to $\vec{b}'_{i+1}$ facing toward $[0, m-1] \times [0, n-1]$.
The process stops at $i = k$, and we define $f(x) = (y_k, t)$ with $f(\vec{b}_1, \ldots, \vec{b}_k) = (\vec{b}'_1, \ldots, \vec{b}'_k, t)$.

For $0 \leq i < k$, let $T_f(x, \vec{b}_1, \ldots, \vec{b}_i)$ be the number $t$ with $M^t(y_i) = x_{i+1}$, that is, the number of steps $M$ takes before leaving the region $[0,m-1] \times [0,n-1]$ for the $i$th time, counted from the last entrance.
In the case $i = 0$ we denote $y_0 = x$.
We say $f$ has \emph{consistent timing}, if for all $1 \leq i < k$ and all $P \in \mathcal{P}$ the number $T_f(x(P), \vec{b}_1, \ldots, \vec{b}_i)$ only depends on $i$, the entrance coordinate $\vec{b}'_{i-1}$ and the exit coordinate $\vec{b}_i$.
In particular, it should be independent of the choice of $P$.
Then we denote $T_f(x(P), \vec{b}_1, \ldots, \vec{b}_i) = T_f(i, \vec{b}'_{i-1}, \vec{b}_i)$.

A \emph{concrete realization} of a border process $(\mathcal{P}, f)$ of size $m,n$ and depth $k$ is a partial configuration $y \in \Sigma^{\Z^2 \setminus [0, m-1] \times [0, n-1]}$ with a `hole' of shape $[0, m-1] \times [0, n-1]$ with the following property.
Take any pattern $P \in \mathcal{P}$, and consider the configuration $x = y \sqcup P$.
Let $t_1 < s_1 < t_2 < s_2 < \cdots$ be the (potentially infinite) sequence of time steps at which the head leaves and enters the region $[0, m-1] \times [0, n-1]$ in the evolution of $x$ by $M$, and let $\vec{b}_1, \vec{b}'_1, \vec{b}_2, \vec{b}'_2, \ldots$ be the elements of $B_{m,n}$ at which this happens.
Then we have $f(\vec{b}_1, \ldots, \vec{b}_k) = (\vec{b}'_1, \ldots, \vec{b}'_k, t)$ with $s_k - t_1 = C + t$ for some constant $C$ depending only on $f$ and $y$.
The intuition is that the partial configuration $y$ implements the process of moving the head to the next coordinate of $B_{m,n}$ given by $f$ after it has left the region $[0, m-1] \times [0, n-1]$, and the number of steps from the first exit to the last entrance is also given by $f$ up to an additive constant.

An \emph{abstract realization} of a border process $(\mathcal{P}, f)$ of size $m,n$ and depth $k$ consists of a constant-weight normed network $A_f = (Q_f, I_f, O_f, \delta_f, c_f)$ over the primitive components $\bar{t}$, $\bar{m}$ and $\bar{s}$, where $I_f = \{ i_j(\vec{b}) \;|\; 1 \leq j \leq k, \vec{b} \in B_{m,n} \}$ and $O_f = \{ o_j(\vec{b}) \;|\; 1 \leq j < k, \vec{b} \in B_{m,n} \} \cup \{ o_k(\vec{b}, t) \;|\; \vec{b} \in B_{m,n}, 0 \leq t \leq T(f) \}$, and an initial state $q_0 \in Q_f$ such that for all border coordinates $\vec{b}_1, \ldots, \vec{b}_k \in B_{m,n}$ we have $\delta_f(q_0, i_1(\vec{b}_1) \cdots i_k(\vec{b}_k)) = (p, o_1(\vec{b}'_1) \cdots o_{k-1}(\vec{b}'_{k-1}) o_k(\vec{b}'_k, t))$ for some state $p \in Q_f$ with $f(\vec{b}_1, \ldots, \vec{b}_k) = (\vec{b}'_1, \ldots, \vec{b}'_k, t)$.
The \emph{abstract complexity} of $(\mathcal{P}, f)$ is the size of its smallest abstract realization.
Note that these concepts do not depend on the set $\mathcal{P}$, which justifies the notation $A_f$.
Figure~\ref{fig:AbstractRealization} depicts an abstract realization as a single s-automaton.
Each shaded vertical bar represents the head of $M$ exiting the region $[0,m-1] \times [0,n-1]$ via some $\vec{b} \in B_{m,n}$, corresponding to terminal $i_j(\vec{b})$, and re-entering via some $\vec{b}' \in B_{m,n}$ determined by $A_f$, which corresponds to terminal $o_j(\vec{b}')$ (or $o_k(\vec{b}', t)$ in the case $j = k$).

\begin{figure}[htp]
  \begin{center}
    \begin{tikzpicture}[scale=1.3]

      \draw [fill=black!15] (3.3,0) -- (0,0) -- (0,0.5) --
      ++(0,2.5) -- ++(0.5,0) -- ++(0,-2.5) -- ++(1.5,0) --
      %++(0,2.5) -- ++(0.5,0) -- ++(0,-2.5) -- ++(1.5,0) --
      ++(0,2.5) -- ++(0.5,0) -- ++(0,-2.5) -- ++(0.8,0);
      
      \draw [fill=black!15] (4.2,0) -- (7.5,0) -- (7.5,3) -- (7,3) -- (7,0.5) -- (5.5,0.5) --
      (5.5,3) -- (5,3) -- (5,0.5) -- (4.2,0.5);
      
      \foreach \x in {0,2,5,7}{
        \foreach \y in {0.75,1.75,2.25,2.75}{     
          \draw (\x,\y) -- (\x-0.2,\y);
          \draw (\x+0.5,\y) -- (\x+0.7,\y);
        }
        \node at (\x-0.15,1.3) {$\vdots$};
        \node at (\x+0.65,1.3) {$\vdots$};
      }
      %\draw (6.9,1.75) -- (6.7,1.75);
      %\foreach \y in {1.25,2.25,2.75}{
      %  \draw (7.4,\y) -- (7.6,\y);
      %}
      %\node at (7.55,1.8) {$\vdots$};
      
      \node [left] at (-0.2,0.75) {$i_1^s$};
      \node [left] at (-0.2,1.75) {$i_1^3$};
      \node [left] at (-0.2,2.25) {$i_1^2$};
      \node [left] at (-0.2,2.75) {$i_1^1$};
      \node [right] at (0.7,0.75) {$o_1^s$};
      \node [right] at (0.7,1.75) {$o_1^3$};
      \node [right] at (0.7,2.25) {$o_1^2$};
      \node [right] at (0.7,2.75) {$o_1^1$};
      
      \node [left] at (1.8,0.75) {$i_2^s$};
      \node [left] at (1.8,1.75) {$i_2^3$};
      \node [left] at (1.8,2.25) {$i_2^2$};
      \node [left] at (1.8,2.75) {$i_2^1$};
      \node [right] at (2.7,0.75) {$o_2^s$};
      \node [right] at (2.7,1.75) {$o_2^3$};
      \node [right] at (2.7,2.25) {$o_2^2$};
      \node [right] at (2.7,2.75) {$o_2^1$};
     
      \node [left] at (4.8,0.75) {$i_{k-1}^s$};
      \node [left] at (4.8,1.75) {$i_{k-1}^3$};
      \node [left] at (4.8,2.25) {$i_{k-1}^2$};
      \node [left] at (4.8,2.75) {$i_{k-1}^1$};
      \node [right] at (5.7,0.75) {$o_{k-1}^s$};
      \node [right] at (5.7,1.75) {$o_{k-1}^3$};
      \node [right] at (5.7,2.25) {$o_{k-1}^2$};
      \node [right] at (5.7,2.75) {$o_{k-1}^1$};
      
      \node [left] at (6.8,0.75) {$i_k^s$};
      \node [left] at (6.8,1.75) {$i_k^3$};
      \node [left] at (6.8,2.25) {$i_k^2$};
      \node [left] at (6.8,2.75) {$i_k^1$};
      \node [right] at (7.7,0.75) {$o_k^s(T(f))$};
      \node [right] at (7.7,1.75) {$o_k^1(2)$};
      \node [right] at (7.7,2.25) {$o_k^1(1)$};
      \node [right] at (7.7,2.75) {$o_k^1(0)$};

      \node at (3.75,0.25) {$\cdots$};
      
      \node at (2.25,0.25) {$A_f$};
      
    \end{tikzpicture}
  \end{center}
  \caption{An abstract realization of a border process, with shortened notation $i_j^p = i_j(\vec{b}_p)$, $o_j^p = o_j(\vec{b}_p)$ and $o_k^p(t) = o_k(\vec{b}_p, t)$ for some fixed enumeration $\{\vec{b}_1, \ldots, \vec{b}_s\} = B_{m,n}$.}
  \label{fig:AbstractRealization}
\end{figure}
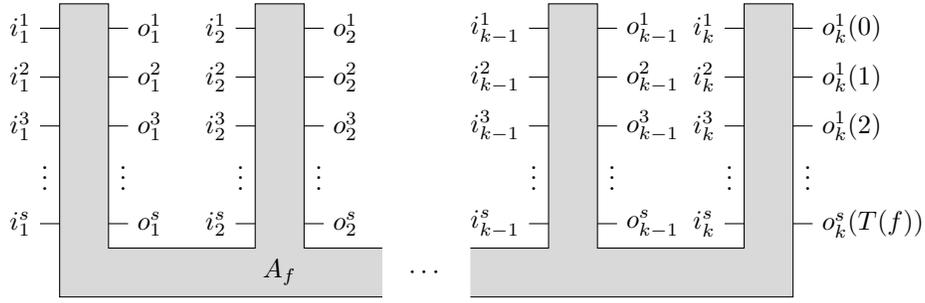

We now show that every border process $(\mathcal{P}, f)$ with consistent timing has a concrete realization whose size is polynomial in its abstract complexity.
Recall from Lemma~\ref{lem:TMSimulatesNetwork} that the machine $M$ can simulate an arbitrary normed network -- in particular an abstract realization $A_f$ of $(\mathcal{P}, f)$ -- using a gadget of polynomial size.
We construct a system of gadgets that acts as an `interface' between the rectangular region that $f$ acts on and the gadget that simulates $A_f$, guiding the head of $M$ to the correct terminals and border cells.

\begin{definition}
  Let $a, n \in \N$.
  A \emph{west $(a,n)$-catcher} is the pattern of shape $[0, 2 a + 9] \times [-a-3, n+1]$ with a $1$ at $(0, -a-2)$, $(3,-a-3)$, $(3,-a-2)$, $(3,n+1)$, $(2 a + 6, -a-3)$, $(2 a + 6, n+1)$ and $(2 a + 8, -a)$, and a $0$ in every other cell.
  The pattern is divided into three parts: the middle part consists of rows $0, 1, \ldots, n-1$, and the top and bottom parts are formed by the remaining rows.
  By rotating the pattern by $90$, $180$ and $270$ degrees, we obtain the \emph{north, east and south $(a,n)$-catchers}.
\end{definition}

The catcher is initially \emph{inactive}, and can be \emph{activated} by sending the head into the pattern from its south border on the westmost column.
The head travels on a spiraling path, pulling four $1$s closer to the center, and exits from the eastmost column of the south border.
Once the catcher is activated, we can use it to intercept the head of $M$ if it arrives from the east on row $n-a$.
See Figure~\ref{fig:Catcher} for a visualization.

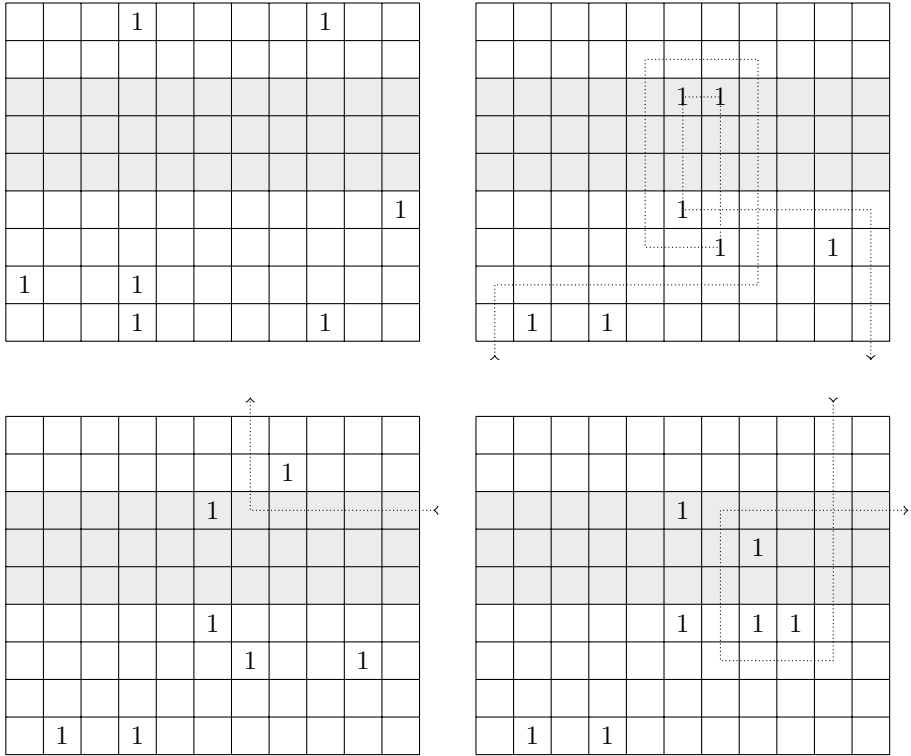
\begin{figure}[htp]
  \begin{center}
    \begin{tikzpicture}[scale=1/2]
      
      % m = 3, a = 1
      % [0, 11] x [-4, 4]
      
      \fill [gray!15] (0,0) rectangle (11,3);
      
      \draw (0,-4) grid (11,5);

      \foreach \x/\y in {0/-3,3/-4,3/-3,3/4,8/-4,8/4,10/-1}{
        \node at (\x+.5,\y+.5) {$1$};
      }
      
      \begin{scope}[xshift=12.5cm]%[yshift=-11cm]
      
      \fill [gray!15] (0,0) rectangle (11,3);
      
      \draw (0,-4) grid (11,5);

      \foreach \x/\y in {1/-4,3/-4,5/-1,5/2,6/-2,6/2,9/-2}{
        \node at (\x+.5,\y+.5) {$1$};
      }
      
      \draw [densely dotted,>->] (0.5,-4.5) |- ++(7,2) |- ++(-3,6) |- ++(2,-5) |- ++(-1,4) |- ++(5,-3) -- ++(0,-4);
      
      \end{scope}
      
      \begin{scope}[yshift=-11cm]%[yshift=-11cm]
      
      \fill [gray!15] (0,0) rectangle (11,3);
      
      \draw (0,-4) grid (11,5);

      \foreach \x/\y in {1/-4,3/-4,5/-1,5/2,6/-2,7/3,9/-2}{
        \node at (\x+.5,\y+.5) {$1$};
      }
      
      \draw [densely dotted,>->] (11.5,2.5) -| (6.5,5.5);
      
      \end{scope}
      
      \begin{scope}[xshift=12.5cm,yshift=-11cm]
      
      \fill [gray!15] (0,0) rectangle (11,3);
      
      \draw (0,-4) grid (11,5);

      \foreach \x/\y in {1/-4,3/-4,5/-1,5/2,7/-1,7/1,8/-1}{
        \node at (\x+.5,\y+.5) {$1$};
      }
      
      \draw [densely dotted,>->] (9.5,5.5) |- ++(-3,-7) |- ++(5,4);
      
      \end{scope}

%      
%      \begin{scope}[xshift=4cm]
%        \draw [step=1/2] (0,0) grid (3,5);
%        
%        \node at (0.25,0.25) {$1$};
%        \node at (1.25,1.75) {$1$};
%        \node at (1.25,3.75) {$1$};
%        \node at (1.75,1.25) {$1$};
%        \node at (1.75,3.75) {$1$};
%
%        \draw [densely dotted,>->] (-0.25,0.75) -| (2.25,4.25) -| (0.75,1.25) -| (1.75,3.75) -| (1.25,1.75) -- (3.25,1.75);
%      \end{scope}
%      
%      \begin{scope}[xshift=8cm]
%        \draw [step=1/2] (0,0) grid (3,5);
%        
%        \node at (0.25,0.25) {$1$};
%        \node at (1.25,1.75) {$1$};
%        \node at (1.25,3.75) {$1$};
%        \node at (2.25,1.75) {$1$};
%        \node at (2.25,3.25) {$1$};
%
%        \draw [densely dotted,>->] (3.25,1.25) -| (1.75,3.75) -- (3.25,3.75);
%      \end{scope}
      
    \end{tikzpicture}
    \caption{Top left: an inactive west $(1,3)$-catcher. Top right: activating the catcher. Bottom left: intercepting the head on row $2$ and redirecting it to the north. Bottom right: additional processing performed during partial activation of a catcher array. The middle part of the catcher, corresponding to rows $0$, $1$ and $2$, is shaded in each figure.}
    \label{fig:Catcher}
  \end{center}
\end{figure}

\begin{definition}
  Let $n \in \N$.
  A \emph{west catcher array of width $n$} is a pattern consisting of one west $(a,n)$-catcher for each $a = 1, 2, \ldots, n$ positioned side by side in this order.
  North, east and south arrays are defined as rotated versions of these patterns.
  
  A \emph{catcher system of width $n$} consists of west, north, east and south catcher arrays of width $n$ positioned so that their middle parts align with the square $[0, n-1]^2$, and each array lies in the respective direction from this square.
\end{definition}

We say that a catcher array or catcher system is inactive or activated if each individual catcher in it has this status.
Once a west catcher array is activated, it can intercept the head of $M$ as it arrives from the east on any of the rows on its middle part.
The $(a,n)$-catcher is responsible for row $n-a$: if the head arrives on row $n - a$, it passes through the $(n,n)$-catcher, then the $(n-1,n)$-catcher, all the way up to the $(a,n)$-catcher which finally intercepts it and sends it north.
In other words, the eastmost catcher (i.e. the one closest to $[0,n-1]^2$) is responsible for row $0$, the next one for row $1$, and so on, with the westmost catcher being responsible for row $n-1$.
The order is important, since it prevents the returning head from interacting with a catcher that is not responsible for its row.
On the other hand, if the array is inactive and the head enters it from the east or west along its middle part, the array does not interact with it, letting it pass directly through.
The south, east and north catcher arrays behave analogously, so an activated catcher system intercepts the head as it leaves the square $[0, n-1]^2$ from any border coordinate, while an inactive catcher system does not interact with it.

In our construction, the idea is to have a sequence of nested catcher systems which are activated one by one, starting with the outermost system.
The activated system intercepts the head as it exits $[0, n-1]^2$, so that we have a record of the exit coordinate.
Then the head activates the next system to prepare it for another exit.
However, this scheme does not give us a way to actually send the head back into $[0, n-1]^2$ as required by a border process, so we modify it slightly.

Suppose we want to send the head into $[0, n-1]^2$ from the west on row $i \in [0, n-1]$ and intercept it when it later exits the square.
To achieve this, we will activate the south, east and north catcher arrays of the catcher system and \emph{partially activate} the west catcher array: For each $a = 1, \ldots, n$ except for $i + 1$, we activate the $(a,n)$-catcher of the array.
The west $(i+1,n)$-catcher, which is closer to $[0, n-1]^2$ than the $(i,n)$-catcher of the same array, is left inactive.
Then we send the head in from the north edge of the activated $(i,n)$-catcher as in the bottom right part of Figure~\ref{fig:Catcher}, so that it is guided to the west on row $i$ and eventually enters $[0, n-1]^2$ on this row.
Note that since the west $(i+1,n)$-catcher was left inactive, the head will reach $[0, n-1]^2$ unobstructed.
The $(i,n)$-catcher is left with two $1$s that can intercept the head either on row $n-i$ or $n-(i+1)$.
The partial activation of south, east and north catcher arrays is defined similarly.

In addition to the catchers, we need some wiring to guide the head during the activation of the catcher arrays, and to feed the head into our computational gadget after it has been intercepted.
By the definition of a concrete realization, the time taken by the head to travel these wires and visit the square $[0, n-1]^2$ should not depend on the initial contents of the square.
This can be achieved due to our border process having consistent timing.

\begin{lemma}
  \label{lem:ConcreteRealization}
  Let $(\mathcal{P}, f)$ be a border process of size $n$, depth $k$, abstract complexity $C$ and consistent timing.
  Then $(\mathcal{P}, f)$ has a concrete realization $y \in \Sigma^{\Z^2 \setminus [0, n-1]^2}$ in which all $1$s are contained in $[-p, p]^2$ for some $p = O(C^3 + n^4 (k + T(f)))$.
\end{lemma}

\begin{proof}
By Lemma~\ref{lem:TMSimulatesNetwork}, there exist numbers $m = O(C^3)$ and $m' = O(C^2)$, a function $F(i) = c i + d$ with $c, d = O(C^2)$, and an $(m, m', F)$-simulation of $A_f$ by $M$.
We denote by $P_f$ the $m \times m'$-pattern realizing this simulation, and place it in $y$ at an arbitrary location that is sufficiently far away from the origin.

We place around the square $[0, n-1]^2$ one activated catcher system $C_0$ of width $n$ called the \emph{initial catcher system}.
For each output terminal $o_j(\vec{b})$ or $o_k(\vec{b}, t)$ of $A_f$, we also place one inactive catcher system $C_{j, \vec{b}}$ or $C_{k, \vec{b}, t}$ of width $n$, in such a way that $C_0$ is the outermost system, the systems $C_{j, \vec{b}}$ are placed successively closer to the origin in ascending order of $j$, and the systems $C_{k, \vec{b}, t}$ are placed the closest to the origin.

We add some wires to $y$ that connect the catcher systems to the pattern $P_f$.
For each $\vec{b} \in B_n$, from the position of the initial catcher array $C_0$ where the head of $M$ is redirected if it exits $[0, n-1]^2$ via $\vec{b}$, we add $1$-cells to $y$ that redirect it to the input terminal $i_0(\vec{b})$ of $P_f$.
For each output terminal $o_j(\vec{b})$ of $P_f$ with $j < k$, we add $1$-cells that redirect the head to the inactive catcher system $C_{j,\vec{b}}$, then guide it to partially activate the system so that the head can be sent into $[0, n-1]^2$ through the coordinate $\vec{b}$, and finally redirect it into the system to be sent in through that coordinate.
For each $\vec{b}' \in B_n$ we also add $1$-cells that, once the catcher system $C_{j,\vec{b}}$ is activated and catches the head as it exits $[0, n-1]^2$ via $\vec{b}'$, redirect it to the input terminal $i_j(\vec{b}')$ of $P_f$.
From each output terminal $o_k(\vec{b}, t)$ of $P_f$, we similarly guide the head to partially activate $C_{k, \vec{b}, t}$ and send the head into $[0, n-1]^2$ via $\vec{b}$.

As in the proof of Lemma~\ref{lem:TMSimulatesNetwork}, we have a lot of control on the number of steps the head of $M$ takes during these redirections.
First, we guarantee that the number of steps between the first exit from $[0, n-1]^2$ and the first entrance into $P_f$ via some terminal $i_1(\vec{b})$ is a constant, say $T_1$.
Recall that since $f$ has consistent timing, the number of steps the head spends inside $[0, n-1]^2$ after being sent in by the catcher array $C_{j, \vec{b}}$ is a function of $j$, $\vec{b}$ and the exit coordinate $\vec{b}'$, and we denote it by $T_f(j, \vec{b}, \vec{b}')$.
We can thus guarantee that the number of steps between the head exiting $P_f$ from terminal $o_j(\vec{b})$ and re-entering via terminal $i_{j+1}(\vec{b}')$ is also a constant, say $T_{j+1}$.
Finally, we guarantee that the number of steps between the head exiting $P_f$ via terminal $o_k(\vec{b}, t)$ and entering $[0, n-1]^2$ is $T_{k+1} + t$ for some constant $T_{k+1}$.

It is now easy to see that $y$ is a concrete realization of $(\mathcal{P}, f)$.
It remains to bound the number $p$.
The pattern $P_f$ has size $O(C^3) \times O(C^2)$, and we have a total of $O(k n^2 + T(f) n^2)$ catcher systems and ``wires'' guiding the head of $M$ from $P_f$ to these systems and back.
By Lemma~\ref{lem:FastEscape}, $O(n^4)$ delay components suffice for each wire, which can be laid out in an $O(n^2) \times O(n^2)$ pattern.
All this can be fit into a square pattern with side length $O(C^3 + n^4 (k + T(f)))$.
\end{proof}

\begin{figure}
\begin{center}
\begin{tikzpicture}[scale=0.5]

\fill [gray!15] (-6,0.2) rectangle (7,0.8);
\fill [gray!15] (0.2,-6) rectangle (0.8,7);

\draw (0.2,0.2) rectangle (0.8,0.8);

\foreach \x in {1.5,3.5,5.5}{
  \foreach \xx/\h in {0/0.7,0.5/0.5,1/0.3}{
    \draw [densely dotted] (\x+\xx,-0.1) rectangle (\x+\xx+0.4,1.1);
    \draw (\x+\xx,\h) -- ++(0.4,0);
    
    \draw [densely dotted] (-0.1,\x+\xx) rectangle (1.1,\x+\xx+0.4);
    \draw (1-\h,\x+\xx) -- ++(0,0.4);
    
    \draw [densely dotted] (0.6-\x-\xx,-0.1) rectangle (0.6-\x-\xx+0.4,1.1);
    \draw (0.6-\x-\xx,1-\h) -- ++(0.4,0);
    
    \draw [densely dotted] (-0.1,0.6-\x-\xx) rectangle (1.1,0.6-\x-\xx+0.4);
    \draw (\h,0.6-\x-\xx) -- ++(0,0.4);
  }
}

\end{tikzpicture}
\end{center}
\caption{The collection of catcher systems constructed in the proof of Lemma~\ref{lem:ConcreteRealization}, not drawn to scale. The solid lines mark the row or column each catcher is responsible for. The outermost catcher system is $C_0$.}
\label{fig:ManyArrays}
\end{figure}
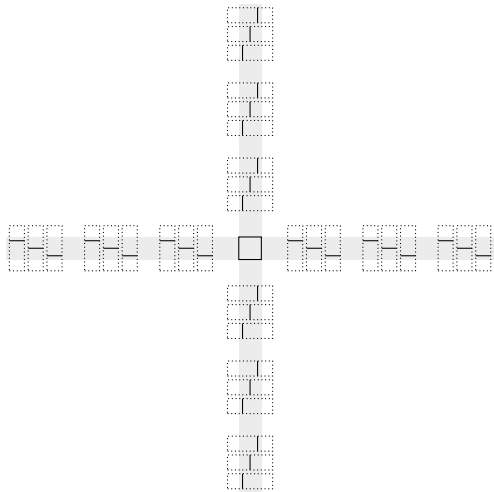

\section{Implementation of Transformations by Border Processes}

In this section, we prove that a border process can implement an arbitrary transformation on the set of local patterns of shape $m \times m$, if they are embedded inside a larger $n \times n$-pattern in a suitable way.
The statement is visualized in Figure~\ref{fig:ProcessCanDo}.
We can also guarantee that the process has consistent timing, as required by Lemma~\ref{lem:ConcreteRealization}, and we have precise control over its time component.

\begin{lemma}
\label{lem:ProcessCanDo}
  Let $m \in \N$, let $A \subset \npats{m}{Q}{\Sigma}$ be a set of patterns with disjoint $0$-orbits that contain the head of $M$, and let $C_1$ be the circuit complexity of $A$.
  Let $g : A \to \Sigma^{[0, m-1]^2}$ and $g' : A \to \N$ be any functions, and let $C_2$ be the circuit complexity of $R \mapsto (g(R), g'(R))$.
  Then there exists $n = O(m^4)$, a pattern $P \in \Sigma^{[0, 2n + m - 1]^2 \setminus [n, n+m-1]^2}$ and a border process $(\mathcal{P}, f)$ of size $2n+m$, depth $k = O(m^8)$ and abstract complexity $O(m^{32} + (m^8 C_1 + C_2 + m^6)^2)$ with consistent timing such that $\mathcal{P} = \{ R \sqcup P \;|\; R \in A \}$ and for each $R \in A$, denoting $x = 0^{\Z^2 \setminus [0, 2n+m-1]^2} \sqcup P \sqcup R$, we have $f(x) = (y,g'(R))$ such that $y|_{[n, n+m-1]^2}$ is a translated version of $g(R)$ and $x_{\vec{v}} = y_{\vec{v}} = 0$ for all $\vec{v} \in ([0, 2n+m-1] \times [n-3, n+m+2] \cup [n-3, n+m+2] \times [0, 2n+m-1]) \setminus [n,n+m-1]^2$.
\end{lemma}

% TODO: make pictures

\begin{figure}[htp]
  \begin{center}
    \begin{tikzpicture}

      % \foreach \x in {0,0.1,...,3.9,5,5.1,...,8.9}{
      %   \foreach \y in {0,0.1,...,3.9}{
      %     \pgfmathrandominteger{\c}{0}{7}
      %     \pgfmathparse{(\c==0 && (\x<4 || ((\x<6.3 || \x>7.6) && (\y<1.3 || \y>2.6)))) || (\c<=3 && \y>=1.5 && \y<2.5 && \x>=1.5 && \x<7.5 && (\x<2.5 || \x>=6.5)) ? 1 : 0}
      %     \ifnum\pgfmathresult=1
      %       \fill [black!30] (\x,\y) rectangle ++(0.1,0.1);
      %     \fi
      %   }
      % }

      % \foreach \x in {5,5.1,...,8.9}{
      %   \foreach \y in {0,0.1,...,3.9}{
      %     \pgfmathrandominteger{\c}{0}{6}
      %     \pgfmathparse{\c==0 && ((\x<6.3 || \x>7.6) && (\y<1.3 || \y>2.6)) ? 1 : 0}
      %     \ifnum\pgfmathresult=1
      %       \fill [black!30] (\x,\y) rectangle ++(0.1,0.1);
      %     \fi
      %   }
      %   }

      \fill [black!30] (1.4,0) rectangle (0,1.4);
      \fill [black!30] (1.4,4) rectangle (0,2.6);
      \fill [black!30] (2.6,0) rectangle (4,1.4);
      \fill [black!30] (2.6,4) rectangle (4,2.6);
      
      \fill [black!30] (6.4,0) rectangle (5,1.4);
      \fill [black!30] (6.4,4) rectangle (5,2.6);
      \fill [black!30] (7.6,0) rectangle (9,1.4);
      \fill [black!30] (7.6,4) rectangle (9,2.6);

      \draw [dashed] (1.4,0) -- (1.4,1.4) -- (0,1.4);
      \draw [dashed] (1.4,4) -- (1.4,2.6) -- (0,2.6);
      \draw [dashed] (2.6,0) -- (2.6,1.4) -- (4,1.4);
      \draw [dashed] (2.6,4) -- (2.6,2.6) -- (4,2.6);
      \draw (0,0) rectangle (4,4);
      \draw [fill=black!15] (1.5,1.5) rectangle (2.5,2.5);

      \draw (5,0) rectangle (9,4);
      \draw [fill=black!15] (6.5,1.5) rectangle (7.5,2.5);

      \draw [dashed] (6.4,0) -- (6.4,1.4) -- (5,1.4);
      \draw [dashed] (6.4,4) -- (6.4,2.6) -- (5,2.6);
      \draw [dashed] (7.6,0) -- (7.6,1.4) -- (9,1.4);
      \draw [dashed] (7.6,4) -- (7.6,2.6) -- (9,2.6);

      \node at (4.5,2) {$\stackrel{f}{\mapsto}$};
      \node at (2,2) {$R$};
      \node at (0.75,0.75) {$P$};
      \node at (7,2) {$g(R)$};
      \node at (5.75,0.75) {$?$};
      
    \end{tikzpicture}
  \end{center}
  \caption{A diagram of Lemma~\ref{lem:ProcessCanDo}, not drawn to scale. The inner and outer squares have size $m \times m$ and $(2n+m) \times (2n+m)$, respectively.}
  \label{fig:ProcessCanDo}
\end{figure}
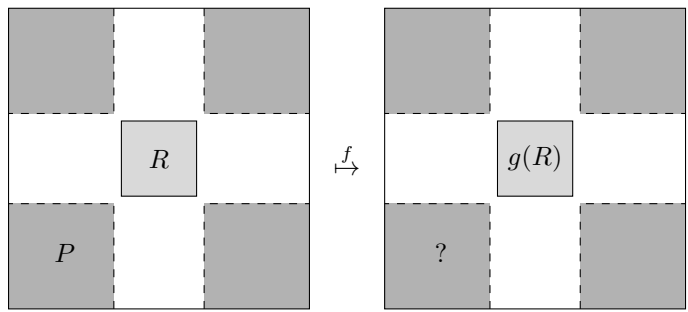

\begin{proof}
  We explicitly construct the pattern $P$, border process $(\mathcal{P}, f)$ and small abstract realization $A_f$ that have the required properties.
  The construction proceeds in five stages, which we call the \emph{escape stage}, \emph{probe stage}, \emph{transport stage}, \emph{computation stage} and \emph{sculpt stage}, some of which are further divided into substages.
  Each (sub)stage consists of a constant number of \emph{rounds}, and a round corresponds to defining $f[\vec{b}_1, \ldots, \vec{b}_p]$ for all applicable sequences of a fixed length $p$.
  Equivalently, round number $p$ consists of the head leaving the square $[0, 2n+m-1]^2$ for the $p$th time and returning on some border coordinate.
  
  In the escape stage, the head of $M$ escapes the region $R = [n, n+m-1]^2$ and then the larger square $[0, 2n + m - 1]^2$.
  In the probe stage, the head is repeatedly sent in from the west to probe the x-coordinate of a single $1$ occurring on a specific horizontal row of $[n, n+m-1]^2$, after which it is moved far away from the region by repeatedly colliding the head with it.
  After the probe stage, the region $[n, n+m-1]^2$ contains no more $1$s, and the transport stage consists of moving a large solid block of $1$s onto it.
  This is again achieved by repeatedly colliding the head with a single $1$, moving it toward the target by one diagonal step.
  In the computation stage, a separate part of the normed network $A_f$ simulates a Boolean circuit that computes the output pattern.
  The number $t = \tau^M_{[0, 2n+m-1]^2}(x)$ is also computed at this point.
  In the sculpt stage, some of the $1$s in the solid block are `carved out' and moved out of the region, in a process that mirrors the probe stage.
  In this way, the desired output pattern can be constructed on the region $[n, n+m-1]^2$ at the completion of the border process.
  
  We begin with the escape stage, and for this, we define the pattern $P$ to be filled with $0$s for now; we will add some $1$s to it later.
  By Lemma~\ref{lem:NoEscape}, the $1$s that are originally in the region $R$ will not travel outside $R_1 = R + [-1, 1]^2$ before the head of $M$ leaves the region $R_1$ either along some row with y-coordinate in $[n-2, n+m+1]$, or along some column with x-coordinate in $[n-2, n+m+1]$.
  Hence it will travel to some border of $P$, and escape at some coordinate $\vec b \in B_{2 n + m}$.
  Denote by $P' \in \Sigma^{R_1}$ the square pattern that remains at $R_1$ after this.
  
  In the abstract realization $A_f$, we add $4(2n+m)$ switches $\bar{s}_{1,\vec{b}}$, one for each $\vec{b} \in B_{2n+m}$, and connect each input terminal $i_0(\vec{b})$ of $A_f$ to terminal $i$ of $\bar{s}_{1,\vec{b}}$.
  These are called the \emph{escape switches}.
  We connect the terminal $o$ of each $\bar{s}_{1,\vec{b}}$ into one $4(2n+m)$-merge that combines them all into a single terminal.
  This is depicted in Figure~\ref{fig:EscapeNetwork}, which shows the input terminals $i_0(\vec{b})$ for all choices of $\vec{b} \in B_{2n+m}$, the escape switches and the merge.
  The $4(2n+m)$-merge increases the size of $A_f$ by $O(n^4)$, since we can implement a disposable $k$-merge with $O(k^2)$ primitive components.
  
  \begin{figure}[ht]
  \begin{center}
  \begin{tikzpicture}[yscale=-1]
  
  \fill [black!15] (0.7,-0.5) rectangle (3,3.5);
  
  \foreach \y in {0,0.5,1,3}{
    \draw (0,\y) -- (1.25,\y);
    \draw [fill=white] (1,\y-0.125) rectangle (1.75,\y+0.125);
    \draw [dotted] (1,\y) -- (1.75,\y);
    \draw (1.75,\y) -- (2,\y);
  }
  \draw (2,2) -- (2.5,2);
  \foreach \x/\y in {0.25/2}{
    \node [above=-0.24cm] at (\x,\y) {$\vdots$};
  }
  
  %\draw (0.7,-0.5) -- (0.7,3.5) -- (3,3.5);
  \draw (3,-0.5) -- (0.7,-0.5) -- (0.7,3.5);
  
  \node [left] at (0,0) {$i_0((-1,0))$};
  \node [left] at (0,0.5) {$i_0((-1,1))$};
  \node [left] at (0,1) {$i_0((-1,2))$};
  \node [left] at (0,3) {$i_0((2n+m-1,2n+m))$};
  \node [right] at (2.5,2) {$o$};
  
  \draw [ultra thick] (2,-0.25) -- (2,3.25);
    
  \end{tikzpicture}
  \end{center}
  \caption{The escape stage in $A_f$.}
  \label{fig:EscapeNetwork}
  \end{figure}
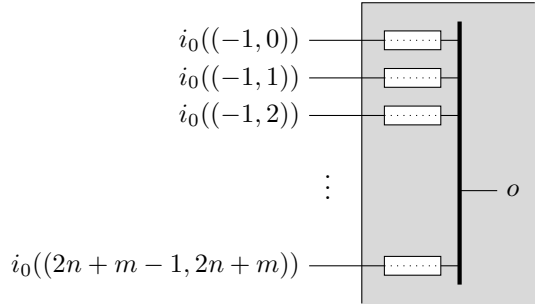
  
  Next, we handle the probe stage.
  It is divided into several substages, one for each $i \in [-1, m]$, and the purpose of substage $i$ is to determine the positions of $1$s on row $n+i$.
  For each substage $i$ and for each $j \in [0, m+1]$, choose some large number $n > N_{i,j} = O(m^4)$ with $N_{i,j} > N_{i',j'}$ when $(i,j) < (i',j')$ in the lexicographic order.
  For each $0 \leq k < N_{i,j}$, which constitute a total of $\sum_{i,j} N_{i,j} = O(m^6)$ rounds, we define $f$ to place the head on the west border of $P$ at y-coordinate $n+i-k$, regardless of where it exited on the previous rounds.
  
  Denote by $p(i,j,k) \in \N$ the number of the round corresponding to $i$, $j$ and $k$.
  In the normed network $A_f$, we implement each of these rounds by simply combining all input terminals $i_{p(i,j,k)}(\vec{b})$ into a $4(2n+m)$-merge and connecting its output into the correct output terminal, except when $k = 0$.
  In the case $k = 0$, we proceed as follows.
  For each $(a,b) \in R_1$ we add to $A_f$ an $(m+2,1)$-switch $\bar{s}^{(m+2,1)}_{(a,b)}$ whose task is to remember the bit $P'_{(a, b)}$.
  These switches are called the \emph{input switches}.
  For each $i$, $j$ and $a$, we connect the input terminal $i_{p(i,j,0)}((a,-1))$ of $A_f$ to the input terminal $i_j$ of $\bar{s}^{(m+2,1)}_{(a,n+i)}$.
  The outputs $o_j$ of these switches, as well as other input terminals $i_{p(i,j,0)}(\vec{b})$, are connected to a $4(2n+m)$-merge that connects them to the next output terminal.
  We also add $O(m^2)$ trivial merges to maintain the property of constant weight.
  See Figure~\ref{fig:ProbeAut} and Figure~\ref{fig:ProbeAut2} for a visualization.
  
  \begin{figure}[ht]
    \begin{center}
      \begin{tikzpicture}[yscale=-1]
        
        \fill [black!15] (-0.5,0) rectangle (1,3.5);
        % \draw (-0.5,3.5) -- (1,3.5) -- (1,0);
        \draw (1,3.5) -- (1,0) -- (-0.5,0);

        \fill [black!15] (5,0) rectangle (7,3.5);

        % \draw (5,0) -- (5,3.5) -- (7,3.5);
        \draw (7,0) -- (5,0) -- (5,3.5);

        \draw (0.5,1.75) -- (1.5,1.75);
        \node [left] at (0.5,1.75) {$o$};
        \node [right] at (1.5,1.75) {$o_{p-1}(\vec{b})$};

        \foreach \y in {0.5,1,2.5,3}{
          \draw (4.5,\y) -- (5.5,\y);
        }
        \node at (4.65,1.7) {$\vdots$};
        \node [left] at (4.5,0.5) {$i_p((-1,0))$};
        \node [left] at (4.5,1) {$i_p((-1,1))$};
        \node [left] at (4.5,2.5) {$i_p((2n+m-2,2n+m))$};
        \node [left] at (4.5,3) {$i_p((2n+m-1,2n+m))$};

        \draw [ultra thick] (5.5,0.25) -- (5.5,3.25);
        \draw (5.5,1.75) -- (6,1.75);
        \node [right] at (6,1.75) {$o$};

      \end{tikzpicture}
    \end{center}
    \caption{Implementation of round $p = p(i,j,k)$ of the probe stage in $A_f$ with $k \neq 0$, where $\vec{b} = (2n+m,n+i-k)$. The leftmost $o$ is the output terminal of the $4(2m+n)$-merge constructed in round $p-1$.}
    \label{fig:ProbeAut}
  \end{figure}
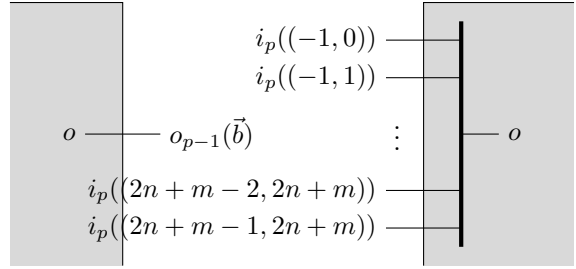
  
  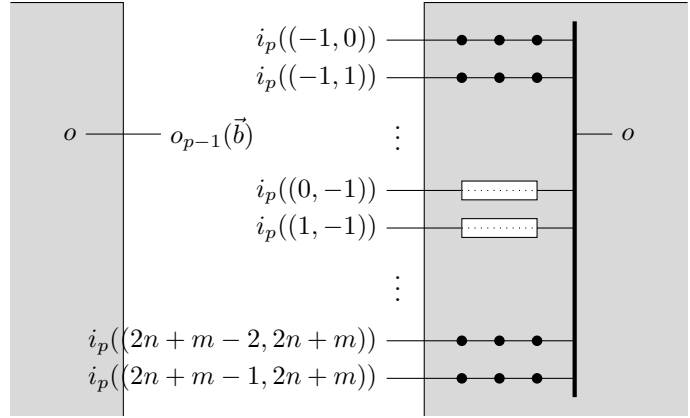
\begin{figure}[ht]
    \begin{center}
      \begin{tikzpicture}[yscale=-1]

        \fill [black!15] (-0.5,0) rectangle (1,5.5);
        %\draw (-0.5,5.5) -- (1,5.5) -- (1,0);
        \draw (1,5.5) -- (1,0) -- (-0.5,0);

        \fill [black!15] (5,0) rectangle (8.5,5.5);
        % \draw (5,0) -- (5,5.5) -- (8.5,5.5);
        \draw (8.5,0) -- (5,0) -- (5,5.5);

        \draw (0.5,1.75) -- (1.5,1.75);
        \node [left] at (0.5,1.75) {$o$};
        \node [right] at (1.5,1.75) {$o_{p-1}(\vec{b})$};

        \foreach \y in {0.5,1,4.5,5}{
          \draw (4.5,\y) -- (7,\y);
          \foreach \x in {5.5,6,6.5}{
            \fill (\x,\y) circle (0.07);
          }
          %\draw (4.5,\y) -- (5.5,\y);
          %\draw [fill=white] (5.5,\y-0.125) rectangle (6.5,\y+0.125);
          %\draw [dashed] (5.5,\y) -- (6.5,\y);
          %\draw (6.5,\y) -- (7,\y);
        }
        \foreach \y in {2.5,3}{
          \draw (4.5,\y) -- (5.5,\y);
          \draw [fill=white] (5.5,\y-0.125) rectangle (6.5,\y+0.125);
          \draw [dotted] (5.5,\y) -- (6.5,\y);
          \draw (6.5,\y) -- (7,\y);
        }
        \node at (4.65,1.7) {$\vdots$};
        \node at (4.65,3.7) {$\vdots$};
        \node [left] at (4.5,0.5) {$i_p((-1,0))$};
        \node [left] at (4.5,1) {$i_p((-1,1))$};
        \node [left] at (4.5,2.5) {$i_p((0,-1))$};
        \node [left] at (4.5,3) {$i_p((1,-1))$};
        \node [left] at (4.5,4.5) {$i_p((2n+m-2,2n+m))$};
        \node [left] at (4.5,5) {$i_p((2n+m-1,2n+m))$};

        \draw [ultra thick] (7,0.25) -- (7,5.25);
        \draw (7,1.75) -- (7.5,1.75);
        \node [right] at (7.5,1.75) {$o$};
  
        % \foreach \y in {0,0.5,1,3}{
        % \draw (0,\y) -- (2,\y);
        % \foreach \x in {0.5,1,1.5}{
        % \fill (\x,\y) circle (0.07);
        % }
        %   %   \draw (0.5,\y-0.125) rectangle (1.5,\y+0.125);
        %   %   \draw [dashed] (0.5,\y) -- (1.5,\y);
        %   %   \draw (1.5,\y) -- (2,\y);
        %   \draw (2.5,\y) -- (3,\y);
        % }
        %   \foreach \y in {1.75,2.25}{
        %   \draw (0,\y) -- (0.5,\y);
        %   \draw (0.5,\y-0.125) rectangle (1.5,\y+0.125);
        %   \draw [dashed] (0.5,\y) -- (1.5,\y);
        %   \draw (1.5,\y) -- (2,\y);
        % }
        
        %   \draw (2,2) -- (3,2);
        %   \foreach \x/\y in {0.25/1.375,0.25/2.625,2.75/1.5,2.75/2.5}{
        %   \node [above=-0.24cm] at (\x,\y) {$\vdots$};
        % }
        
        %   \node [left] at (0,0) {$i_{p(i,j,0)}((-1,0))$};
        %   \node [left] at (0,0.5) {$i_{p(i,j,0)}((-1,1))$};
        %   \node [left] at (0,1) {$i_{p(i,j,0)}((-1,2))$};
        %   \node [left] at (0,1.75) {$i_{p(i,j,0)}((0,-1))$};
        %   \node [left] at (0,2.25) {$i_{p(i,j,0)}((1,-1))$};
        %   \node [left] at (0,3) {$i_{p(i,j,0)}((m-1,n))$};
        
        %   \node [right] at (3,0) {$o_{p(i,j,0)}((-1,0))$};
        %   \node [right] at (3,0.5) {$o_{p(i,j,0)}((-1,1))$};
        %   \node [right] at (3,1) {$o_{p(i,j,0)}((-1,2))$};
        %   \node [right] at (3,2) {$o_{p(i,j,0)}((n,n+i))$};
        %   \node [right] at (3,3) {$o_{p(i,j,0)}((m-1,n))$};
        
        %   \draw [ultra thick] (2,-0.25) -- (2,3.25);
        
      \end{tikzpicture}
    \end{center}
    \caption{Implementation of round $p = p(i,j,0)$ of the probe stage in $A_f$ with $\vec{b} = (2n+m,n+i-k)$. The leftmost $o$ is the output terminal of the $4(2m+n)$-merge constructed in round $p-1$. Each terminal $i_p((a,-1))$ is connected to the input switch $\bar{s}^{(m+2,1)}_{(a,n+i)}$.}
    \label{fig:ProbeAut2}
  \end{figure}
  
  The effect of this definition is the following.
  When $i = -1$ and $j = k = 0$, the head enters on row $n-1$, which is the lowest row of $R_1$.
  The head travels to the east until it encounters either a $1$ on the lowest row of $R_1$ or the east border of $P$.
  If a $1$ is encountered at some coordinate $(a, n-1)$, then the head makes a right turn, moves the $1$ to $(a-1,n-2)$, travels to the south border of $P$ and exits at coordinate $(a, -1)$.
  If no such $1$ was encountered, the the head exits at $(2n+m, n-1)$.
  We informally say that $f$ `remembers' this information, as its subsequent output values may depend on it.
  In general, on round $i,j,k$ of the probe stage $f$ remembers the exit coordinate of the head if $k = 0$.
  In the normed network $A_f$, this is implemented with the input switches $\bar{s}^{(m+2,1)}_{(a,b)}$.
  
  Having obtained the position of the $1$, $f$ places the head on the west border of $P$ at y-coordinate $n-2$, where it encounters the $1$ again, this time at coordinate $(a-1, n-2)$, moving it to $(a-2, n-3)$, and again exiting at the south border.
  We collide with the $1$ for a total of $N_{0,0}$ times, eventually placing it at $(a-N_{0,0}, n-1-N_{0,0})$.
  The purpose of this is to transport the $1$ out of the way, so that we can safely probe for the position of all $1$s in $P'$, and perform the remaining stages.
  This process is visualized in Figure~\ref{fig:ProbeStage}.

  \begin{figure}[htp]
    \begin{center}
      \begin{tikzpicture}[scale=0.5]
        
        \begin{scope}
          \foreach \x/\y in {1/4,3/3,3/5,4/3,4/4,4/6,5/5,6/3,6/4,7/6}{
            \fill [black!15] (\x,\y) rectangle ++ (1,1);
            \node at (\x+1/2,\y+1/2) {$1$};
          }
          
          \draw (0,0) grid (8,7);
        \end{scope}
        
        \begin{scope}[xshift=10cm]
          \foreach \x/\y in {1/4,2/2,3/5,4/3,4/4,4/6,5/5,6/3,6/4,7/6}{
            \fill [black!15] (\x,\y) rectangle ++ (1,1);
            \node at (\x+1/2,\y+1/2) {$1$};
          }
          
          \draw [densely dotted,>->] (-0.5,3.5) -| (3.5,-0.5);
          
          \draw (0,0) grid (8,7);
        \end{scope}
        
        \begin{scope}[yshift=-9cm]
          \foreach \x/\y in {1/4,1/1,3/5,4/3,4/4,4/6,5/5,6/3,6/4,7/6}{
            \fill [black!15] (\x,\y) rectangle ++ (1,1);
            \node at (\x+1/2,\y+1/2) {$1$};
          }

          \draw [densely dotted,>->] (-0.5,2.5) -| (2.5,-0.5);
          
          \draw (0,0) grid (8,7);
        \end{scope}
        
        \begin{scope}[xshift=10cm,yshift=-9cm]
          \foreach \x/\y in {1/4,0/0,3/5,4/3,4/4,4/6,5/5,6/3,6/4,7/6}{
            \fill [black!15] (\x,\y) rectangle ++ (1,1);
            \node at (\x+1/2,\y+1/2) {$1$};
          }

          \draw [densely dotted,>->] (-0.5,1.5) -| (1.5,-0.5);
          
          \draw (0,0) grid (8,7);
        \end{scope}
        
      \end{tikzpicture}
    \end{center}
    \caption{The head of $M$ extracts a $1$ from the bottom row of the region $R'$.}
    \label{fig:ProbeStage}
  \end{figure}
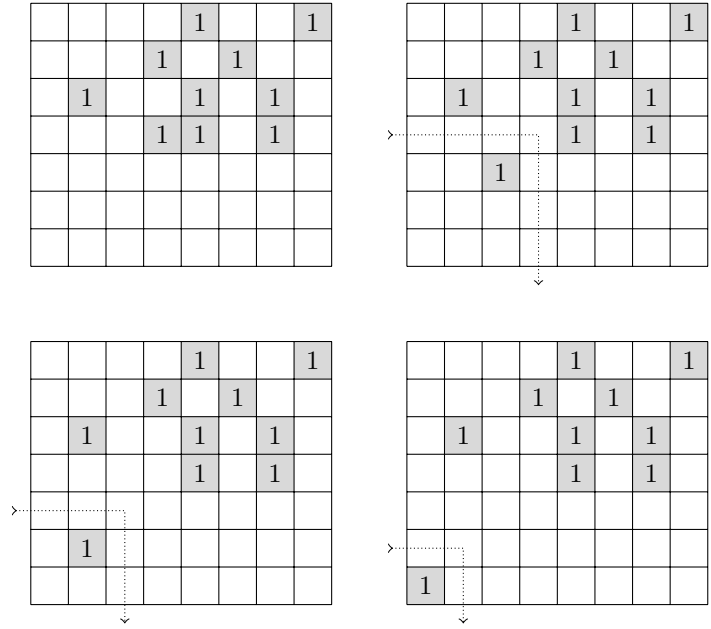

  Indeed, after the process of moving the $1$, $f$ places the head on row $n-1$ again, probing for the position of another $1$.
  This corresponds to round $p(-1,1,0)$.
  This process repeats $m+2$ times regardless of how many $1$s were found, after which we proceed to substage $0$ and start probing row $n$.
  We continue in this way, repeatedly probing for a $1$ on the lowest nonempty row, moving it out of the way, and proceeding to the next row after $m+2$ repetitions.
  Whenever the head encounters a $1$ at some coordinate $(a,n+i)$ during round $p(i,j,0)$, it exits $P$ at coordinate $(a,-1)$, and the token of $A_f$ is routed through terminal $i_j$ of input switch $\bar{s}^{(m+2,1)}_{(a,n+i)}$.
  The reason for using an $(m+2,1)$-switch is that the value of $j$ depends on the number of $1$-symbols on row $n+i$ to the left of $(a,n+i)$, and can vary between different input patterns.
  
  The last round of this process is $p(m,m+1,N_{m,m+1})$.
  If $n$ is large enough ($n = \Omega(m^4)$ suffices for this), there is enough room to transport each $1$ out of the way following the first collision with it, after which it does not affect the remaining steps of this stage.
  Once the top row of $R_1$ has been completely probed, the head exits through the east border of $P$ and we have complete information of the contents of $R_1$ at the end of the escape stage.
  To implement the probe stage, we added to $A_f$ the $O(m^2)$ input switches, and for each of the $O(m^6)$ rounds, a single $4(2n+m)$-merge and possibly $O(m^2)$ trivial merges.
  These changes can be implemented by $O(m^8 n^4)$ primitive components.
  This concludes the probe stage.
  
  We now describe the transport stage.
  At this point, the region $R$ contains only $0$s, and there are also $0$s to the east, north, west and south of it.
  There is a collection of $1$s to the southwest of $R$, which are the remnants of the probe stage. %, and we can assume that the distance $d$ from $R$ to the closest of these $1$s is at least $K m^2$ for some large $K > 0$.
  We introduce new $1$s into the pattern $P$, to the northeast of the region $R$.
  For each $j \in [0, 5 m^2 + 1]$ and $i \in [0, 5 m^2 + 1 + j]$, we add a $1$ at the coordinate $(n + 3 + i, n + 3 + j)$.
  We call this pattern of $1$s the \emph{clay}, as it will be used in the sculpt stage.
  Of course, the clay also present during the escape and probe stages, but does not affect them, as its southmost row lies two steps to the north of the row $n + m + 1$, which is the northmost row of $R_1$.
  Note that the bottom row of $1$s in this pattern is the longest one, and the lengths of consecutive rows differ by $1$.
  
  We describe a continuation of the border process $f$ that translates the clay to the southwest by a single step.
  This is done one row at a time, starting from the bottom row, and each row is translated one cell at a time, from west to east.
  We send the head of $M$ in from the south border of $P$, on the column $n + m + 2$, which is one step to the west of the westmost column of the clay.
  The head travels north until it is diagonally adjacent to the southwest corner of the clay, where it makes a left turn to the west and moves the $1$ at the corner from $(n + m + 3, n + m + 3)$ to $(n+m+2,n+m+2)$, and then travels to the west border of $P$.
  Next, the head is placed on the column $n + m + 3$, one step to the east.
  It again travels north, moves a $1$ from $(n + m + 4, n + m + 3)$ to $(n+m+3, n + m + 2)$ by a left turn, and travels to the west border of $P$.
  We continue in this manner, placing the head on column $n+m+2+i$ for each $i \in [0,5m^2+1]$ and moving a single $1$ on the row $n+m+2$.
  We then repeat this procedure for the other rows: for each $j$ from $0$ to $5m^2+1$, and for each $i$ from $0$ to $5m^2+1+j$, we place the head on the south border of $P$ on column $n+m+2+i$, from where it travels north and moves a $1$ from $(n+m+3+i, n+m+3+j)$ to $(n+m+2+i, n+m+2+j)$ by a left turn.
  The rules of $M$ prevent the head from making any additional turns; the distinct heights of the columns are used to guarantee this when the eastmost $1$s are moved.
  When this operation is complete, we have moved the entire clay one step to the southeast.
  By repeating it, we move the clay for a total of $5m^2+m+3$ steps, after which we have a $1$ at coordinate $(n + i, n + j)$ for all $j \in [m-5m^2, m-1]$ and $i \in [m-5m^2, m-1+j]$.
  Since the values of $f$ during this stage are fixed, for each step of the process we only need to add a single $4(m+2n)$-merge to $A_f$ that directs the token to the correct output terminal, as we did in the probe stage for the rounds $p(i,j,k)$ with $k \neq 0$.
  The number of primitive components needed for this is $O(m^6 n^4)$.
  This concludes the transport stage.
  
  We continue with the computation stage, which consists of zero rounds, but takes up a substantial part of the network $A_f$.
  For a pattern $P' \in A$, denote by $\tau(P') = \tau^M_{[0,2n+m-1]^2}(0^{\Z^2 \setminus [0,2n+m-1]^2} \sqcup P' \sqcup P)$ the escape time of the head from the region $[0,2n+m-1]^2$.
  From Lemma~\ref{lem:CircuitSimulation} and Lemma~\ref{lem:FastEscape} we obtain a circuit of size $O(n^4 (C_1 + n^2))$ that computes the function $\rho^M(P' \sqcup P) \mapsto (P', \tau(P'))$ for $P' \in A$.
  By assumption, there is a circuit of size $C_2$ that computes the function $P' \mapsto (g(P'), g'(P'))$.
  We chain these circuits together to obtain a circuit $C$ of size $O(n^4 (C_1 + n^2) + C_2)$ that computes the function $\rho^M(P') \mapsto (g(P'), g'(P'))$.
  We apply Lemma~\ref{lem:CircuitAutomaton} to transform this circuit into a normed network $A_C$ of size $O((n^4 (C_1 + n^2) + C_2)^2)$.
  We connect the output terminals $o'(x)$ of the input switches $\bar{s}^{(m+2,1)}_{(a,b)}$ into the input terminals $i_j(x)$ of $A_C$ in a suitable way.
  Each output terminal $o_j$ of $A_C$ is connected to the input terminal $i_{j+1}(x)$ of the next switch, except the last one, which loops to the input terminal $\#$.
  The output terminal $\$$ is connected to the input terminal $i_0$ of $A_C$.
  The output terminals $o_j(0)$ and $o_j(1)$ are directed to a merge, with the latter running through a $(1,O(m^2))$-switch $\bar{s}^{(1,O(m^2))}_j$ and setting its bit to $1$, and the merge is directed to the next input terminal $i_{j+1}$ of $A_C$.
  On the last step, we do not connect the merge to anything yet.
  See Figure~\ref{fig:CompStage} for an illustration.
  
  We now have a collection of $(1,O(m^2))$-switches whose internal values encode the contents of the target pattern as well as the number $g'(P')$, and we call them the \emph{output switches}.
  They are used in the next stage of the construction.
  This concludes the computation stage.
  
  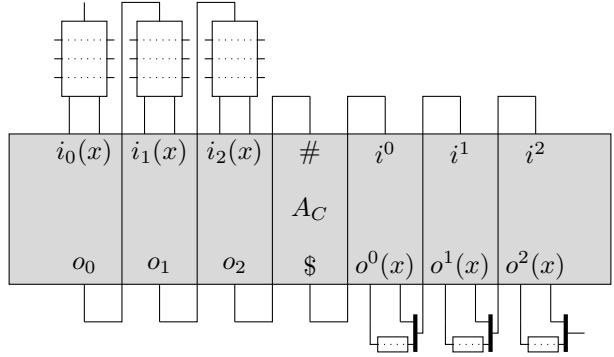
\begin{figure}[htp]
  \begin{center}
  \begin{tikzpicture}
  
    % Big block C
    \draw [fill=black!15] (0,0) rectangle (8,2);
    \node at (4,1) {$A_C$};
    
    \foreach \x in {0.8,1.2, 1.8,2.2, 2.8,3.2,  4,  5, 6, 7}{
      \draw (\x,2) -- ++(0,0.5);
    }
    \node at (1,1.75) {$i_0(x)$};
    \node at (2,1.75) {$i_1(x)$};
    \node at (3,1.75) {$i_2(x)$};
    \node at (4,1.75) {$\#$};
    \node at (5,1.75) {$i^0$};
    \node at (6,1.75) {$i^1$};
    \node at (7,1.75) {$i^2$};
    
    \foreach \x in {1, 2, 3,  4,  4.8,5.2, 5.8,6.2, 6.8,7.2}{
      \draw (\x,0) -- ++(0,-0.5);
    }
    \node at (1,0.25) {$o_0$};
    \node at (2,0.25) {$o_1$};
    \node at (3,0.25) {$o_2$};
    \node at (4,0.25) {$\$$};
    \node at (5,0.25) {$o^0(x)$};
    \node at (6,0.25) {$o^1(x)$};
    \node at (7,0.25) {$o^2(x)$};
    
    % Input switches
    \foreach \x in {1,2,3}{
      \draw (\x-0.3,2.5) rectangle (\x+0.3,3.5);
      \foreach \y in {2.75,3,3.25}{
        \draw (\x-0.3,\y) -- ++(-0.1,0);
        \draw [dotted] (\x-0.3,\y) -- (\x+0.3,\y);
        \draw (\x+0.3,\y) -- ++(0.1,0);
      }
    }
    \draw (1,3.5) -- (1,3.75);
    \foreach \x in {1,2}{
      \draw (\x,-0.5) -| (\x+0.5,3.75) -| (\x+1,3.5);
    }
    
    % Middle wires
    \foreach \x in {3,4}{
      \draw (\x,-0.5) -| (\x+0.5,2.5) -- (\x+1,2.5);
    }
    
    % Output switches and merges
    \foreach \x in {5,6,7}{
      \draw (\x-0.1,-0.7) rectangle (\x+0.3,-0.9);
      \draw [dotted] (\x-0.1,-0.8) -- (\x+0.3,-0.8);
      \draw (\x-0.2,-0.5) |- (\x-0.1,-0.8);
      \draw (\x+0.3,-0.8) -- (\x+0.4,-0.8);
      \draw [ultra thick] (\x+0.4,-0.4) -- (\x+0.4,-0.9);
      \draw (\x+0.2,-0.5) -- (\x+0.4,-0.5);
    }
    \foreach \x in {5,6}{
      \draw (\x+0.4,-0.65) -| (\x+0.5,2.5) -- (\x+1,2.5);
    }
    \draw (7.4,-0.65) -- ++(0.25,0);
  
  \end{tikzpicture}
  \end{center}
  \caption{The implementation of the computation stage in $A_f$.}
  \label{fig:CompStage}
  \end{figure}
  
  We describe the sculpt stage.
  Intuitively, it resembles the probe stage, but in reverse and with the roles of $0$ and $1$ inverted.
  Recall that after the previous stages, the clay occupies the coordinates $(n+m-5m^2,n+m-5m^2) + (i,j)$ for $j \in [0,5m^2-1]$ and $i \in [0,5m^2+j-1]$.
  Our goal is to remove some $1$s from the clay so that the output patterm $g(R)$ is formed at the coordinates $[n, n+m-1]^2$.
  At this point, the process $f$ has full information of $g(R)$.
  Let $E_0(R) \subset [0, m-1]^2$ be the set of coordinates $\vec{v}$ with $g(R)_{\vec{v}} = 0$.
  The idea is that for each $\vec{v} = (i,j) \in [0, m-1]^2$ in decreasing order of $i+j$, we do the following:
  \begin{enumerate}
  \item
    Remove rows from the south border of the clay or columns from its west border until its southwest corner is on the same diagonal as $\vec{v} + (n,n)$.
  \item
    If $\vec{v} \in E_0(R)$, move each $1$ on the diagonal path from $(n,n) + \vec{v}$ to the southwest corner of the clay one step to the southwest.
    Then move the $1$ that was previously at the corner of the clay out of the way.
    % This is illustrated in Figure~\ref{fig:SculptStage}.
    If $\vec{v} \notin E_0(R)$, send the head into the region $P$ the same number of times as we would in the first case, but along an empty row.
  \end{enumerate}
  %It is not hard to see that at the beginning of the first bullet point, the southwest corner of the clay is at $(n-5m^2+2m+(2m-1)i, n-5m^2+2(mi+j))$, which lies on the same diagonal as $(n,n) + \vec{v}_{i,j}$ to the southwest of it.
  %Also, as $i$ runs through $[0,2m]$ and $j$ runs through $[0,m-1]$, the values of $\vec{v}_{i,j}$ cover $R$.
  
  \begin{figure}[htp]
  \begin{center}
  \begin{tikzpicture}[scale=0.5]
    
    \begin{scope}
    \foreach \x in {3,4,5,6,7}{
      \foreach \y in {3,4,5,6}{
        \fill [black!15] (\x,\y) rectangle ++ (1,1);
        \node at (\x+1/2,\y+1/2) {$1$};
      }
    }
  
    \draw (0,0) grid (8,7);
    \end{scope}
    
    \begin{scope}[xshift=10cm]
    \foreach \x in {4,5,6,7}{
      \foreach \y in {3,4,5,6}{
        \fill [black!15] (\x,\y) rectangle ++ (1,1);
        \node at (\x+1/2,\y+1/2) {$1$};
      }
    }
    \foreach \y in {4,5,6}{
      \fill [black!15] (3,\y) rectangle ++ (1,1);
      \node at (3+1/2,\y+1/2) {$1$};
    }
    
    \fill [black!15] (2,2) rectangle ++ (1,1);
    \node at (2+1/2,2+1/2) {$1$};
    
    \draw [densely dotted,>->] (2.5,-0.5) |- (-0.5,2.5);
  
    \draw (0,0) grid (8,7);
    \end{scope}
    
    \begin{scope}[yshift=-9cm]
    \foreach \x in {4,5,6,7}{
      \foreach \y in {3,4,5,6}{
        \fill [black!15] (\x,\y) rectangle ++ (1,1);
        \node at (\x+1/2,\y+1/2) {$1$};
      }
    }
    \foreach \y in {4,5,6}{
      \fill [black!15] (3,\y) rectangle ++ (1,1);
      \node at (3+1/2,\y+1/2) {$1$};
    }
    
    \fill [black!15] (3,1) rectangle ++ (1,1);
    \node at (3+1/2,1+1/2) {$1$};
    
    \draw [densely dotted,>->] (8.5,1.5) -| (3.5,-0.5);
  
    \draw (0,0) grid (8,7);
    \end{scope}
    
    \begin{scope}[xshift=10cm,yshift=-9cm]
    \foreach \x in {3,5,6,7}{
      \foreach \y in {3,4,5,6}{
        \fill [black!15] (\x,\y) rectangle ++ (1,1);
        \node at (\x+1/2,\y+1/2) {$1$};
      }
    }
    \foreach \y in {2,3,5,6}{
      \fill [black!15] (4,\y) rectangle ++ (1,1);
      \node at (4+1/2,\y+1/2) {$1$};
    }
    
    \draw [densely dotted,>->] (8.5,1.5) -| (3.5,3.5) -- (-0.5,3.5);
  
    \draw (0,0) grid (8,7);
    \end{scope}
    
    \begin{scope}[yshift=-18cm]
    \foreach \x in {3,6,7}{
      \foreach \y in {3,4,5,6}{
        \fill [black!15] (\x,\y) rectangle ++ (1,1);
        \node at (\x+1/2,\y+1/2) {$1$};
      }
    }
    \foreach \y in {2,3,4,5,6}{
      \fill [black!15] (4,\y) rectangle ++ (1,1);
      \node at (4+1/2,\y+1/2) {$1$};
    }
    \foreach \y in {3,4,6}{
      \fill [black!15] (5,\y) rectangle ++ (1,1);
      \node at (5+1/2,\y+1/2) {$1$};
    }
    
    \draw [densely dotted,>->] (-0.5,5.5) -| (5.5,-0.5);
  
    \draw (0,0) grid (8,7);
    \end{scope}
    
    \begin{scope}[xshift=10cm,yshift=-18cm]
    \foreach \x in {3,5,7}{
      \foreach \y in {3,4,5,6}{
        \fill [black!15] (\x,\y) rectangle ++ (1,1);
        \node at (\x+1/2,\y+1/2) {$1$};
      }
    }
    \foreach \y in {2,3,4,5,6}{
      \fill [black!15] (4,\y) rectangle ++ (1,1);
      \node at (4+1/2,\y+1/2) {$1$};
    }
    \foreach \y in {3,4,5}{
      \fill [black!15] (6,\y) rectangle ++ (1,1);
      \node at (6+1/2,\y+1/2) {$1$};
    }
    
    \draw [densely dotted,>->] (-0.5,6.5) -| (6.5,-0.5);
  
    \draw (0,0) grid (8,7);
    \end{scope}
  
  \end{tikzpicture}
  \end{center}
  \caption{The head of $M$ extracts a $1$ from the corner of the clay, creating a movable ``hole''.}
  \label{fig:SculptStage}
  \end{figure}
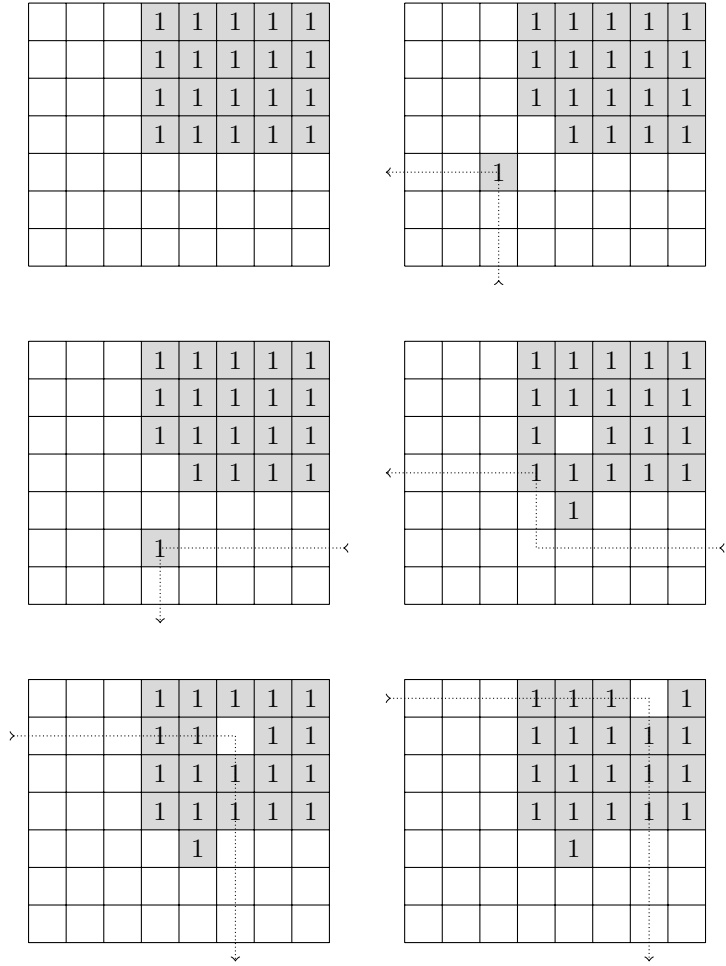

  \begin{figure}[ht]
    \begin{center}
      \begin{tikzpicture}

        \fill [black!15] (0.7,1.2) -- (5.2,1.2) -- (8,4) -- (0.7,4) -- cycle;
        \fill [black!30] (0,0) -- (4,0) -- (5.2,1.2) -- (0.7,1.2) -- (0.7,4) -- (0,4);
        
        \draw (0,0) -- (4,0) -- (8,4) -- (0,4) -- cycle;
        \draw (2,4) -- (2,3) -- (3,3) -- (3,4);

        \fill (2,3) circle (0.05);
        \node [below left] at (2,3) {$(n,n)$};
        
        \draw [dashed] (0,0) -- (3,3);
        \draw [dashed] (4,0) -- (4,4);
        \draw [dashed] (0.7,1.2) -- (2.8,3.3);
        % \draw [dashed] (0.7,4) -- (0.7,1.2) -- (5.2,1.2);

        \node [fill=black!15,above right] at (2.8,3.3) {$(n,n) + \vec{v}_{i,j}$};
        \fill (2.8,3.3) circle (0.05);

        \draw [decorate,decoration={brace,amplitude=5pt}] (2,4.3) -- node [midway,above=0.3cm] {$m$} (3,4.3);
        \draw [decorate,decoration={brace,amplitude=5pt}] (3,4.3) -- node [midway,above=0.3cm] {$m$} (4,4.3);
        \draw [decorate,decoration={brace,amplitude=5pt}] (1.7,3) -- node [midway,left=0.3cm] {$m$} (1.7,4);
        \draw [decorate,decoration={brace,amplitude=5pt}] (-.3,0) -- node [midway,left=0.3cm] {$5m^2$} (-.3,4);
        \draw [decorate,decoration={brace,amplitude=5pt}] (4,-.3) -- node [midway,below=0.3cm] {$5m^2$} (0,-.3);
        
      \end{tikzpicture}
    \end{center}
    \caption{The sculpting process, not drawn to scale. The dark gray part of the clay is removed before the coordinate $\vec{v}_{i,j}$ is considered.}
    \label{fig:SculptTheClay}
  \end{figure}
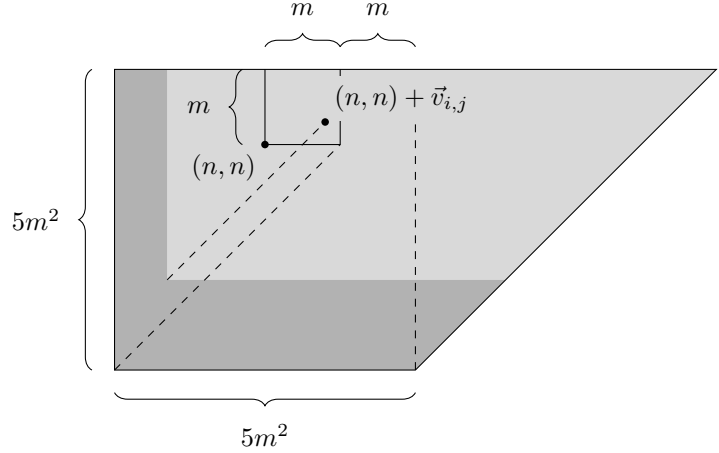

  We describe the sculpt process and its implementation in the normed network $A_f$ in more detail.
  The process is divided into $m^2$ substages, one for each choice of $(i,j) \in [0, m-1]^2$, and substage $i,j$ consists of $N'_{i,j} = O(m^6)$ rounds.
  Substage $i,j$, begins with the removal of some number of southmost rows or westmost columns from the clay.
  As in the probe stage, this is done by moving each $1$ a distance of $O(m^4)$, taking $O(m^4)$ rounds; this is enough as there are a total of $O(m^4)$ cells in the clay.
  Moving each $1$ of a row or column of length $O(m^2)$ for a distance of $O(m^4)$ steps takes $O(m^6)$ rounds.
  If we choose the order of $(i,j) \in [0, m-1]^2$ to be increasing in $j$ when $j+i$ is even and decreasing in $j$ otherwise, it suffices to remove at most two rows or columns from the clay on each substage.

  Next, we check whether $(i,j) \in E_0(R)$.
  In $A_f$, this is achieved by routing the token through the output switch that corresponds to $(i,j)$.
  Since this switch has $O(m^2)$ input terminals and corresponding pairs of output terminals, we can perform this check on each substage $i,j$.
  Based on the result, the token is sent to one of two terminals, so we can send the head of $M$ into $P$ from one of two coordinates.
  If $(i,j) \in E_0(R)$, we send the head from the south, then east and east again, as illustrated in Figure~\ref{fig:SculptStage}, in order to move a $1$ out of the southwest corner of the clay.
  For the next $O(m^2)$ rounds, we send the head from the west border of the clay in order to move this $0$ one step to the northeast within the clay, until it reaches $(i,j)$.
  In the case $(i,j) \notin E_0(R)$, the above process is replaced by the head repeatedly entering $P$ on some row that contains only $0$s, say the northmost row.
  This concludes substage $i,j$.
  The sculpting process is illustrated in Figure~\ref{fig:SculptTheClay}.
  
  Finally, there is an additional substage in which all remaining $1$s outside the square $[n, n+m-1]^2$ are transported out of the region $[0, 2n+m-1] \times [n-3, n+m+2] \cup [n-3, n+m+2] \times [0, 2n+m-1]$ in order to comply with the last condition of the claim.
  This substage takes $O(m^8)$ additional rounds, and is implemented similarly to the probe stage, moving each $1$ diagonally for $O(m^4)$ steps (we can assume that the remaining part of the clay is still surrounded by an annulus of $0$ of thickness $O(m^4)$).
  All in all, the sculpt stage takes $O(m^8)$ rounds and requires $O(m^8 n^4)$ primitive components in $A_f$.
  
  After these five stages, the output region $[n, n+m-1]^2$ contains a translated version of the pattern $g(R)$.
  It remains to guide the token to the correct output terminal $o_k(\vec{b}, t)$, which also contains the number $t = g'(P') = O(C_2)$ determined during the computation stage.
  For each value of $t$ there is a $(1,4n^2)$-switch among the output switches that stores it, and we now use these switches to guide the token to the correct output terminal after $\vec{b}$ has been determined.
  A choice of $n = O(m^4)$ is enough to guarantee that the construction can be carried out, since then we have enough room to transport any $1$s out of the way during the probe and sculpt stages.
  The size of $A_f$ is then $O(m^8 n^4 + (n^4 C_1 + C_2 + n^2)^2) = O(m^{32} + (m^8 C_1 + C_2 + m^6)^2)$, and the depth of $f$ is $O(m^8)$.
  It is clear from the construction that $(\mathcal{P}, f)$ has consistent timing.
\end{proof}

\section{Physical Universality}

We are now ready to prove our main result, the physical universality of $M$.

\begin{theorem}
\label{thm:MainM}
The Turing machine $M$ is efficiently physically universal in the moving head model.
\end{theorem}

\begin{proof}
Let $m \in \N$, let $A \subset \npats{m}{Q}{\Sigma}$ be a set of patterns with disjoint $0$-orbits containing the head of $M$, and let $g : A \to \pats{m}{Q}{\Sigma}$ be a function.
Let $A' \subset A$ be the set of those patterns $R \in A$ for which $g(R)$ contains the head of $M$.
For each $R \in A'$, Lemma~\ref{lem:FastEscape} and its analogue for the inverse machine $M^{-1}$ states that the escape time of $M$ from $R$ when it is surrounded by $0$s, and the escape time of $M^{-1}$ from $g(R)$ when it is surrounded by $1$s, are both $O(m^4)$.
Denote their sum by $\tau'(R)$, and for $R \in A \setminus A'$, define $\tau'(R)$ as just the escape time of $M$ from $R$.
For $R \in A'$, let $b(R) \in B_m$ be the coordinate where $M^{-1}$ escapes from $g(R)$, let $g'(R)$ be the contents of $[-1, m]^2$ as this happens.
For $R \in A \setminus A'$, we (somewhat arbitrarily) define $b(R) = (-1,0)$ and $g'(R) = g(R)$.
Lemma~\ref{lem:CircuitSimulation} lets us compute $\tau'$ and $b$ by a circuit of polynomial size in $m$ and the circuit complexity of $g$.

For $R \in A$, let $p(R) \in \npats{m+2}{Q}{\Sigma}$ be $R$ padded with a thickness-$1$ border of $0$s on all sides, and let $p'(R) \in \pats{m+2}{Q}{\Sigma}$ be $g'(R)$ padded similarly by $1$s.
By Lemma~\ref{lem:ProcessCanDo}, there exists a number $n$, a pattern $P \in \Sigma^{[0, 2n+m+1]^2 \setminus [n, n+m+1]^2}$ and a border process $(\mathcal{P}, f)$ of size $2n+m+2$, polynomial depth and abstract complexity, and consistent timing that transforms a pattern $P \sqcup p(R) \in \mathcal{P}$ for $R \in A$ into another pattern containing $p'(R)$ at $[n, n+m+1]^2$ and the head pointed toward $b(R) + (n+1, n+1)$ on the border of $P$, and also outputs the number $C - \tau'(R)$ for some constant $C \geq \tau'(R)$.
Lemma~\ref{lem:ConcreteRealization} gives us a concrete realization $y \in \Sigma^{\Z^2 \setminus [0, 2n+m+1]^2}$ for $(\mathcal{P}, f)$ where all $1$s are within polynomial distance from the origin.

In this concrete realization, the number of steps $M$ takes between leaving $P$ for the first time and entering it for the final time is $C' + C - \tau'(R)$ for some constant $C'$.
Since the pattern $R$ is surrounded by $0$-cells in $P$ on all four sides, as is $g'(R)$ when the process ends, the number of steps taken by the head between leaving $[n, n+m-1]^2$ for the first time and re-entering for the last time is $C' + C - \tau'(R) + 2n + 2$.
We choose $t = C' + C + 2n + 2$ as the number of time steps to implement $g$.
Note that the padding provided by $p$ and $p'$ does not affect the movement of the head as it enters or leaves the square $[n, n+m-1]^2$.
If $R \in A'$, then we have $M^t(y \sqcup P \sqcup R)|_{[n, n+m-1]} = g(R)$ since the additional $\tau'(R)$ steps account exactly for the head leaving $[n, n+m-1]$ for the first time, and transforming $g'(R)$ into $g(R)$ after re-entering for the last time.
If $R \in A \setminus A'$, then the same holds because when the head is sent into $P$ for the last time, it is directed toward the coordinate $b(R) + (n,n)$ but will not enter $g'(R)$ at time $t$.
Since $C'$, $C$ and $n$ are all polynomial in $m$ and the circuit complexities of $A$ and $g$, the claim follows.
\end{proof}

If a Turing machine is (efficiently) physically universal in the moving head model, then the same holds for the moving tape model.

\begin{corollary}
There exists a two-dimensional reversible Turing machine which is efficiently physically universal in both the moving head model and the moving tape model.
\end{corollary}

A Turing machine with one state that always moves to the right is obviously topologically mixing in the moving tape model.
In the moving head model, it is unknown whether Turing machines can be topologically mixing.
The one-dimensional case is the most interesting one, but our machine provides at least a two-dimensional example, since it is easy to see from the construction that we can freely choose the time parameter $t$ as long as it is large enough (topological transitivity follows directly as a special case of physical universality).

\begin{theorem}
There exists a reversible Turing machine which is mixing of all finite orders in the moving head model.
\end{theorem}

Constructing a one-dimensional physically universal Turing machine seems more difficult.
If one is to mimic our construction, information from the initial pattern needs to be fetched by some type of gadget rather than by simply shooting the head in the correct cell.
We conjecture that such Turing machines exist, but we do not have a candidate.

\section{Modifications}
\label{sec:Further}

We can make slight modifications to our physically universal machine $M$ and obtain machines with different universality properties.

\begin{theorem}
There exists a two-dimensional reversible Turing machine that is physically universal in the moving tape model, but not the moving head model.
\end{theorem}

\begin{proof}[sketch]
Define a Turing machine $M'$ as follows.
The tape alphabet is $\Sigma = \{0, 1\}^2$, and the state set is $Q' = Q \times \{1, 2\}$, where $Q$ is the state set of $M$.
The idea is that $M'$ has two binary tapes, one of which is analogous to that of $M$ and the other is immutable but affects the movement of the head.
The machine behaves as follows:
\begin{enumerate}
\item If the machine is in state $(q, b)$, and there are $1$-symbols on both layers of the tape under the head, and $1$-symbols on the first layer in all eight neighbors, then the state becomes $(q, 3-b)$.
\item After this, if the machine is in state $(q, 1)$, then it takes one step forward, retains its state and does not modify the tapes. Intuitively, in a state $(q,1)$ the machine is partially deactivated.
\item If the machine is in state $(q, 2)$ instead, then it behaves as $M$ would on the first layer of the tape, reading and modifying it, then moving into a new position and assuming a new state $(q', 2)$.
\end{enumerate}
It is immediately clear that $M'$ is not physically universal in the moving head model, since it cannot modify the second layer of the tape.
In dynamical systems terms, the second tape is a nontrivial invariant factor, which prevents even transitivity.
Lemmas~\ref{lem:FastEscape} and~\ref{lem:NoEscape} apply to $M'$ with only minor modifications.

We sketch the proof of physical universality in the moving tape model.
Consider a set of patterns $A \subset \npats{m}{Q'}{\Sigma}$ and a function $g : A \to \pats{m}{Q'}{\Sigma}$.
As with $M$, we can construct a catcher system around the square $[0, m-1]$ that redirects the head after its exit.
In case the head exits in some state $(q,1)$, it cannot be redirected, so we add width-$3$ stripes of $1$-cells on both layers behind the four catcher arrays that change the state to $(q,2)$, and add a second catcher system that is then able to redirect the head.
At the same time we obtain full information about its state.

After catching the head, we repeatedly probe the square $[0, m-1]^2$ to obtain the positions of $1$s on the first layer, as in the probe stage of the proof of Lemma~\ref{lem:ProcessCanDo}.
Since the head will always see a $0$ in some neighboring coordinate, it will not change its state into any $(q, 2)$ during this process, ignoring the second layer and behaving like $M$ instead.
Next, on each coordinate of $[0, m-1]^2$ in turn we transport a $3 \times 3$ pattern of $1$-cells on the first layer and send the head through it to test whether there is a $1$ on the second layer, then move these $1$s out of the way again.

Once we have obtained full information about the input pattern, the head is redirected into one of $|A|$ different patterns, one for each $R \in A$, which evolves into $g(R)$ after a certain number of steps when the head is sent into it through a certain border coordinate.
We can control the number of steps taken by the head of $M'$ so that the total number of steps required for this does not depend on the input pattern.
This construction shows that $M'$ is physically universal in the moving tape model.
\end{proof}

Note that the above proof does not show $M'$ to be efficiently physically universal, since the number of patterns we insert into the gadget may be exponential in $m$.
It remains an open problem whether there exists a Turing machine that is efficiently physically universal in the moving tape model, but not in the moving head model.

Our machine $M$ is allowed to inspect several tape cells simultaneously, which is not the case for a classical Turing machine.
It is possible to simulate general reversible Turing machines by classical Turing machines in the sense of \cite{BaKaSa16}.
We sketch the proof that this can be done in a way that preserves physical universality.

\begin{theorem}
There exists a two-dimensional classical reversible Turing machine that is physically universal.
\end{theorem}

\begin{proof}[sketch]
A classical reversible Turing machine in the sense of \cite{BaKaSa16} first applies a permutation $\pi \in \Sym(Q \times \Sigma)$ to the current state and the current tape symbol, and then performs a \emph{state-dependent shift}, i.e.\ moves as a function of the current state only. Our machine $M$ has the property that its behavior is composed of a joint permutation of the state and a local neighborhood on the tape (a \emph{local permutation} in the terminology of \cite{BaKaSa16}) and then a state-dependent shift.

The basic idea is the following. For any machine that admits such a decomposition, with state set $Q$ and alphabet $\Sigma$, we can construct a reversible classical Turing machine which is an ``$m$th root'' for it, for any large enough $m$. For this, let $Q' = Q \times \{0,1\}^2 \times \{0,1,...,m-1\}$ for large $m$, and $\Sigma' = \Sigma$. On every time step, we increment the final component modulo $m$. This allows us to think of the dynamics as being composed of $m$ distinct steps. On these steps, the head walks around deterministically, performing even permutations on the current tape symbol and the $Q \times \{0,1\}^2$-component of its state. Since we can swap tape bits with the two bits memorized in the state, we can effectively perform arbitrary joint permutations on $Q$ and two symbols of the tape. By Lemma~5 of \cite{BaKaSa16}, this allows us to perform any even permutation of $Q \times \{0,1\}^2 \times \{0,1\}^N$ where $N \subset \Z^2$ is finite. Every permutation that fixes the two bits in the $\{0,1\}^2$-component of the state is even, so this allows us to perfectly simulate the behavior of an arbitrary Turing machine composed of a local permutation followed by a state-dependent shift.

This does not quite give physical universality, as the simulation happens every $m$ steps rather than on every step, leading to a parity (or rather modulo $m$) issue, and also because the state contains two bits whose values are never modified. To combat this for our specific machine $M$, replace $Q'$ with a set $Q'' = \{(q, a, b, c) \;|\; q \in Q, (a,b) \in \{0,1\}^2, c \in N_q \}$, where $N_q = \{0,1,\ldots,m-1\}$ for $q \neq \stE$ and $N_{\stE} = \{0,1,\ldots,m\}$. Modify the rule so that if $c$-component holds $m-1$, then if $q$-component is equal to $\stE$, we increment $c$ to $m$, and if the $c$-component holds $m$, we behave as if it held $m-1$. Furthermore, modify the rule so that whenever the counter $c$ overflows to $0$, we additionally increment $(a, b)$ as a modulo-$4$ counter. After these modifications, there is no local restriction to physical universality (no matter what the value of $m$ is), and the proof of PU works the same as the main proof, up to timing details.
\end{proof}

\section{Additional Information and Open Questions}
\label{sec:AdditionalShit}

In this section, we collect questions from previous sections under one heading, and ask several more questions. We also sketch the proofs of two additional results: we show that physical universality of Turing machines is closed under a natural group of transformations, and show that the maximal number of steps the head of $\bar M$ can stay inside $[0,n-1]^2$ is $\Theta(n^3)$ when restricted to configurations with $O(n)$ many $1$s.

\subsection{Questions asked in previous sections}

\begin{conjecture}
There exists a one-dimensional reversible Turing machine which is physically universal in both the moving head model and the moving tape model, and is topologically mixing.
\end{conjecture}

\begin{question}
Does there exist a Turing machine that is efficiently physically universal in the moving tape model, but not in the moving head model?
\end{question}

If a Turing machine with the following property exists, it is vacuously physically universal in the moving tape model, under our definition.

\begin{question}
Does there exist a Turing machine $M$ such that every configuration with finite support is in the same orbit in the moving tape model?
\end{question}

\subsection{Alternative definitions}

\subsubsection{PU from headless patterns}

A natural alternative to requiring that \emph{all} the patterns $P_i$ in the definition of PU contain the head is to require that none of them do. Say a machine is \emph{physically universal from headless patterns}, if, whenever $P, R$ are $k$-tuples of patterns and the patterns in $P$ contain no heads, then $P$ is physically transformable to $R$ in the moving tape model. Physical universality from headless patterns holds for our machine under this definition as a simple consequence\footnote{This implication holds for all machines that admit a semilinear spacetime diagram in the moving head model, from at least one finite initial configuration. Our machine has this property for \emph{all} finite initial configurations. Langton's ant has it from some initial finite patterns and conjecturally all of them.} of Theorem~\ref{thm:MainM}.

\subsubsection{PU from context-separable patterns}

Another possible alternative definition of physical universality would be that the patterns $P_i \in \npats{m}{Q}{\Sigma}$ do not necessarily have distinct $0$-orbits, but simply admit some context $x \in \pats{\Z^d \setminus [0, m-1]^d}{Q}{\Sigma}$ such that the resulting orbits are disjoint. Say a Turing machine is physically universal from \emph{context-separable patterns} if any tuple $P$ admitting such a context is physically transformable to any tuple $R$ in the moving head model. One can state an analogous definition in the moving tape model.

\begin{question}
Is $M$ physically universal from context-separable patterns, in either model? Is there any Turing machine with this property?
\end{question}

\subsubsection{Combinatorial Entanglement}

Our machine $M$ is symbol-conserving, so one can find a continuous function $P : Q \times \Sigma^{\Z^d} \to \Sym_0(\Z^d)$ where $\Sym_0(\Z^d)$ denotes the discrete group of permutations with finite support, such that in the moving tape model we have $M(q, x) = (p, \tau_{\vec w}(P(q, x)(x)))$ for all $(q, x) \in Q \times \Sigma^{\Z^d}$ (compare to the definition of the moving tape model in Section~\ref{sec:Prelim}). The choice of permutation is not unique, but in our informal description of $M$ and $\bar M$ we essentially specified $P$ already: the permutation used is the one that swaps the $0$ and $1$ visible in \eqref{eq:Transes}.

Fixing this $P$, one can ask for a stronger physical universality allowing ``entanglement'': in addition to initial patterns $P_1,\ldots,P_k$, whose $0$-padded orbits are disjoint and final patterns $R_1,\ldots,R_k$, we may assume a partially defined bijection between positions of $P_j$ and $R_j$ sharing the same symbol, and require that not only $M^t(P_j \sqcup x)_{[0,n-1]^2} = R_j$ but the cocycle $P(M^t)$ respects the specified partial bijection. For example, if the transformation $P_j \mapsto R_j$ being implemented only flips the central bit, it is natural to require that the other bits are literally not moved, i.e. the partially defined bijection fixes them. In our proof of physical universality, these bits would be sculpted from the clay, and the original bits would become garbage. It seems very difficult to undo this damage.

Let us say a machine is \emph{physically universal on combinatorially entangled $0$-finite configurations} if it is physically universal in the sense of the previous paragraph.

\begin{question}
Is $M$ physically universal on combinatorially entangled $0$-finite configurations? Does there exist a Turing machine that is?
\end{question}

\subsubsection{Physical universality in quantum Turing machines}

One may wonder if there are quantum analogues of our results. We are not aware of formal frameworks for TMH and TMT dynamics of quantum Turing machines, so we leave open no precise question. However, given the existence of PU quantum CA proved in \cite{Sc15b}, it is of interest to try to prove quantum analogs of our results.

\begin{question}
  Is there a natural definition of PU for quantum Turing machines?
  If so, is there a PU quantum Turing machine?
\end{question}

\subsection{Turmites}

\subsubsection{Langton's ant}

It seems natural to ask questions about Langton's ant, since it is perhaps the best-known individual Turing machine. To each pattern $P \in \npats{m}{Q}{\Sigma}$, where $|Q| = 4, |\Sigma| = 2$ are the states and tape symbols of Langton's ant, we can associate a \emph{parity} $\pi(P) \in \{0,1,?\}$ by adding modulo $2$ the coordinates of the head and the indicator bit of whether it is traveling vertically, and defining $\pi(P) = {?}$ if a head is not visible.

\begin{question}
Can Langton's ant physically transform a tuple $P \in (\npats{m}{Q}{\Sigma})^k$ to $R \in (\pats{m}{Q}{\Sigma})^k$ whenever the patterns in $P$ are context-separable and all patterns in $P$ (respectively $R$) with a non-$?$ parity have pairwise the same parity?
\end{question}

It seems difficult to resolve this question without first resolving topological transitivity and the question of whether Langton's ant eventually enters one of the standard cycles from all finite initial patterns.

One can state many variants, e.g.\ replace context-separability by disjoint $0$-orbits as we did with $M$. Another, perhaps somewhat more tractable type of PU one could investigate for Langton's ant is the following: Say a machine is \emph{headless-to-headless physically universal} if $(P_0,P_1,...,P_{k-1})$ is physically transformable to $(R_0,R_1,...,R_{k-1})$ whenever none of the $P_i$ or $R_i$ contain a head. This is of course a weakening of PU from headless configurations, so our machine $M$ has this property.

\begin{question}
Is Langton's ant headless-to-headless physically universal in the moving head model?
\end{question}

It is easy to verify experimentally that this is true with pattern size $m = 1$.

\begin{question}
Is Langton's ant headless-to-headless physically universal when restricted to patterns $P, R \in (\pats{m}{Q}{\Sigma} \setminus \npats{m}{Q}{\Sigma})^k$ with $m = 2$?
\end{question}

We also wonder if Langton's ant can physically transform any individual pattern to the all-zero pattern (a very special case of topological transitivity).

\subsubsection{Other turmites}

In \cite{MaGaMeMo18}, it is shown that every turmite in the sense of \cite{De89b} is either Turing universal in a certain sense or $4$-periodic in the moving head model.
Like Langton's ant, turmites that are not $4$-periodic have no periodic points in the moving head model \cite{GaPrSuTr98}. It seems difficult to perform controlled tape manipulation with any of them, or to disprove the possibility of such control.

\begin{question}
Is every aperiodic turmite PU up to parity caveats (see the previous section)? Is any?
\end{question}

\subsection{Invariance and the inverse rule}

\subsubsection{Invariance}

We show that physical universality on finite configurations has good invariance properties. The group $\GL(d,\Z)$ acts on $X_Q \times \Sigma^{\Z^d}$ in the obvious way: $A x_{\vec v} = x_{A^{-1} \vec v}$. In the following statement, for notational convenicence we consider $\Sigma^{\Z^d}$ in a natural way as a subset of $X_Q \times \Sigma^{\Z^d}$. The proof is simply a matter of opening up the definitions, but we go through the motions.

\begin{proposition}
  \label{prop:InvariantUnderMorphisms}
Let $G \leq \Homeo(X_Q \times \Sigma^{\Z^d})$ be the group generated by $\GL(d,\Z)$ and $\Aut(X_Q \times \Sigma^{\Z^d})$. Let $M = (Q, \Sigma, N, \delta)$ be a Turing machine, considered in the moving head model. If $M$ is physically universal on $a$-finite configurations and $g \in G$, then $g \circ M \circ g^{-1}$ is physically universal on $b$-finite configurations where $g(a^{\Z^d}) = b^{\Z^d}$.
\end{proposition}

\begin{proof}
  In this proof, the notation $[P]_{\vec v}$ for a pattern $P \in \pats{m}{Q}{\Sigma}$ means the \emph{cylinder set} $\{ x \in X_Q \times \Sigma^{\Z^d} \;|\; (\tau^{\vec v} x)|_N = P \}$ of those configurations that contain the pattern at position $\vec v$.
  
%here's the def:
%phys uni = the following holds for all $m, k \in \N$:
%Let $P_0, \ldots, P_{k-1} \in \npats{m}{\Sigma}{Q}$ and $R_0, \ldots, R_{k-1} \in \pats{m}{\Sigma}{Q}$ be %patterns in $X_Q \times \Sigma^{\Z^d}$ such that the configurations $P_j \sqcup 0^{\Z^d \setminus [0, m-1]^d}$ have disjoint $M$-orbits.
%Then there exists a partial configuration $x \in \pats{\Z^d \setminus [0, m-1]^d}{\Sigma}{Q}$ and $t \in \N$ such that $M^t(P_j \sqcup x)|_{[0, m-1]^d} = R_j$ for all $j \in [0, k-1]$.
%We suppose $M$ is physically universal on $a$-finites and show $g \circ M \circ g^{-1}$ is physically universal on $b$-finites where $g(a^{\Z^d}) = b^{\Z^d}$.
Suppose $m, k \in \N$, and let $P_0, \ldots, P_{k-1} \in \npats{m}{Q}{\Sigma}$ and $R_0, \ldots, R_{k-1} \in \pats{m}{Q}{\Sigma}$ be patterns in $X_Q \times \Sigma^{\Z^d}$ such that the configurations $P_j \sqcup b^{\Z^d \setminus [0, m-1]^d}$ have disjoint $(g \circ M \circ g^{-1})$-orbits. We need to show that there exists a partial configuration $x \in \pats{\Z^d \setminus [0, m-1]^d}{Q}{\Sigma}$ and $t \in \N$ such that $(g \circ M \circ g^{-1})^t(P_j \sqcup x)|_{[0, m-1]^d} = R_j$ for all $j \in [0, k-1]$.

For this, consider the configurations $y_j = g^{-1}(P_j \sqcup b^{\Z^d \setminus [0, m-1]^d})$. By basic properties of $\Aut(X_Q \times \Sigma^{\Z^d})$, the $y_j$ are configurations of finite $a$-support containing a head, and since the configurations $P_j \sqcup b^{\Z^d \setminus [0, m-1]^d}$ have disjoint $(g \circ M \circ g^{-1})$-orbits, the configurations $y_j$ have disjoint $M$-orbits. Let $P_0', \ldots, P_{k-1}' \in \npats{m'}{Q}{\Sigma}$ for a large even $m'$ be central patterns (i.e. extract the contents of $[-m'/2+1, m'/2]^d$) of the configurations $y_j$ containing the support and the head. 

For $R_0, \ldots, R_{k-1} \in \pats{m}{Q}{\Sigma}$, extend them by $b$s (or anything else), apply $g^{-1}$, and in these configurations apply continuity of $g$ to obtain central patterns $R_0', \ldots, R_{k-1}' \in \pats{m'}{Q}{\Sigma}$ (increasing $m'$ if needed) such that all configurations contained in the cylinder
$[R_i']_{(-m'/2+1,-m'/2+1,...,-m'/2+1)}$
contain $R_0$ in $[0,m-1]^d$.

By physical universality of $M$ (applied through conjugation by translation by $(m'/2-1,m'/2-1,...,m'/2-1)$), there exists a partial configuration $x' \in \pats{\Z^d \setminus [-m'/2+1, m'/2]^d}{Q}{\Sigma}$ and $t \in \N$ such that $M^t(P_j' \sqcup x')|_{[-m'/2+1, m'/2]^d} = R_j'$ for all $j \in [0, k-1]$.

It follows that for each $j$, by conjugating back we get that in $(g \circ M \circ g^{-1})^t( g(P_j' \sqcup x') )$ the contents of $[0,m-1]^d$ is $R_j$. To complete the proof, we need to show that $g(P_j' \sqcup x') = P_j \sqcup x$ for some partial configuration $x \in \pats{\Z^d \setminus [0, m-1]^d}{Q}{\Sigma}$, assuming $m'$ was picked large enough.

Since $P_j'$ was extracted from $g^{-1}(P_j \sqcup b^{\Z^d \setminus [0, m-1]^d})$, just like in the choice of $R_i'$ we have by continuity that $g([P_i']_{(-m'/2+1,-m'/2+1,...,-m'/2+1)})$ contains $P_j$ in $[0,m-1]^d$ if $m'$ was taken large enough. In fact, again by continuity, by picking larger $m'$, we may assume every configuration $g([P_i']_{(-m'/2+1,-m'/2+1,...,-m'/2+1)})$ contains $P_j \sqcup b^{[-\ell, \ell]^d \setminus [0,m-1]^d}$ in $[-\ell, \ell]$, for a large $\ell$.

Observe now that since $g$ has a local rule (up to conjugation by a linear transformation), if $\ell$ is large enough, for any $\vec v \notin [-\ell, \ell]^d$ we have that $g(P_j' \sqcup x')_{\vec v}$ depends only on the values of $P_j' \sqcup x'$ outside the (finite) $a$-support of $y_j$, equivalently outside the non-$a$ support of $P_j'$. Since $P_j'$ contains $a$ outside its non-$a$ support and $x'$ is independent of $j$, we have that indeed there exists some $x$ such that for all $j$, $g(P_j' \sqcup x') = P_j \sqcup x$.
\end{proof}

\begin{corollary}
\label{cor:InversePU}
The machine $M^{-1}$ is efficiently physically universal in the moving head model and the moving tape model with zero symbol $1$.
\end{corollary}

\begin{proof}
  By Lemma~\ref{lem:Inverse} we have $M^{-1} = \sigma' \circ M \circ \sigma'$, and it is easy to verify $\sigma' \in G$ where $G$ is as in Proposition~\ref{prop:InvariantUnderMorphisms}.
\end{proof}

Even if the state sets and alphabets are not necessarily the same, a conjugacy between $X_Q \times \Sigma^{\Z^d}$ and $X_{Q'} \times (\Sigma')^{\Z^d}$ will still conjugate physically universal Turing machines between the subshifts. However, it is easy to show that a conjugacy can only exist if $|Q| = |Q'|$ and $|\Sigma| = |\Sigma'|$, so this is a trivial consequence of the above.

A similar proof works in the moving tape model. One can directly consider the group obtained from $\Aut(X_Q \times \Sigma^{\Z^d})$ by fixing the position of the head, but as it is easy to see that this is just the group generated by reversible cellular automata and reversible Turing machines in the moving tape model in the sense of~\cite{BaKaSa16}, we state the proposition directly in this way. In the following proposition, $\Aut(\Sigma^{\Z^d})$ acts in a natural way on $Q \times \Sigma^{\Z^d}$.

\begin{proposition}
Let $G \leq \Homeo(X_Q \times \Sigma^{\Z^d})$ be the group generated by $G_1 = \GL(d,\Z)$, the group of Turing machines $G_2 = \RTM(Q, \Sigma)$ from~\cite{BaKaSa16}, and $G_3 = \Aut(\Sigma^{\Z^d})$. Let $M = (Q, \Sigma, N, \delta)$ be a Turing machine, considered in the moving tape model. If $M$ is physically universal on $a$-finite configurations and $g \in G_i$, then $g \circ M \circ g^{-1}$ is physically universal on $b$-finite configurations, where
\begin{itemize}
\item $b = a$ if $g \in G_1 \cup G_2$, and
\item $b$ is defined by $g(a^{\Z^d}) = b^{\Z^d}$ if $g \in G_3$.
\end{itemize}
\end{proposition}

\subsubsection{The inverse rule}

The inverse machine $M^{-1}$ is physically universal on $1$-finite configurations by Corollary~\ref{cor:InversePU}. We do not know if it is PU on $0$-finite configurations, equivalently whether $M$ is PU on $1$-finite configurations. Indeed, while Lemma~\ref{lem:FastEscape} applies to the inverse, Lemma~\ref{lem:NoEscape} does not, and indeed on a $0$-finite configuration, $M^{-1}$ does not always eventually escape it by running to infinity along an arithmetic progression, and the spacetime diagram of a finite-support configuration need not be a semilinear set.

\begin{question}
Is $M$ physically universal on $1$-finite configurations? Equivalently, is $M^{-1}$ physically universal on $0$-finite configurations?
\end{question}

\subsection{Escape time}

We show in Lemma~\ref{lem:FastEscape} that the head of the machine escapes a square of side length $n$ in $O(n^4)$ time steps. We do not know if this is optimal.

\begin{question}
\label{q:EscapeTime}
For the machine $M$, do there exist, for infinitely many $n$, configurations with all $1$-symbols and the machine head contained in $[0,n-1]^2$ such that the head stays in $[0,n-1]^2$ for $\Omega(n^4)$ steps?
\end{question}

Obviously in such configurations, $[0,n-1]^2$ must contain a positive density of $1$s since the potential function must have a value of order $\Omega(n^4)$.

The potential function gives the correct polynomial rate of escape for configurations where there are $O(n)$ $1$-symbols in $[0,n-1]^2$. Namely, the potential of such a configuration is obviously $O(n^3)$, so the head will escape in time $O(n^3)$, and indeed we can find configurations where the escape time is $\Omega(n^3)$ and the support consists of $7$ partial arithmetic progressions.

\begin{example}
\label{ex:N3}
Let $a \geq 1, b \geq 0$. Consider the configuration $(\stE[(1,0)], x)$ where the coordinates of $1$-symbols in $x \in \{0,1\}^{\Z^2}$ form the set
\[ A = A_{nw} \cup A_{n} \cup A_{ne} \cup A_{se} \cup A_{se2} \cup A_{s} \cup A_{sw} \] where
\[ A_{nw} = \{ (-3,1) + i(-1,1) \;|\; i \in [0,a-1] \} \]
\[ A_{n} = \{ (0,-1) + i(0,1) \;|\; i \in [0,a] \} \]
\[ A_{ne} = \{ (2,0) + i(3,1) + b(2,0) \;|\; i \in [0,a-1] \} \]
\[ A_{se} = \{ (2,-2) + i(3,-1) + b(2,-2) \;|\; i \in [0,a-1] \} \]
\[ A_{se2} = \{ (2,-4) + i(1,-1) + b(1,-2) + a(0,-1) \;|\; i \in [0,a-1] \} \]
\[ A_{s} = \{ (0,-2) + i(0,-1) + b(0,-2) \;|\; i \in [0,a-1] \} \]
\[ A_{sw} = \{ (-3,-4) + i(-1,-1) + b(0,-2) + a(0,-1) \;|\; i \in [0,a-1] \} \]

We only give a ``proof-by-simulation'', and give an informal description of what happens.\footnote{It should be a straightforward (but long) exercise to translate this into an explicit inductive proof by simply writing the formulas for the (finitely many types of) intermediate situations. We feel that it is not worth the trouble to torture ourselves and the readers with this, as this construction is only a partial result towards Question~\ref{q:EscapeTime}.} To better follow the description below, we suggest running a computer simulation. Appendix~\ref{sec:Golly} contains a description of $\bar M$ in the @RULE format used by the Golly cellular automaton simulator, and Appendix~\ref{sec:Golly2} contains a Lua script for drawing the configuration $(\stE[(1,0)], x)$ in Golly.

Inspecting the trajectories of these configurations for various $a$ and $b$, we see that the head ``rebuilds a blown-up K-shape'', in the sense that the configuration contains four separate segments of $1$-symbols, $A_n$, $A_{ne}$, $A_{se}$ and $A_s$, which the head slowly transports together to form a shape that resembles the letter K.
This happens in $a$ stages, where on the $i$th stage for $i \in [1, a]$, it transports one $1$-symbol to each of the four tips of the K-shape.
This process is similar to the activation of catchers in Section~\ref{sec:BorderProcesses}: the head travels in a rectangular spiral, bringing four $1$-symbols one step closer to the spiral's center on each loop.
When the $1$-symbols reach the K-shape, the head travels south, hits $A_{se2}$, moves clockwise around the blown-up-K and begins the next stage.

Now we analyze the amount of time spent on this configuration inside the smallest rectangle $R$ containing $A$.
On the first stage, the machine visits all but five cells\footnote{It misses the northwest corner $(0,0)$ by the choice of the initial position of the head, and also misses $4$ cells in the center.} of the rectangle $[0, 2b+2] \times [-2b-2, 0]$. For $b$ large enough, this is at least $b^2$ cells. On each stage, the annulus visited grows strictly, and there are $a$ stages, so in total the machine takes at least $a b^2$ steps before exiting $R$. Setting $a = b = n$, the region $R$ is contained in a square of side length
$\max(1 + 4a + 2b, a + 2 + 2a + 2b) \leq 7n$ and the machine takes at least $n^3$ steps on it. This concludes the proof. \hfill \qee
\end{example}

\bibliographystyle{plain}
\bibliography{../../../bib/bib}{}

\def\ocirc#1{\ifmmode\setbox0=\hbox{$#1$}\dimen0=\ht0 \advance\dimen0
  by1pt\rlap{\hbox to\wd0{\hss\raise\dimen0
  \hbox{\hskip.2em$\scriptscriptstyle\circ$}\hss}}#1\else {\accent"17 #1}\fi}
\begin{thebibliography}{10}

\bibitem{BaKaSa16}
Sebastián Barbieri, Jarkko Kari, and Ville Salo.
\newblock The {Group} of {Reversible} {Turing} {Machines}.
\newblock In Matthew Cook and Turlough Neary, editors, {\em Cellular {Automata}
  and {Discrete} {Complex} {Systems}}, Lecture {Notes} in {Computer} {Science},
  pages 49--62. Springer International Publishing, 2016.

\bibitem{BlCaNi02}
Vincent~D. Blondel, Julien Cassaigne, and Codrin Nichitiu.
\newblock On the presence of periodic configurations in {Turing} machines and
  in counter machines.
\newblock {\em Theoretical Computer Science}, 289(1):573--590, October 2002.

\bibitem{BuTr92}
L.~A. Bunimovich and S.~E. Troubetzkoy.
\newblock Recurrence properties of {L}orentz lattice gas cellular automata.
\newblock {\em J. Statist. Phys.}, 67(1-2):289--302, 1992.

\bibitem{BuTr93}
L.~A. Bunimovich and S.~E. Troubetzkoy.
\newblock Topological dynamics of flipping {L}orentz lattice gas models.
\newblock {\em J. Statist. Phys.}, 72(1-2):297--307, 1993.

\bibitem{DeBl04}
Jean-Charles Delvenne and Vincent~D. Blondel.
\newblock Quasi-periodic configurations and undecidable dynamics for tilings,
  infinite words and {Turing} machines.
\newblock {\em Theoretical Computer Science}, 319(1):127--143, June 2004.

\bibitem{De89b}
AK~Dewdney.
\newblock Two-dimensional {T}uring machines and {T}urmites make tracks on a
  plane.
\newblock {\em Scientific American}, 261:180--183, 1989.

\bibitem{Ga03}
Anahi Gajardo.
\newblock {A symbolic projection of Langton's Ant}.
\newblock {\em {Discrete Mathematics \& Theoretical Computer Science}}, {DMTCS
  Proceedings vol. AB, Discrete Models for Complex Systems (DMCS'03)}, January
  2003.

\bibitem{GaPrSuTr98}
David Gale, Jim Propp, Scott Sutherland, and Serge Troubetzkoy.
\newblock Further travels with my ant.
\newblock In David Gale, editor, {\em Tracking the Automatic ANT: And Other
  Mathematical Explorations}, pages 137--149. Springer New York, New York, NY,
  1998.

\bibitem{Ja10}
Dominik {Janzing}.
\newblock {Is there a physically universal cellular automaton or Hamiltonian?}
\newblock {\em arXiv e-prints}, Sep 2010.

\bibitem{Je13}
Emmanuel Jeandel.
\newblock Computability of the entropy of one-tape turing machines.
\newblock {\em CoRR}, abs/1302.1170, 2013.

\bibitem{Ku97}
Petr K{\r{u}}rka.
\newblock On topological dynamics of {T}uring machines.
\newblock {\em Theoret. Comput. Sci.}, 174(1-2):203--216, 1997.

\bibitem{La86}
Christopher~G. Langton.
\newblock Studying artificial life with cellular automata.
\newblock {\em Phys. D}, 22(1-3):120--149, 1986.
\newblock Evolution, games and learning (Los Alamos, N.M., 1985).

\bibitem{MaGaMeMo18}
D.~Maldonado, A.~Gajardo, B.H. De~Menibus, and A.~Moreira.
\newblock Nontrivial turmites are {T}uring-universal.
\newblock {\em Journal of Cellular Automata}, 13(5-6):373--392, 2018.

\bibitem{Mo90}
Cristopher Moore.
\newblock Unpredictability and undecidability in dynamical systems.
\newblock {\em Physical Review Letters}, 64(20):2354--2357, May 1990.

\bibitem{Pr83}
Lutz Priese.
\newblock Automata and concurrency.
\newblock {\em Theoretical Computer Science}, 25(3):221 -- 265, 1983.

\bibitem{SaTo17b}
Ville Salo and Ilkka T\"{o}rm\"{a}.
\newblock A one-dimensional physically universal cellular automaton.
\newblock In {\em Unveiling dynamics and complexity}, volume 10307 of {\em
  Lecture Notes in Comput. Sci.}, pages 375--386. Springer, Cham, 2017.

\bibitem{Sc15a}
Luke Schaeffer.
\newblock A physically universal cellular automaton.
\newblock In {\em I{TCS}'15---{P}roceedings of the 6th {I}nnovations in
  {T}heoretical {C}omputer {S}cience}, pages 237--246. ACM, New York, 2015.

\bibitem{Sc15b}
Luke Schaeffer.
\newblock A physically universal quantum cellular automaton.
\newblock In {\em Cellular automata and discrete complex systems}, volume 9099
  of {\em Lecture Notes in Comput. Sci.}, pages 46--58. Springer, Heidelberg,
  2015.

\end{thebibliography}

\clearpage

\appendix

\section{Golly compatible @RULE file for (a ``square root'' of) $\bar M$}
\label{sec:Golly}

\begin{Verbatim}
@RULE PUAnt

@TABLE
n_states:18
# 0 and 1 represent tape bits
# 2,4,6,8 represent tape bit 0, head moving E,N,W,S, and the next step is the local involution
# 3,5,7,9 represent tape bit 1, head moving E,N,W,S, and the next step is the local involution
# 10,12,14,16 represent tape bit 0, head moving E,N,W,S, and the next step is movement
# 11,13,15,17 represent tape bit 1, head moving E,N,W,S, and the next step is movement
neighborhood:Moore
symmetries:none

var a={0,1}
var b={0,1}
var c={0,1}
var d={0,1}
var e={0,1}
var f={0,1}
var g={0,1}
var h={0,1}

# moving head with bit z=0 or o=1, ordering
var mz={10,12,14,16}
var mo={11,13,15,17}

# the local involution
2,a,b,c,1,c,d,e,f,13
1,a,b,c,d,e,f,a,2,0
2,a,b,c,d,e,f,g,h,10
3,a,b,c,d,e,0,e,f,16
0,a,3,a,b,c,d,e,f,1
3,a,b,c,d,e,f,g,h,11
4,c,1,c,d,e,f,a,b,15
1,c,d,e,f,a,4,a,b,0
4,c,d,e,f,g,h,a,b,12
5,c,d,e,0,e,f,a,b,10
0,a,b,c,d,e,f,a,5,1
5,c,d,e,f,g,h,a,b,13
6,c,d,e,f,a,b,c,1,17
1,e,f,a,6,a,b,c,d,0
6,e,f,g,h,a,b,c,d,14
7,e,0,e,f,a,b,c,d,12
0,c,d,e,f,a,7,a,b,1
7,e,f,g,h,a,b,c,d,15
8,e,f,a,b,c,1,c,d,11
1,a,8,a,b,c,d,e,f,0
8,g,h,a,b,c,d,e,f,16
9,e,f,a,b,c,d,e,0,14
0,e,f,a,9,a,b,c,d,1
9,g,h,a,b,c,d,e,f,17

# moving heads that leave the origin
mz,a,b,c,d,e,f,g,h,0
mo,a,b,c,d,e,f,g,h,1

# moving heads that enter the origin
0,a,b,c,d,e,f,10,g,2
1,a,b,c,d,e,f,10,g,3
0,a,b,c,d,e,f,11,g,2
1,a,b,c,d,e,f,11,g,3
0,a,b,c,d,12,e,f,g,4
1,a,b,c,d,12,e,f,g,5
0,a,b,c,d,13,e,f,g,4
1,a,b,c,d,13,e,f,g,5
0,a,b,14,c,d,e,f,g,6
1,a,b,14,c,d,e,f,g,7
0,a,b,15,c,d,e,f,g,6
1,a,b,15,c,d,e,f,g,7
0,16,a,b,c,d,e,f,g,8
1,16,a,b,c,d,e,f,g,9
0,17,a,b,c,d,e,f,g,8
1,17,a,b,c,d,e,f,g,9

@ICONS

XPM

/* width height num_colors chars_per_pixel */
"7 119 2 1"
/* colors */
"A c #FFFFFF"
". c #000000"
/* icon for state 1 */
"AAAAAAA"
"AAAAAAA"
"AAAAAAA"
"AAAAAAA"
"AAAAAAA"
"AAAAAAA"
"AAAAAAA"
/* icon for state 2 */
"......."
".AA...."
"...AA.."
".....AA"
"...AA.."
".AA...."
"......."
/* icon for state 3 */
"AAAAAAA"
"A..AAAA"
"AAA..AA"
"AAAAA.."
"AAA..AA"
"A..AAAA"
"AAAAAAA"
/* icon for state 4 */
"...A..."
"...A..."
"..A.A.."
"..A.A.."
".A...A."
".A...A."
"......."
/* icon for state 5 */
"AAA.AAA"
"AAA.AAA"
"AA.A.AA"
"AA.A.AA"
"A.AAA.A"
"A.AAA.A"
"AAAAAAA"
/* icon for state 6 */
"......."
"....AA."
"..AA..."
"AA....."
"..AA..."
"....AA."
"......."
/* icon for state 7 */
"AAAAAAA"
"AAAA..A"
"AA..AAA"
"..AAAAA"
"AA..AAA"
"AAAA..A"
"AAAAAAA"
/* icon for state 8 */
"......."
".A...A."
".A...A."
"..A.A.."
"..A.A.."
"...A..."
"...A..."
/* icon for state 9 */
"AAAAAAA"
"A.AAA.A"
"A.AAA.A"
"AA.A.AA"
"AA.A.AA"
"AAA.AAA"
"AAA.AAA"
/* icon for state 10 */
"......."
".AA...."
"...AA.."
"...AAAA"
"...AA.."
".AA...."
"......."
/* icon for state 11 */
"AAAAAAA"
"A..AAAA"
"AAA..AA"
"AAA...."
"AAA..AA"
"A..AAAA"
"AAAAAAA"
/* icon for state 12 */
"...A..."
"...A..."
"..AAA.."
"..AAA.."
".A...A."
".A...A."
"......."
/* icon for state 13 */
"AAA.AAA"
"AAA.AAA"
"AA...AA"
"AA...AA"
"A.AAA.A"
"A.AAA.A"
"AAAAAAA"
/* icon for state 14 */
"......."
"....AA."
"..AA..."
"AAAA..."
"..AA..."
"....AA."
"......."
/* icon for state 15 */
"AAAAAAA"
"AAAA..A"
"AA..AAA"
"....AAA"
"AA..AAA"
"AAAA..A"
"AAAAAAA"
/* icon for state 16 */
"......."
".A...A."
".A...A."
"..AAA.."
"..AAA.."
"...A..."
"...A..."
/* icon for state 17 */
"AAAAAAA"
"A.AAA.A"
"A.AAA.A"
"AA...AA"
"AA...AA"
"AAA.AAA"
"AAA.AAA"
\end{Verbatim}

\section{Lua script implementing Example~\ref{ex:N3} in Golly}
\label{sec:Golly2}

\begin{Verbatim}

local g = golly()
local gp = require "gplus"

--[[ Set rule to PUAnt2 --]]
g.setrule("PUAnt2")

--[[ Ask user for a and b --]]
local param_a = tonumber(g.getstring("Give value for parameter a", "5"))
local param_b = tonumber(g.getstring("Give value for parameter b", "3"))

--[[ Get current position --]]
local cur_x, cur_y = gp.getposint()

--[[ Set cell values --]]

function setcell_inv (a, b, v)
   --[[ Flip y-axis to comply with Example --]]
   g.setcell(cur_x+a, cur_y-b, v)
end

--[[ Head --]]
setcell_inv(1,0,2)

for i=0,param_a do
   --[[ A_nw --]]
   setcell_inv(-3-i, 1+i, 1)
   --[[ A_n --]]
   setcell_inv(0, i-1, 1)
   --[[ A_ne --]]
   setcell_inv(2+3*i+2*param_b, i, 1)
   --[[ A_se --]]
   setcell_inv(2+3*i+2*param_b, -2-i-2*param_b, 1)
   --[[ A_se2 --]]
   setcell_inv(2+i+param_b, -4-i-2*param_b-param_a, 1)
   --[[ A_s --]]
   setcell_inv(0, -2-i-2*param_b, 1)
   --[[ A_sw --]]
   setcell_inv(-3-i, -4-i-2*param_b-param_a, 1)
end

--[[ Part of A_n --]]
setcell_inv(0,-1+param_a, 1)
\end{Verbatim}

\end{document}